\documentclass[11pt,reqno]{amsart}
\usepackage[utf8]{inputenc}
\usepackage{xcolor}
\usepackage[T1]{fontenc}                        
\usepackage{geometry}

\usepackage{amsfonts}
\usepackage{comment}                              
\usepackage{mparhack} 
\usepackage{amssymb}                            
\usepackage{enumerate}
\usepackage{enumitem}
\usepackage{booktabs}                             
\usepackage{graphicx,subfig}                    
\usepackage{wrapfig}                            
\usepackage[bookmarks=true,colorlinks=true]{hyperref}\hypersetup{urlcolor=blue,citecolor=red,linkcolor=black}
\usepackage{bbm}

\usepackage{bm}                                   
\usepackage{dsfont}

\newtheorem{theorem}{Theorem}[section]

\newtheorem{proposition}[theorem]{Proposition}
\newtheorem{lemma}[theorem]{Lemma}
\newtheorem{corollary}[theorem]{Corollary}
\theoremstyle{definition}
\newtheorem{example}[theorem]{Example}
\newtheorem{remark}[theorem]{Remark}

\newtheorem{hypothesis}[theorem]{Hypothesis}

\let\originalleft\left
\let\originalright\right
\renewcommand{\left}{\mathopen{}\mathclose\bgroup\originalleft}
\renewcommand{\right}{\aftergroup\egroup\originalright}

\DeclareMathOperator*{\argmin}{argmin}
\newcommand{\me}{\mathcal{E}}
\newcommand{\md}{\mathcal{D}}
\newcommand{\mf}{\mathcal{F}}

\newcommand{\mh}{\mathcal{H}}
\newcommand{\ml}{\mathcal{L}}
\newcommand{\mk}{\mathcal{K}}

\newcommand{\mn}{\mathcal{N}}
\newcommand{\mi}{\mathcal{I}}

\newcommand{\mr}{\mathcal{R}}
\newcommand{\ms}{\mathcal{S}}

\newcommand{\eepstau}{\me_{\eps, \tau}}
\newcommand{\eepstaurhat}[1]{\me_{\eps, \tau}(#1|\hat\rrho)}

\newcommand{\sspc}{{\mathfrak X}_0}

\newcommand{\dst}{\mathrm{d}}
\newcommand{\wass}{{\mathbf W}}
\newcommand{\dd}{\,\mathrm{d}}
\newcommand{\dn}{\mathrm{d}}

\newcommand{\mptrd}{{\mathcal{P}_2}(\Rd)}
\newcommand{\mom}{{\mathfrak m}}

\newcommand{\R}{{\mathbb R}}

\newcommand{\Q}{{\mathbb Q}}
\newcommand{\Rp}{{\mathbb R}_{>0}}
\newcommand{\Rnn}{{\mathbb R}_{\ge0}}
\newcommand{\Rd}{{\mathbb{R}^{d}}}

\newcommand{\dff}{{\mathrm D}}
\newcommand{\supp}{\mathrm{supp}}

\newcommand{\dv}{\mathrm{div}}
\newcommand{\eps}{\varepsilon}
\newcommand{\intrd}{\int_{\mathbb{R}^d}}
\newcommand{\intd}{\mathrm{d}}
\newcommand{\iintrdrd}{\iint_{\mathbb{R}^d \times \mathbb{R}^d}}
\newcommand{\ddszerop}{\left. \frac{\dd}{\dd s} \right|_{s = 0^+}}

\newcommand{\weakto}{\rightharpoonup}

\newcommand{\rhoone}{\rho_{1}}
\newcommand{\rhotwo}{\rho_{2}}

\newcommand{\rrho}{\bm{\rho}}
\newcommand{\rr}{{\mathbf r}}

\newcommand{\eeta}{\bm{\eta}}

\newcommand{\uu}{\bm{u}}

\newcommand{\vv}{\bm{v}}

\title[Convergence to equilibrium with nonlocal interaction]{Convergence to equilibrium for cross diffusion systems with nonlocal interaction}
\author{Daniel Matthes} 
\author{Christian Parsch}
\thanks{This research was supported by the DFG Collaborative Research Center TRR 109, ``Discretization in Geometry and Dynamics''.}

\date{\today}

\begin{document}


\address{D. Matthes, C. Parsch --  Zentrum Mathematik/M8, Technische Universit\"{a}t M\"{u}nchen, Boltzmannstrasse 3, D-80538 Garching, Germany}
\email{matthes@ma.tum.de}
\email{christian.parsch@tum.de}

\begin{abstract}
    We study the existence and the rate of equilibration of weak solutions to a two-component system of non-linear diffusion-aggregation equations, with small cross diffusion effects. The aggregation term is assumed to be purely attractive, and in the absence of cross diffusion, the flow is exponentially contractive towards a compactly supported steady state. Our main result is that for small cross diffusion, the system still converges, at a slightly lower rate, to a deformed but still compactly supported steady state.

    Our approach relies on the interpretation of the PDE system as a gradient flow in a two-component Wasserstein metric. The energy consists of a uniformly convex part responsible for self-diffusion and non-local aggregation, and a totally non-convex part that generates cross diffusion; the latter is scaled by a coupling parameter $\eps>0$. The core idea of the proof is to perform an $\eps$-dependent modification of the convex/non-convex splitting and establish a control on the non-convex terms by the convex ones.
    %
\end{abstract}

\subjclass[2020]{}

\maketitle

\section{Introduction}

In this paper, we consider the following system of diffusion-aggregation equations with (small) cross diffusion on $\R^d$:
\begin{equation}\label{eq:system-intro}
    \begin{cases}
        \partial_t\rhoone=\dv\left(\rhoone\nabla[F_1'(\rhoone)+\varepsilon\partial_{r_1}h(\rhoone,\rhotwo)+K*\rhotwo]\right),\\
        \partial_t\rhotwo=\dv\left(\rhotwo\nabla[F_2'(\rhotwo)+\varepsilon\partial_{r_2}h(\rhoone,\rhotwo)+K*\rhoone]\right).
    \end{cases}
\end{equation}
Above, the pair $\rrho=(\rho_1,\rho_2)$ of time-dependent probability densities is sought, $F_1,\,F_2$ are given non-linearities for the (self-)diffusions, $h$ induces cross diffusion of strength $\eps\ge0$, and $K:\R^d\to\R$ is an attractive interaction kernel. We shall assume that $F_1$ and $F_2$ are convex and degenerate at zero, $F_1'(0)=F_2'(0)=0$, that the derivatives of $h$ can be controlled by those of $F_1$ and $F_2$ in a way precised further below, and that $K$'s convexity is uniform, $\nabla^2 K\ge \lambda\,\mathbf{1}$ with some $\lambda>0$. The system is augmented with initial conditions,
\begin{align}
    \label{eq:ic-intro}
    \rho_1(0,\cdot) = \rho_1^0,\quad \rho_2(0,\cdot) = \rho_2^0,
\end{align}
which we assume to have finite energy, see below. The PDE system \eqref{eq:system-intro} falls into a general class of macroscopic models for multi-species many-particle systems in which long-range attraction forces compete with local diffusion --- the former tend to accumulate particles, the latter prevent overcrowding. Such competition is fundamental in various models for interacting particles, including chemotaxis \cite{BCC,Luckhaus}, granular media \cite{Gentil,CMV}, social dynamics \cite{gangs,Schlake}, semi-conductors \cite{KMX} and various further applications. Here, we shall not focus on specific choices of the parameters $F_1$, $F_2$, $K$ and $h$ related to any of these models, but consider \eqref{eq:system-intro} under quite general hypotheses.

\subsection{Motivation: persistence of equilibration}
Our goal is to show at the example of the system \eqref{eq:system-intro} the persistence of the equilibration behaviour in contractive aggregation-(self-)diffusion flows under the influence of cross diffusion. Specifically, we perform a rigorous analysis of existence and long-time asymptotics of weak solutions for sufficiently small $\eps>0$, and prove that the solutions are global in time and converge to the unique equilibrium at almost the same exponential rate as in the unperturbed situation $\eps=0$. This result is beyond perturbation theory not only because of its global character, but also because the perturbation is singular as we shall explain next.
 
The point of departure for our analysis is the gradient flow formulation of \eqref{eq:system-intro}: the respective energy functional is
\begin{equation}
    \label{eq:Eeps_def}
    \me_\eps(\rrho):= \intrd [F_1(\rhoone)+F_2(\rhotwo)+\eps h(\rhoone,\rhotwo) + \rhoone K*\rhotwo ]\,\intd x,
\end{equation}
the metric is a combined $L^2$-Wasserstein distance for $\rho_1$ and $\rho_2$. In the absence of cross diffusion, i.e., for $\eps=0$, that functional is uniformly geodesically convex, and --- thanks to the general theory of metric gradient flows \cite{AGS} --- the generated gradient flow \eqref{eq:system-intro} is uniformly exponentially contractive. In particular, there is a unique stationary solution $\bar{\rrho}$ to \eqref{eq:system-intro}, and that is approached at an exponential rate $\lambda>0$, see e.g. \cite[Chapter 11]{AGS}. Now, the perturbation induced by cross diffusion is singular in the sense that for any $\eps>0$, the energy $\me_\eps$ is not even semi-convex anymore, and consequently, the gradient flow \eqref{eq:system-intro} is not contractive. That devastating effect of cross diffusion has been proven in a related situation in \cite[Prop. 2.4]{BMZ}.

Despite the failure of the abstract machinery, we shall prove that for sufficiently small $\eps>0$, the equilibration behaviour persists in the sense that there is a unique stationary solution $\bar{\rrho}_\eps$ to \eqref{eq:system-intro}, which is a deformation of $\bar{\rrho}$ above, and that arbitrary solutions converge to $\bar{\rrho}_\eps$ at an exponential rate $\lambda_\eps\ge\lambda-C_0\eps$. Our result thus contributes to a line of studies on the exponential equilibration in non-convex metric gradient flows, see \cite{Alasio,BMZ,LM,Markowich,MMS,Zinsl-KS,Zinsl-KS1d,Zinsl-NP}. Differently from these previous works, the confining force leading to equilibration in the unperturbed system \eqref{eq:system-intro} is not induced by external potentials, but is self-induced via the non-local interaction of the species. This is very natural from a modeling point of view, but induces additional difficulties both on the technical and on the conceptual level. Generally speaking, the equilibrating force of non-local interactions is more fragile than that of fixed external potentials, and requires a more refined analysis. Most prominently, while the integral $\intrd \rho_1K\ast\rho_2\dd x$ is a uniformly geodesically convex functional of the pair $\rrho=(\rho_1,\rho_2)$ in the considered metric \cite{Zinsl}, this functional is \emph{not} convex in the sense of affine interpolation, which leads to non-convexity of the entire energy $\me_\eps$. Thus we are, for instance, not able to exclude the emergence of additional critical points in the perturbed energy $\me_\eps$ a priori by standard method from the calculus of variations, but need to obtain the uniqueness of the stationary solution $\bar{\rrho}_\eps$ a posteriori from the global equilibration estimates.

There has been a great recent interest in understanding the long-time behaviour of multi-species cross-diffusion systems, and particularly those with non-local effects. The focus has been on proving qualitative properties of the steady states, in particular segregation phenomena, see \cite{Segregation,Sorting,Numerics,Zoology} and references therein for analytical and numerical studies. It appears, however, that there are no rigorous results on the global exponential stability of these stationary solutions, not even in the non-segregated case that we study here.


\subsection{Main result}
Below, we summarize the main results of this paper in Theorem \ref{thm:main}. For a concise statement, introduce the space $\mptrd$ of probability measures on $\Rd$ with finite second moment, and the following subspace $\sspc\subset\mptrd^2$ of pairs of measures whose combined center of mass is at the origin:
\begin{align}
    \label{eq:sspc_def}
    \sspc := \left\{\rrho= (\rho_1, \rho_2) \in \mptrd^2: \intrd x\,\left(\rho_1 + \rho_2\right)\,\intd x = 0\right\}.
\end{align}
We prove:
\begin{theorem} \label{thm:main}
    Assuming that $F_1,F_2,h$ and $K$ satisfy Hypotheses \ref{hyp:general} and \ref{hyp:theta} further below, there exists some $\bar\eps > 0$ such that for every $\eps \in (0, \bar\eps]$, the following is true.
    \begin{itemize}
        \item The energy functional $\me_\eps$ possesses a unique $\eps$-dependent global minimizer $\bar\rrho=(\bar\rho_1,\bar\rho_2)$ in $\sspc$. The densities $\bar\rho_1$ and $\bar\rho_2$ are continuous in $\Rd$, radially symmetric about the origin, and compactly supported. 
        They satisfy, with suitable constants $C_1, C_2 > 0$, the Euler-Lagrange system
        \begin{equation} \label{eq:euler-lagrange_intro}
            \begin{split}
                F_1'(\bar\rho_1) + \varepsilon \partial_{r_1} h(\bar\rho_1, \bar\rho_2) &= \left(C_1 - K * \bar\rho_2\right)_+ \\
                F_2'(\bar\rho_2) + \varepsilon \partial_{r_2} h(\bar\rho_1, \bar\rho_2) &= \left(C_2 - K * \bar\rho_1\right)_+\, .
            \end{split}
        \end{equation}
        \item Given initial data $\rrho^0 \in \sspc$ of finite energy, i.e. $\me_\eps(\rrho^0) < +\infty$, there exists a weakly continuous curve $\rrho: [0, +\infty) \to \sspc$ 
        with $\rrho(0) = \rrho^0$ that satisfies the system  \eqref{eq:system-intro} in the sense that for every compactly supported test function $\zeta\in C^\infty_c(\Rp\times\Rd)$, 
        \begin{equation*}
            \begin{split}
            \int_0^\infty\intrd \left[\rho_1\, \partial_t \zeta - \rho_1\,\nabla\left[F_1'(\rho_1) + \eps \partial_{r_1}h(\rho_1, \rho_2) + K \ast \rho_2\right]\cdot\nabla \zeta \right]\,\intd x\intd t 
             &= 0, \\
            \int_0^\infty\intrd \left[  \rho_2\, \partial_t \zeta - \rho_2\,\nabla\left[F_2'(\rho_2) + \eps \partial_{r_2}h(\rho_1, \rho_2) + K \ast \rho_1\right]\cdot\nabla \zeta \right]\,\intd x\intd t 
            &= 0.
            \end{split}
        \end{equation*}
        \item There is a constant $C_0 > 0$, independent of $\eps \in (0, \bar\eps]$, such that the energy decays exponentially fast along this solution $\rrho$ with rate $\lambda - C_0 \eps$:
        \begin{align} \label{eq:expdecay_intro}
            \me_\eps(\rrho(t)) - \me_\eps(\bar\rrho) \leq C_1 \,\left(\me_\eps(\rrho^0) - \me_\eps(\bar\rrho)\right)\,e^{-2(\lambda - C_0 \eps) t}
        \end{align}
        at every $t \geq 0$ for some constant $C_1$ independent of $\eps$ and $\rrho^0$. Further, $\rrho$ converges to the unique global minimizer $\bar\rrho$ exponentially in $L^1$ with the same rate. 
    \end{itemize}
\end{theorem}


\subsection{Strategy of proof}
We give cartoon of our approach to proving the persistence of long-time asymptotics in ODE language. Assume that two potentials $E_0,\,H:\R\to\R$ are given, where $E_0\in C^2(\R)$ is uniformly convex of modulus $\lambda>0$, i.e., $E_0''\ge\lambda$, and $H\in C^1(\R)$ is bounded below but \emph{not} convex. For $\eps\ge0$, define the energy $E_\eps := E_0+\eps H$, which is in general not convex for any $\eps>0$. We wish to study the long-time asymptotics of the associated gradient flow, 
\begin{align}
    \label{eq:ode}
    \dot x = -E_\eps'(x) = -E_0'(x) - \eps H'(x). 
\end{align}
Thanks to the lower bound on $H$, there exists a global minimizer $\bar x_\eps$ of $E_\eps$ for every $\eps\ge0$; a priori, $\bar x_\eps$ need not be unique. Now introduce the two auxiliary functions
\begin{align*}
    L_\eps(x) = E_0(x) + \eps (x-\bar x_\eps)H'(\bar x_\eps),
    \quad
    N_\eps(x) = H(x) - (x-\bar x_\eps)H'(\bar x_\eps),
\end{align*}
and note that $\bar x_\eps$ is the global minimum of the $\lambda$-convex function $L_\eps$, and is a critical point of the non-convex function $N_\eps$. Finally, we add our core hypothesis on the relation between $H$ and $E_0$: for $0<\eps<1$, the following inequality holds for all $x\in\R$:
\begin{align}
    \label{eq:artificial}
    |N_\eps'(x)|\le|L_\eps'(x)| .
\end{align}
With that at hand, $L_\eps$ becomes a Lyapunov functional: for any solution $x$ to \eqref{eq:ode},
\begin{align*}
    -\frac{\dd}{\dn t}\big[L_\eps(x)-L_\eps(x_\eps)\big]
    = -L_\eps'(x)\dot x 
    &= L_\eps'(x)^2 + \eps L_\eps'(x)N_\eps'(x) \\
    &\ge (1-\eps) L_\eps'(x)^2 
    \ge 2\lambda(1-\eps)\big[L_\eps(x)-L_\eps(x_\eps)\big],
\end{align*}
the last inequality being a consequence of $L_\eps$'s $\lambda$-convexity. It follows that $L_\eps(x)$ tends to its minimal value at the (diminished) exponential rate $2\lambda(1-\eps)$. It then further follows easily that $E_\eps(x)$ decays at the same exponential rate, and also that the solution $x$ itself goes to $\bar x_\eps$ at that rate. And in particular, the chosen minimizer $\bar x_\eps$ turns out to be the unique one.


For \eqref{eq:system-intro}, the roles of $E_\eps$ and of $H$ are played by $\me_\eps$, which is $\lambda$-uniformly geodesically convex for $\eps=0$, and by the integral of $h(\rho_1,\rho_2)$, that has no geodesic convexity, respectively. The auxiliary functionals $L_\eps$ and $N_\eps$ correspond to, respectively,
\begin{align}
    \label{eq:L_def}
    \ml_\eps(\rrho) &:= \intrd \left[F_1(\rho_1) + F_2(\rho_2) + \rho_1 K \ast \rho_2 + \eps (\rho_1 V_1 + \rho_2 V_2)\right]\,\intd x \\
    \label{eq:N_def}
    \mn_\eps(\rrho) &:= \intrd \left[h(\rho_1, \rho_2) - (\rho_1 V_1 + \rho_2 V_2)\right]\,\intd x,
\end{align}
where the auxiliary potentials $V_1$ and $V_2$ are chosen to make a global minimizer $\bar\rrho$ of $\me_\eps$, which we obtain via the direct methods, the unique global minimizer of $\ml_\eps$, i.e.,
\begin{align}
    \label{eq:V_def}
    V_1 := \partial_{r_1}h\big(\bar\rho_1,\bar\rho_2),
    \quad
    V_2 := \partial_{r_2}h\big(\bar\rho_1,\bar\rho_2).    
\end{align}
In contrast to $L_\eps$ in the ODE cartoon above, the modulus of convexity of $\ml_\eps$ is not precisely $\lambda$, but $\lambda-\eps K_0$, where $K_0$ is an upper bound on the second derivatives of $V_1$, $V_2$, see Corollary \ref{cor:l_conv_min}. Finally, the core relation \eqref{eq:artificial} is a functional inequality, see Proposition \ref{thm:key_estimate}.

\subsection{Outline}
We start in Section \ref{sec:hyps} by specifying the assumptions on the nonlinearities $F_j$, the coupling term $h$ and the interaction potential $K$. A generic example of functions which satisfy all assumptions is given in the Appendix.
In Section \ref{sec:energy_min}, we prove existence of a global minimizer $\bar\rrho$ of the energy $\me_\eps$ over $\sspc$ by means of the direct methods from the calculus of variations, derive the Euler-Lagrange system, and obtain additional regularity properties. In Section \ref{sec:min_move_scheme}, we analyze the time-discrete minimizing movement scheme for the energy $\me_\eps$, which is then used to construct a weak solution to \eqref{eq:system-intro}--\eqref{eq:ic-intro}. Finally, in Section \ref{sec:exp_cvgce}, we analyze the long-time behaviour of this weak solution and prove exponential convergence to the unique equilibrium for sufficiently small $\eps>0$.

\section{Preliminaries and hypotheses} 
\label{sec:hyps}

\subsection{Notations and definitions}
Recall that $\mptrd$ is the space of probability measures with finite second moment on $\Rd$; we consider $\mptrd$ with the topology induced by narrow convergence. By abuse of notation, we do not distinguish between an absolutely continuous element $\rho\in\mptrd$ and its (Lebesgue-)density $\rho\in L^1(\Rd)$. The $L^2$-Wasserstein distance $\wass$ is a metric on $\mptrd$, and makes $\mptrd$ a complete metric space. Convergence in $\wass$ is equivalent to weak convergence of measures and convergence of the second moment. For an introduction to Wasserstein metrics and the related gradient flows, we refer to \cite{Santa,Vill}. From the Wasserstein metric on $\mptrd$, we directly obtain a metric $\dst$ on the space $\mptrd^2$ of pairs $\rrho=(\rho_1,\rho_2)$ of measures,
\begin{align*}
    \dst\left(\rrho,\eeta\right)^2 := \wass(\rho_1,\eta_1)^2 + \wass(\rho_2,\eta_2)^2.
\end{align*}

The energy functional $\me_\eps:\mptrd^2\to\R\cup\{+\infty\}$ formally introduced in \eqref{eq:Eeps_def} is defined on pairs $\rrho\in\mptrd^2$ with the interpretation that $\me_\eps(\rrho)=+\infty$ unless $\rho_1$ and $\rho_2$ are both absolutely continuous, and the integral expression is in $L^1(\Rd)$; we shall see below that this definition makes $\me_\eps$ weakly lower semi-continuous, at least for sufficiently small $\eps>0$. For brevity, introduce $\mf_\eps: \Rnn^2 \to \Rnn^2$ by $\mf_\eps(\rr) := F_1(r_1) + F_2(r_2) + \eps h(\rr)$, which allows to rewrite $\me_\eps$ as
\begin{align*}
    \me_\eps(\rrho) = \intrd[\mf_\eps(\rrho)+\rho_1K\ast\rho_2]\dd x.
\end{align*}
The functional $\rrho \mapsto \intrd \rho_1 K \ast \rho_2\,\intd x$ will be called the interaction energy.

Next, recall the definition \eqref{eq:sspc_def} of the set $\sspc\subset\mptrd^2$ of pairs $\rrho\in\mptrd^2$ with combined center of mass at zero,
\begin{align*}
    \mom_1[\rho_1] + \mom_1[\rho_2] = 0 
    \quad \text{where} \quad \mom_1[\rho_j]:=\intrd x\,\rho_j \dd x.
\end{align*}
Since $\me_\eps$ is invariant under translations, i.e.,
\begin{align*}
    \me_\eps(\rrho) = \me_\eps(\sigma_v\rrho) \quad \text{with}\ \sigma_v\rrho = \left(\rho_1(\cdot + v), \rho_2(\cdot + v)\right) \in \mptrd
\end{align*}
for any translation $v\in\R^d$ and any $\rrho\in\mptrd$, and since \eqref{eq:system-intro} conserves the combined center of mass, there is no loss of generality to restrict attention to initial data \eqref{eq:ic-intro} and solutions to \eqref{eq:system-intro} belonging to $\sspc$. 

Finally, we introduce the notation $j'=3-j$ for $j\in\{1,2\}$, so that $\rho_{1'} = \rho_2$ and $\rho_{2'}=\rho_1$.

\subsection{Summary of hypotheses}
Throughout this paper, we consider $F_1, F_2, h$ and $K$ fullfilling the following assumptions:
\begin{hypothesis} \label{hyp:general}
Assumptions on the interaction potential $K:\Rd \to \Rnn$:
\begin{itemize}
\item $K$ is radially symmetric about the origin and it holds $K(0) = 0$.
\item $K$ is $\lambda$-convex with some modulus $\lambda > 0$.
\item It holds $K \in C^2(\Rd)$, and there exists a constant $C_K < +\infty$ with $\|\nabla^2 K(z)\| \leq C_K$ for every $z \in \Rd$.
\end{itemize}
Assumptions on the nonlinearities $F_1, F_2: \Rnn \to \Rnn$ and coupling term $h: \Rnn^2 \to \R$:
\begin{itemize}
\item $F_1$ and $F_2$ are $C^2$ and strictly convex with $F_j''(r) > 0$ for every $r > 0$ and $j = 1,2$.
\item It holds $F_1(0) = F_2(0) = F_1'(0) = F_2'(0) = 0$.
\item Their derivatives $F_1'$ and $F_2'$ grow at least linearly for large inputs, but not faster than $F_1$ and $F_2$, respectively. More precisely, it holds for $j = 1,2$:
\begin{equation} \label{eq:fjp_growth}
\liminf_{r\to +\infty} F_j''(r) > 0,\quad \limsup_{r \to +\infty} \frac{F_j'(r)}{F_j(r)} < +\infty
\end{equation}
\item For small inputs, their second derivatives behave similarly to powers of $r$. Explicitly, for $j = 1,2$ there exist constants $m_j, M_j > 0$, exponents $\beta_j \geq 0$ and an $r_0 > 0$ such that for all $0 \leq r \leq r_0$, it holds
\begin{equation} \label{eq:fjpp_power}
m_j r^{\beta_j}  \leq F_j''(r) \leq M_j r^{\beta_j}.
\end{equation}
\item They satisfy the McCann-condition for displacement convexity: For all $r \geq 0$,
\begin{equation} \label{eq:mccann_cond}
F_j(r) - rF_j'(r) + r^2 F_j''(r) \geq 0.
\end{equation}
\item The coupling function $h$ is $C^2$ and satisfies $h(r_1, 0) = h(0, r_2) = \partial_{r_j}h(r_1,0) = \partial_{r_j}h(0, r_2) = 0$ for $j = 1,2$ and every $r_1, r_2 \geq 0$.
\end{itemize}
\end{hypothesis}
It directly follows from the convexity of $F_j$ that the functions $F_j':\Rnn \to \Rnn$ are strictly increasing and thus invertible, since they vanish at $0$ and grow at least linearly. We can thus introduce the following notations: For $j = 1,2$ and $\uu = (u_1, u_2) \in \Rnn^2$, we define
\begin{equation*}
\theta_j(\uu) := \partial_{r_j} h \left((F_1')^{-1}(u_1), (F_2')^{-1}(u_2) \right).
\end{equation*}
In other words, $\theta_j: \Rnn^2 \to \R$ is such that $ \theta_j(F_1'(r_1), F_2'(r_2)) = \partial_{r_j} h(\rr)$ for all $\rr \in \Rnn^2$.
Further, we define for $i,j = 1,2$ the functions $\theta_{j,i}: \Rnn^2 \to \R$ by
\begin{equation*}
\theta_{j,i}(\uu) := \partial_{u_i} \theta_j(\uu) = \frac{\partial_{r_i} \partial_{r_j} h(\rr)}{F_i''(r_i)},
\end{equation*}
where $\rr = (r_1, r_2) = ((F_1')^{-1}(u_1), (F_2')^{-1}(u_2)) $. In addition to Hypothesis \ref{hyp:general}, we impose the following conditions on the $\theta_{j,i}$:

\begin{hypothesis} \label{hyp:theta}
We assume that the $F_j$ and $h$ are such that
\begin{itemize}
\item The $\theta_{j,i}$ are locally Lipschitz in $\Rnn^2$.
\item There are constants $\kappa_{j, i} > 0$ such that
\begin{equation} \label{eq:thetabdd_swap}
\left|\theta_{j,i}(\uu) \right| \leq \kappa_{j,i} \, \min\left\{1, u_1, u_2, \sqrt{\frac{(F_i')^{-1}(u_i)}{(F_j')^{-1}(u_j)}} \right\}
\end{equation}
for all $\uu \in \Rnn^2$, which can also be stated in terms of $F_1, F_2$ and $h$ as
\begin{equation} \label{eq:bddswap_h_r}
\left| \partial_{r_i} \partial_{r_j} h(\rr) \right| \leq \kappa_{j,i}\, F_i''(r_i)\, \min\left\{1, F_1'(r_1), F_2'(r_2), \sqrt{\frac{r_i}{r_j}}\right\}
\end{equation}
for all $\rr \in \Rnn^2$.
\end{itemize}
\end{hypothesis}
\begin{example}
A simple example for an interaction potential $K$ that satisfies Hypothesis \ref{hyp:general} is given by the quadratic potential $K = \frac{\lambda}{2}|\cdot|^2$ with some $\lambda > 0$. An example for $F_1, F_2$ and $h$ that fulfill Hypotheses \ref{hyp:general} and \ref{hyp:theta} is given by Proposition \ref{prop:example} in the appendix. Note that the assumptions on $K$ are independent from those on $F_1, F_2$ and $h$.
\end{example}

\subsection{Basic implications of the hypotheses}

We assume in all of the following that $F_1,F_2,h$ and $K$ are such that Hypotheses \ref{hyp:general} and \ref{hyp:theta} are satisfied.  We start by investigating the assumptions more closely and proving some general consequences of them, which will be useful in our analysis. We first prove some more properties of $F_j$ and $h$.
\begin{proposition} \label{prop:asmpt_corollaries}
The following hold:
\begin{itemize}
\item There is $\eps_0 > 0$ such that the function \begin{equation} \label{eq:ftwoeps}
\rr \mapsto \mf_{2\eps}(\rr) = F_1(r_1) + F_2(r_2) + 2\eps h(\rr)
\end{equation}
is convex on $\Rnn^2$ for every $0 < \eps \leq \eps_0$, strictly if $\eps <  \eps_0$.
\item For all $\rr \in \Rnn^2$ and $\eps \leq \eps_0$, it holds with the constants $\kappa_{j,i}$ from \eqref{eq:thetabdd_swap}:
\begin{equation} \label{eq:fj_feps_estim}
\frac{1}{2} F_1(r_1) + \frac{1}{2} F_2(r_2) \leq \mf_\eps(\rr) \leq \left(1 + \frac{\eps}{2}\kappa_{1,1}\right)F_1(r_1) + \left(1 + \frac{\eps}{2}\kappa_{2,2}\right)F_2(r_2).
\end{equation}
\item With the constants $r_0$ and $m_j,M_j,\beta_j$ from \eqref{eq:fjpp_power}, it holds for $0 \leq r \leq r_0$:
\begin{align} \label{eq:fjp_power}
\frac{m_j}{\beta_j + 1} r^{\beta_j + 1} &\leq F_j'(r) \leq \frac{M_j}{\beta_j + 1} r^{\beta_j + 1}, \\ \label{eq:fj_power}
\frac{m_j}{(\beta_j + 1)(\beta_j + 2)} r^{\beta_j + 2} &\leq F_j(r) \leq \frac{M_j}{(\beta_j + 1)(\beta_j + 2)} r^{\beta_j + 2}.
\end{align}
\item For $j = 1,2$, there are constants $\alpha_j$ such that for all $r \geq 0$, it holds
\begin{equation} \label{eq:estim_dfj}
F_j'(r) \leq \alpha_j\, \left( \min\{1,r\} + F_j(r)\right).
\end{equation}
\item For any bound $H > 0$, there is a constant $\beta_H > 0$ such that for all $\bar\rr \in \Rnn^2$ with $\bar r_1, \bar r_2 \leq H$, all $\rr \in \Rnn^2$ and any $i,j \in \{1,2 \}$, the estimate
\begin{equation} \label{eq:thetabregmanbd}
\max \{r_1, r_2\}\, \left| \theta_{j,i}(\uu) - \theta_{j,i}(\bar\uu)\right|^2 \leq \beta_H\, \left( \dst_{F_1}(r_1| \bar r_1) + \dst_{F_2}(r_2 | \bar r_2) \right)
\end{equation}
holds, where $\uu = \left( F_1'(r_1), F_2'(r_2)\right)$, $\bar\uu = \left( F_1'(\bar r_1), F_2'(\bar r_2)\right)$ and $\dst_F (\cdot | \cdot)$ is the \upshape Bregman divergence \itshape of a convex function $F$, given by
\begin{equation*}
\dst_F(r|\bar r) := F(r) - F(\bar r) - F'(\bar r)(r-\bar r).
\end{equation*}
\end{itemize}
\end{proposition}
\begin{proof}
We prove convexity of \eqref{eq:ftwoeps} by observing that for every $\rr \in \Rnn^2$, it holds
\begin{align*}
\nabla_\rr^2 \mf_{2\eps}(\rr) = \begin{pmatrix}
F_1''(r_1) + 2 \eps \partial_{r_1} \partial_{r_1} h(\rr) & 2\eps \partial_{r_2} \partial_{r_1} h(\rr) \\ 2\eps\partial_{r_1} \partial_{r_2} h(\rr) & F_2''(r_2) + 2\eps\partial_{r_2} \partial_{r_2} h(\rr)
\end{pmatrix}.
\end{align*}
We show that for some $\eps_0 > 0$, this matrix is positive semi-definite at every $\rr \in \Rnn^2$. Indeed, it holds for any $\vv = (v_1, v_2) \in \R^2$ by assumption \eqref{eq:bddswap_h_r} and $2v_1v_2 \geq -v_1^2-v_2^2$:
\begin{align*}
\vv^T \nabla_\rr^2 \mf_{2\eps}(\rr) \vv &=  \left(F_1''(r_1) + 2 \eps \partial_{r_1} \partial_{r_1} h(\rr)\right)v_1^2 + \left( F_2''(r_2) + \eps\partial_{r_2} \partial_{r_2} h(\rr)\right)v_2^2  + 4\eps \partial_{r_1} \partial_{r_2} h(\rr)v_1v_2 \\
&\geq \left(1-2\eps(\kappa_{1,1} + \kappa_{2,1})\right) \,F_1''(r_1) v_1^2 + \left(1-2\eps(\kappa_{1,2} + \kappa_{2,2})\right)\, F_2''(r_2) v_2^2 \geq 0
\end{align*}
for every sufficiently small $\eps > 0$ by non-negativity of $F_j''(\rr)$. Hence there exists $\eps_0 >0$ such that $F_{2\eps_0}$ is convex on $\Rnn^2$. For every $\eps < \eps_0$, it follows that
\begin{align*}
\mf_{2\eps}(\rr) = \frac{\eps}{\eps_0} F_{2\eps_0}(\rr) + \left(1- \frac{\eps}{\eps_0}\right) \left(F_1(r_1) + F_2(r_2)\right)
\end{align*}
is strictly convex on $\Rnn^2$ by strict convexity of $F_1$ and $F_2$, proving the first claim. To show the first inequality in \eqref{eq:fj_feps_estim}, notice that $\mf_{2\eps}$ is non-negative in $\Rnn^2$, since it is convex and its derivative vanishes at 0. Hence 
\begin{align*}
2 \mf_\eps(\rr) = \mf_{2\eps}(\rr) + F_1(r_1) + F_2(r_2) \geq F_1(r_1) + F_2(r_2).
\end{align*}
The second inequality in \eqref{eq:fj_feps_estim} will follow from the fact that $|\partial^2_{r_i r_i} h(\rr)| \leq \kappa_{i,i} F_i''(r_i)$ for all $\rr \in \Rnn^2$ and $i = 1,2$ by assumption \eqref{eq:bddswap_h_r}. Integrated twice from $0$ to $r_i$ and using the assumption that $h$ and its first derivatives vanish on $\partial\Rnn^2$, this yields $|h(\rr)| \leq \kappa_{i,i} F_i(r_i)$ for all $\rr$, implying \eqref{eq:fj_feps_estim} by definition of $\mf_\eps$.

Inequalities \eqref{eq:fjp_power} and \eqref{eq:fj_power} follow directly by integrating \eqref{eq:fjpp_power} from 0 to $r$ and using that $F_j(0) = F_j'(0) = 0$. The proof of \eqref{eq:estim_dfj} can be done for $r \leq r_0$ and $r > r_0$ separately: For $r \leq r_0$, the inequality follows from \eqref{eq:fjp_power} since $\beta_j \geq 0$ and thus $r^{\beta_j + 1}$ decays at least linearly for small $r$. The case $r > r_0$ follows from the growth assumption \eqref{eq:fjp_growth} and the fact that the right-hand side of \eqref{eq:fj_power} is bounded away from 0 for $ r > r_0$.

It remains to prove \eqref{eq:thetabregmanbd}. Fix $H > 0$ and let $\rr, \bar\rr \in \Rnn^2$ with $\bar r_1, \bar r_2 \leq H$. Define
\begin{align*}
m_H := \min_{s \geq H,\, j = 1,2} F_j''(s), \quad M_H := \max_{s \in [0, 3H],\, j = 1,2} F_j''(s).
\end{align*}
It holds $0 < m_H \leq M_H < +\infty$, by $F_j''(r) > 0$ for $r > 0$ and assumption \eqref{eq:fjp_growth}.
By symmetry of the assumptions and the claim, we can assume w.l.o.g. $r_1 \geq r_2$. We split the proof into two cases: $r_1 \geq 3H$ and $r_1 \leq 3H$. 

Consider first $r_1 \geq 3H$. We have $\bar r_1 \leq H$ and $r_1 \geq \frac{r_1 + H}{2}$, thus it holds
\begin{align*}
\dst_{F_1}(r_1|\bar r_1) = \int_{\bar r_1}^{r_1} F_1''(s)(r_1 - s)\,\intd s \geq \int_H^{\frac{r_1 + H}{2}} F_1''(s)(r_1 - s)\,\intd s \geq \frac{r_1 - H}{2}\, m_H H
\end{align*}
where we have used that $F_1''(s) \geq m_H$ and $r_1 - s \geq r_1 - \frac{r_1 + H}{2} \geq H$ for every $s \in \left[H, \frac{r_1 + H}{2}\right]$ by definition of $m_H$ and $r_1 \geq 3H$, respectively. From \eqref{eq:thetabdd_swap} and $H \leq \frac{r_1 - H}{2}$, it follows
\begin{align*}
r_1 \left|\theta_{j,i}(\uu) - \theta_{j,i}(\bar\uu)\right|^2 \leq 4 \kappa_{j,i}^2 r_1 = 4\kappa_{j,i}^2(r_1 - H + H) \leq 6 \kappa_{j,i}^2(r_1 - H) \leq \frac{12\kappa_{j,i}^2}{m_HH}\, \dst_{F_1}(r_1|\bar r_1).
\end{align*}

In the second case, we have $r_1 \leq 3H$, and thus also $r_2 \leq 3H$ since we assumed $r_1 \geq r_2$. By the monotonicity of $F_1'$ and $F_2'$, this implies that $\uu, \bar \uu \in [0, F_1'(3H)]\times[0,F_2'(3H)]$, which is a compact subset of $\Rnn^2$. Hence by Hypothesis \ref{hyp:theta}, there is a constant $L_H$ such that $\theta_{j,i}$ is $L_H$-Lipschitz on this subset. Thus
\begin{align*}
r_1 \left|\theta_{j,i}(\uu) - \theta_{j,i}(\bar\uu)\right|^2 \leq 3HL_H^2\left((u_1 - \bar u_1)^2 + (u_2 - \bar u_2)^2\right).
\end{align*}
We estimate $(u_j - \bar u_j)^2$ by means of $\dst_{F_j}(r_j|\bar r_j)$: We have
\begin{align*}
(u_j - \bar u_j)^2 = \left(F_j'(r_j) - F_j'(\bar r_j)\right)^2 = \int_{\bar r_j}^{r_j} 2 \left(F_j'(s) - F_j'(\bar r_j)\right) F_j''(s)\,\intd s \leq 2M_H\,\dst_{F_j}(r_j|\bar r_j),
\end{align*}
since the term $F_j'(s) - F_j'(\bar r_j)$ is negative if and only if $r_j < \bar r_j$, in which case the endpoints $\bar r_j$ and $r_j$ of the integral are swapped.  This finishes the proof of \eqref{eq:thetabregmanbd}.
\end{proof}
We additionally prove some important general properties of the energy functional $\me_\eps$ and its components. In all of the following, we assume $\eps \leq \eps_0$ with $\eps_0$ from Proposition \ref{prop:asmpt_corollaries}.

\begin{lemma} \label{lem:eeps_bounds}
Let $\rrho \in \mptrd^2$ be any pair and let $\ms \subset \mptrd^2$ be any family of pairs.
\begin{itemize}
\item The following estimates hold:
\begin{align} \label{eq:estim_fj_int_eeps}
\intrd [F_1(\rho_1) + F_2(\rho_2)]\,\intd x \leq 2\me_\eps(\rrho), \quad \intrd \rho_1 K \ast \rho_2\,\intd x \leq \me_\eps(\rrho).
\end{align}
In particular, the energy $\me_\eps(\rrho)$ is uniformly bounded for $\rrho \in \ms$ if and only if both $F_j(\rho_j)$ and $\rho_1 K \ast \rho_2$ are uniformly bounded in $L^1(\Rd)$ for $\rrho \in \ms$.
\item With the constants $\alpha_j$ from \eqref{eq:estim_dfj}, it holds:
\begin{equation} \label{eq:estim_fjp_eeps}
\intrd F_j'(\rho_j)\,\intd x \leq \alpha_j\,(1+2\me_\eps(\rrho)).
\end{equation}
\item For $\rrho \in \sspc$, the interaction energy $\intrd \rho_1 K \ast \rho_2\,\intd x$ satisfies the bounds
\begin{equation} \label{eq:estim_inter_2mom}
\frac{\lambda}{2} \left(\mom_2[\rho_1] + \mom_2[\rho_2]\right) \leq \intrd \rho_1 K \ast \rho_2\,\intd x \leq C_K\left(\mom_2[\rho_1] + \mom_2[\rho_2]\right).
\end{equation}
In particular, for $\ms \subset \sspc$, $\rrho$'s interaction energy is uniformly bounded for $\rrho \in \ms$ if and only if $\rho_1$ and $\rho_2$ have uniformly bounded second moments for $\rrho \in \ms$.
\end{itemize}
\end{lemma}
\begin{proof}
The interaction energy is clearly non-negative by our assumptions on $K$, thus the first estimate in \eqref{eq:estim_fj_int_eeps} follows directly from \eqref{eq:fj_feps_estim}. The second one holds since $\mf_\eps$ is a convex function whose derivative vanishes at 0, and thus is non-negative in $\Rnn^2$. By \eqref{eq:estim_dfj}, it holds $F_j'(\rho_j) \leq \alpha_j\,(\rho_j + F_j(\rho_j))$. From this we obtain \eqref{eq:estim_fjp_eeps} by integrating and using \eqref{eq:estim_fj_int_eeps} and the fact that $\rho_j$ has unit mass.

In order to prove \eqref{eq:estim_inter_2mom}, note that by $\lambda$-convexity of $K$ and the bound on its second derivative, it holds $\frac{\lambda}{2}|z|^2 \leq K(z) \leq \frac{C_K}{2} |z|^2$. Since for $\rrho \in \sspc$, we have $\mom_1[\rho_2] = -\mom_1[\rho_1]$, we thus obtain the lower bound
\begin{align*}
\intrd \rho_1 K \ast \rho_2\, \intd x &= \iintrdrd \rho_1(x) \rho_2(y) K(x-y) \, \intd x \intd y
\geq \frac{\lambda}{2} \iintrdrd \rho_1(x) \rho_2(y) |x-y|^2 \, \intd x \intd y \\
&=  \frac{\lambda}{2} \iintrdrd \rho_1(x) \rho_2(y) \left[ |x|^2 + |y|^2 - 2x \cdot y \right] \, \intd x \intd y \\
&= \frac{\lambda}{2} \left( \mom_2[\rho_1] + \mom_2[\rho_2] \right) - \lambda \mom_1[\rho_1] \cdot \mom_1[\rho_2] = \frac{\lambda}{2} \left( \mom_2[\rho_1] + \mom_2[\rho_2] \right) + \lambda |\mom_1[\rho_1]|^2 \\
&\geq \frac{\lambda}{2} \left( \mom_2[\rho_1] + \mom_2[\rho_2] \right).
\end{align*}
For the upper bound, the same argument yields
\begin{align*}
\intrd \rho_1 K \ast \rho_2\,\intd x \leq \frac{C_K}{2}\left( \mom_2[\rho_1] + \mom_2[\rho_2] \right) + C_K |\mom_1[\rho_1]|^2 \leq C_K\left( \mom_2[\rho_1] + \mom_2[\rho_2] \right),
\end{align*}
where the last step follows from $|\mom_1[\rho_1]|^2 = \frac{1}{2}(|\mom_1[\rho_1]|^2 + |\mom_1[\rho_2]|^2) \leq \frac{1}{2}(\mom_2[\rho_1] + \mom_2[\rho_2])$.
\end{proof}
In the next section, we shall use these properties to proof existence of a global minimizer $\bar\rrho \in \sspc$ of $\me_\eps$ via the direct method from the calculus of variations.

\section{Existence and properties of the minimizer for the energy functional} \label{sec:energy_min}

\subsection{Existence of minimizers}
We start by proving the following weak coercivity result for the energy $\me_\eps$. The proof mainly uses the estimates from Lemma \ref{lem:eeps_bounds}.
\begin{lemma} \label{lem:eeps_coer}
The sublevel sets of the energy functional $\me_\eps$ in $\sspc$ are pre-compact in the weak $L^1$-topology, i.e. any sequence of pairs $(\rrho^n) = (\rho_1^n, \rho_2^n) \subset \sspc$ with $\me_\eps(\rrho^n)$ uniformly bounded contains a subsequence converging weakly in $L^1(\Rd)$ to some $\rrho \in \sspc$. The same holds true for sequences $(\rrho^n) \subset \mptrd^2$ with $\me_\eps(\rrho^n)$ uniformly bounded, such that their combined center of mass $\mom_1[\rho_1^n] + \mom_1[\rho_2^n]$ converges to $0$. 
\end{lemma}
\begin{proof}
Take any sequence $(\rrho^n) \subset \mptrd$ such that $\mom_1[\rho_1^n] + \mom_1[\rho_2^n] \to 0$ and $\me_\eps(\rrho^n)$ is uniformly bounded. We consider the shifted pair $\tilde\rrho^n := \rrho^n(\cdot - v_n)$ with the shift $v_n := \frac{1}{2}(\mom_1[\rho_1^n] + \mom_1[\rho_2^n])$. We have $\tilde\rrho^n \in \sspc$, and from the translation invariance of $\me_\eps$, it follows that $\me_\eps(\tilde\rrho^n) = \me_\eps(\rrho^n)$ is uniformly bounded. It thus follows from \eqref{eq:estim_fj_int_eeps} and \eqref{eq:estim_inter_2mom} that the second moments $\mom_2[\tilde\rho_1^n], \mom_2[\tilde\rho_2^n]$ are uniformly bounded. Together with $v_n \to 0$, this implies uniform boundedness of the second moments of $\rho_j^n$:
\begin{align*}
\mom_2[\rho_j^n] = \intrd |x|^2\tilde\rho_j^n(x + v_n)\,\intd x = \intrd |y - v_n|^2\tilde\rho_j^n(y)\,\intd y \leq 2\mom_2[\tilde\rho_j^n] + 2|v_n|^2.
\end{align*}  In particular, the sequence $(\rho_j^n)$ is tight as probability measures for $j=1,2$, and thus converges narrowly to a $\rho_j \in \mptrd$ along some subsequence, again denoted $(\rho_j^n)$. Since narrow convergence and uniform boundedness of the second moments imply $\mom_1[\rho_j^n] \to \mom_1[\rho_j]$, we have $\rrho \in \sspc$.

To prove $\rho_j^n \weakto \rho_j$ weakly in $L^1(\Rd)$, observe that by \eqref{eq:estim_fj_int_eeps}, the integrals of $F_j(\rho_j^n)$ are uniformly bounded. Since $F_j(r)$ grows super-linearly for $r \to +\infty$ by the at least linear growth of $F_j'(r)$, the de la Vallée-Poussin compactness criterion, see e.g. \cite[Theorem 1.28]{Roub}, yields compactness of the sequence $\rho_j^n$ with respect to local weak $L^1$-convergence. Together with the narrow convergence $\rho_j^n \weakto \rho_j$, this implies $\rho_j^n \weakto \rho_j$ weakly in $L^1_{loc}(\Rd)$. Together with tightness, this yields $\rho_j^n \weakto \rho_j$ weakly in $L^1(\Rd)$, proving the claim.
\end{proof}
In order to apply the direct method in the calculus of variations, we need to prove lower semi-continuity of the functional $\me_\eps$ in the same topology.
\begin{lemma} \label{lem:eeps_lsc}
The energy functional $\me_\eps$ is lower semi-continuous in $\mptrd^2$ with respect to weak $L^1$-convergence.
\end{lemma}
\begin{proof}
Since $\mf_\eps$ is non-negative by \eqref{eq:fj_feps_estim}, Fatou's lemma yields lower semi-continuity of the functional $\rrho \mapsto \intrd \mf_\eps(\rrho)\,\intd x$ with respect to strong $L^1$-convergence. Since $\mf_\eps$ is a convex function, this implies weak lower semi-continuity. It remains to show that the interaction energy $\intrd \rho_1 K \ast \rho_2\,\intd x$ is lower semi-continuous with respect to weak $L^1$-convergence. Let $\rrho^n \in \mptrd^2$ with $\rho_j^n \weakto \rho_j$ weakly in $L^1(\Rd)$, and for any radius $R > 0$, define
\begin{align*}
K_R \in L^{\infty}(\Rd),\ K_R(z) := \left\{\begin{array}{ll}
K(z) & |z| \leq R \\ 0 & |z| > R
\end{array}
\right..
\end{align*} 
Also note that $\rho_j^n \weakto \rho_j$ weakly in $L^1(\Rd)$ for $j = 1,2$ implies $\rho_1^n \otimes \rho_2^n \weakto \rho_1 \otimes \rho_2$ weakly in $L^1(\Rd \times \Rd)$. Since $0 \leq K_R \leq K$ in $\Rd$, this yields
\begin{align*}
&\liminf_{n \to \infty} \intrd \rho_1^n K \ast \rho_2^n\,\intd x \geq \liminf_{n \to \infty} \intrd \rho_1^n K_R \ast \rho_2^n\,\intd x + \liminf_{n \to \infty} \intrd \rho_1^n (K - K_R) \ast \rho_2^n\,\intd x \\
&\geq \liminf_{n \to \infty} \iintrdrd \rho_1^n \otimes \rho_2^n(x,y)\, K_R(x-y)\,\intd x\intd y = \intrd \rho_1 K_R \ast \rho_2\,\intd x.
\end{align*}
Since $R$ was arbitrary, we can take the limit $R \to \infty$. By the monotone convergence theorem, the integral on the right-hand side converges to $\intrd \rho_1 K \ast \rho_2 \,\intd x$, proving weak lower semi-continuity of the interaction energy.
\end{proof}
Together, these two lemmas yield existence of a global minimizer of the energy $\me_\eps$ in $\sspc$.
\begin{corollary} \label{cor:ex_min}
There exists a pair $\bar\rrho = (\bar\rho_1, \bar\rho_2) \in \sspc$, depending on $\eps \in (0, \eps_0]$, that minimizes the energy functional $\me_\eps$ over $\mptrd$.
\end{corollary}
\begin{proof}
The existence of $\bar\rrho \in \sspc$ that minimizes $\me_\eps$ over $\sspc$ follows from the previous two lemmas via the direct method from the calculus of variations. Since $\me_\eps$ is translation invariant, the minimality of $\bar\rrho$ over $\sspc$ implies that $\bar\rrho$ also minimizes $\me_\eps(\bar\rrho)$ over the whole space $\mptrd^2$.
\end{proof}
\begin{remark}
Note that we do not obtain uniqueness of the minimizer $\bar\rrho$ directly, since the energy functional $\me_\eps$ is not convex on $\sspc$, neither along geodesics, because of the coupling term, nor in the flat sense, because of the interaction term. We shall prove uniqueness of the minimizer $\bar\rrho$ for every sufficiently small $\eps > 0$ in section \ref{sec:exp_cvgce}, see Corollary \ref{cor:min_unique}.
\end{remark}

\subsection{Euler-Lagrange system}
We now use a variational argument to prove that the minimizer $\bar\rrho$ obtained from Corollary \ref{cor:ex_min} satisfies the Euler-Lagrange system \eqref{eq:euler-lagrange_intro}.

\begin{lemma}\label{lem:el_intineq}
Let $\psi_1, \psi_2 \in L^1(\Rd) \cap L^\infty(\Rd)$ be compactly supported, and assume that  $\intrd \psi_1 = \intrd \psi_2 = 0$ and $\bar\rho_j + \psi_j \geq 0$ in $\Rd$ for $j = 1,2$. Additionally, assume that there exist constants $H_j > 0$ such that $\psi_j = 0$ on the set $\{\bar\rho_j > H_j\}$. Then it holds
    \begin{equation*}
        \intrd \left( F_1'(\bar\rho_1) + \varepsilon \partial_{r_1} h(\bar\rho_1, \bar\rho_2) + K \ast \bar\rho_2 \right) \psi_1 \,\intd x\, + \intrd \left( F_2'(\bar\rho_2) + \varepsilon \partial_{r_2} h(\bar\rho_1, \bar\rho_2) + K \ast \bar\rho_1 \right) \psi_2 \,\intd x \geq 0.
    \end{equation*}
\end{lemma}
\begin{proof}
    Let $\psi_1$ and $\psi_2$ as above. Then for any $s \in[0,1]$, the perturbation $\rrho^s := \bar{\rrho} + s \left(\psi_1, \psi_2\right)$ fulfills $\rho^s_j \geq 0$ and $\int \rho^s_j = 1$ for $j = 1,2$, and thus defines a pair of probability distributions. Since $\bar\rrho$ minimizes $\me_\eps$ over $\mptrd^2$, it holds $\me_\eps(\rrho^s) \geq \me_\eps(\bar\rrho)$ for all $s \in [0,1]$, hence  
    \begin{align*}
        \ddszerop \me_\eps (\rrho^s) = \ddszerop \me_\eps (\bar\rho_1 + s \psi_1, \bar\rho_2 + s \psi_2) \geq 0\,.
    \end{align*}
    Using the properties of $\psi_1$, $\psi_2$, in particular the fact that they are supported on a set where $\bar\rho_j$ is bounded, the dominated convergence theorem applies and provides that 
    \begin{align*}
        0 &\leq \ddszerop \intrd \left[\mf_\eps(\bar\rho_1 + s \psi_1,\bar\rho_2 + s \psi_2) + (\bar\rho_1 + s \psi_1) K \ast (\bar\rho_2 + s \psi_2) \right] \intd x \\
        &= \intrd \left( \partial_{r_1} \mf_\eps (\bar\rho_1, \bar\rho_2) + K \ast \bar\rho_2 \right) \psi_1 \,\intd x + \intrd \left( \partial_{r_2} \mf_\eps (\bar\rho_1, \bar\rho_2) + K \ast \bar\rho_1 \right) \psi_2 \,\intd x \\
        &= \intrd \left( F_1'(\bar\rho_1) + \varepsilon \partial_{r_1} h(\bar\rho_1, \bar\rho_2) + K \ast \bar\rho_2 \right) \psi_1 \,\intd x \\ 
        & \qquad + \intrd \left( F_2'(\bar\rho_2) + \varepsilon \partial_{r_2} h(\bar\rho_1, \bar\rho_2) + K \ast \bar\rho_1 \right) \psi_2 \,\intd x,
    \end{align*}
    which proves the asserted integral inequality.
\end{proof}
\begin{lemma} \label{lem:el_cineq}
There are positive constants $C_1,C_2\in\R$ such that $\bar\rrho$ satisfies
\begin{equation} \label{eq:el_c1ineq}
\begin{split}
F_1'(\bar\rho_1) + \eps\partial_{r_1}h(\bar\rho_1,\bar\rho_2) + K*\bar\rho_2 &\ge C_1 \\
F_2'(\bar\rho_2) + \eps\partial_{r_2}h(\bar\rho_1,\bar\rho_2) + K*\bar\rho_1 &\ge C_2
\end{split}
\end{equation}
with equality a.e. on $\{\bar\rho_1 > 0 \}$ and a.e. on $\{\bar\rho_2 > 0\}$ respectively.
\end{lemma}
\begin{proof}
For brevity, we introduce the notation
\[ EL_1 := F_1'(\bar\rho_1) + \eps\partial_{r_1}h(\bar\rho_1,\bar\rho_2) + K*\bar\rho_2,\]
and similarly $EL_2$. Since both inequalities in \eqref{eq:el_c1ineq} are analogous to each other, we only need to prove the first one. Suppose there does not exist a constant $C_1$ such that $EL_1 \equiv C_1$ a.e on $\{\bar\rho_1 > 0 \}$ and $EL_1 \geq C_1$ a.e. on $\Rd$. Then there exist constants $\alpha < \beta$, a height $H > 0$ and a radius $R > 0$ such that the bounded measurable sets 
\[M_{\alpha} := \{EL_1 \leq \alpha \} \cap \{\bar\rho_1 \leq H\} \cap B_R, \qquad M_{\beta} := \{EL_1 \geq \beta \} \cap \{0 < \bar\rho_1 \leq H\} \cap B_R\]
both have strictly positive measure. Now define
\[ \psi_1 := -\bar\rho_1 \mathds{1}_{M_{\beta}} + \frac{1}{|M_{\alpha}|} \left(\int_{M_{\beta}} \bar\rho_1 \, \intd x \right) \mathds{1}_{M_{\alpha}}.\]
The properties $EL_1 \leq \alpha$ on $M_{\alpha}$ and $EL_1 \geq \beta$ on $M_{\beta}$ yield
\begin{align*}
\intrd EL_1\, \psi_1\,\intd x &= -\int_{M_{\beta}} EL_1\,\bar\rho_1\, \intd x +  \frac{1}{|M_{\alpha}|}\left(\int_{M_{\beta}} \bar\rho_1 \, \intd x \right) \int_{M_{\alpha}} EL_1\,\intd x \\
&\leq (\alpha - \beta) \int_{M_{\beta}} \bar\rho_1 \, \intd x < 0,
\end{align*}
where the last inequality follows from $\bar\rho_1 > 0$ on $M_{\beta}$.
This is however a contradiction to Lemma \ref{lem:el_intineq} with $\psi_2 \equiv 0$, since it is easily checked that $\psi_1 \in L^1(\Rd)$ satisfies all properties needed to apply the lemma.

It remains to show that $C_1$ and $C_2$ are strictly positive. We observe that the strict convexity of $\mf_\eps$ implies for $j = 1,2$:
\begin{equation} \label{eq:DFeps_nn}
F_j'(r_j) + \eps \partial_{r_j} h(\rr) = \partial_{r_j} \mf_\eps(\rr) \geq 0 \quad \text{for all } \rr \in \Rnn^2,
\end{equation}
with equality if and only if $r_j = 0$. The interaction terms $K \ast \bar\rho_j$ are strictly positive on $\Rd$ by positivity of $K$ in $\Rd \setminus \{0\}$ and the fact that $\bar\rho_j$ is non-singular. Hence $EL_j(x) > 0$ for all $x \in \Rd$ and $j = 1,2$, and thus $C_1, C_2 > 0$. This concludes the proof.
\end{proof}
\begin{lemma} \label{lem:el_rhogr0}
For a.e. $x \in \Rd$, it holds
\begin{align*}
K \ast \bar\rho_2(x) > C_1 \quad &\Leftrightarrow \quad \bar\rho_1(x) = 0 \\
K \ast \bar\rho_2(x) < C_1 \quad &\Leftrightarrow \quad \bar\rho_1(x) > 0
\end{align*}
and the same statements hold true for $\bar\rho_2(x) \geqq 0$ and $K \ast \bar\rho_1(x) \lessgtr C_2$.
\end{lemma}
\begin{proof}
If $K \ast \bar\rho_2(x) > C_1$, then it follows from \eqref{eq:DFeps_nn} that \eqref{eq:el_c1ineq} does not hold with equality, hence up to a null-set, $\bar\rho_1(x) = 0$. On the other hand, if $K \ast \bar\rho_2(x) < C_1$, then \eqref{eq:el_c1ineq} implies $F_1'(\bar\rho_1(x)) + \eps\partial_{r_1}h(\bar\rho_1(x),\bar\rho_2(x)) > 0$, and thus $\bar\rho_1(x) > 0$.

It follows immediately that $\bar\rho_1(x) = 0$ implies $K \ast \bar\rho_2(x) \geq C_1$ and that $\bar\rho_1(x) > 0$ implies $K \ast \bar\rho_2(x) \leq C_1$ for a.e. $x$. Since $K$, and thus $K \ast \bar\rho_2$, is $\lambda$-convex on $\Rd$, the set $\{x: K \ast \bar\rho_2(x) = C_1 \}$ has measure zero, which concludes the proof.
\end{proof}
From the previous two results, we obtain the Euler-Lagrange system and a characterization of the supports of $\bar\rho_1$ and $\bar\rho_2$:
\begin{theorem} \label{thm:euler-lagrange}
Any global minimizer $\bar\rrho \in \sspc$ of the energy functional $\me_\eps$ satisfies the Euler-Lagrange system
\begin{equation} \label{eq:euler-lagrange}
    \begin{split}
        F_1'(\bar\rho_1) + \varepsilon \partial_{r_1} h(\bar\rho_1, \bar\rho_2) &= \left(C_1 - K * \bar\rho_2\right)_+ \\
        F_2'(\bar\rho_2) + \varepsilon \partial_{r_2} h(\bar\rho_1, \bar\rho_2) &= \left(C_2 - K * \bar\rho_1\right)_+ 
    \end{split}
\end{equation}
in $\Rd$. Moreover, the supports of $\bar\rho_1$ and $\bar\rho_2$ are given by the compact convex sets
\begin{equation} \label{eq:supp_minimizer}
\begin{split}
\supp\,\bar\rho_1 &= \{K \ast \bar\rho_2 \leq C_1\} \\
\supp\,\bar\rho_2 &= \{K \ast \bar\rho_1 \leq C_2\}.
\end{split}
\end{equation}
\end{theorem}
\begin{proof}
The form \eqref{eq:euler-lagrange} of the Euler-Lagrange system follows immediately from Lemmas \ref{lem:el_cineq} and \ref{lem:el_rhogr0}, and the characterization \eqref{eq:supp_minimizer} of the supports follows from Lemma \ref{lem:el_rhogr0} and the observation $\overline{\{K \ast \bar\rho_2 < C_1\}} = \{K \ast \bar\rho_2 \leq C_1\}$, which holds by $\lambda$-convexity of $K \ast \bar\rho_2$.
\end{proof}
\begin{remark}
In the case that $K = \frac{\lambda}{2}|\cdot|^2$ is the quadratic interaction potential, the Euler-Lagrange system can be written as
\begin{equation} \label{eq:el_quadratic}
\begin{split}
F_1'(\bar\rho_1) + \eps \partial_{r_1}h(\bar\rho_1, \bar\rho_2) &= (\tilde C_1 - \frac{\lambda}{2}|x|^2)_+ \\
F_2'(\bar\rho_2) + \eps \partial_{r_2}h(\bar\rho_1, \bar\rho_2) &= (\tilde C_2 - \frac{\lambda}{2}|x|^2)_+.
\end{split}
\end{equation}
with the constants $\tilde C_1, \tilde C_2 > 0$ chosen uniquely such that $\intrd \bar\rho_j \,\intd x = 1$ for $j = 1,2$. It holds $\tilde C_j = C_j - \frac{\lambda}{2} \mom_2[\bar\rho_{j'}]$ with the constants $C_j$ from \eqref{eq:euler-lagrange}; recall that $j' := 3-j$. The form \eqref{eq:el_quadratic} is considerably simpler to solve than \eqref{eq:euler-lagrange}, since the right-hand side does not explicitly depend on $\bar\rrho$. It can be proven by showing that in the case $K = \frac{\lambda}{2}|\cdot|^2$, the energy functional $\me_\eps$ is strictly convex in the flat sense on $\sspc$, implying that the minimizer $\bar\rrho \in \sspc$ is unique. This uniqueness yields radial symmetry of $\bar\rrho$, see Lemma \ref{lem:min_uniq}, hence $\mom_1[\rho_j] = 0$ for both components. A direct calculation then yields $K \ast \bar\rho_j = \frac{\lambda}{2}\left[|\cdot|^2 + \mom_2[\bar\rho_j]\right]$, thus \eqref{eq:el_quadratic} follows from \eqref{eq:euler-lagrange}. 
\end{remark}

\subsection{Properties of the minimizers}

In the next results, we prove some additional properties of the minimizers $\bar\rrho$.

\begin{lemma} \label{lem:minim_bdd}
For any $\bar\eps \in (0, \eps_0]$, the minimal energy $\me_\eps(\bar\rrho)$, the constants $C_1, C_2$ from the Euler-Lagrange system \eqref{eq:euler-lagrange} and the diameters of $\supp\,\bar\rho_1$ and $\supp\,\bar\rho_2$ are $\eps$-uniformly bounded for $\eps \in (0, \bar\eps]$. If $\bar\eps  < 1 / \max\{\kappa_{1,1}, \kappa_{2,2}\} $, there exists an $\eps$-independent constant $H_0 > 0$ such that $\bar\rho_j \leq H_0$ in $\Rd$ for $j = 1,2$ and any $\eps \in (0, \bar\eps]$.
\end{lemma}
\begin{proof}
Since for any fixed $\rrho \in \sspc$, the energy $\me_\eps(\rrho)$ is $\eps$-uniformly bounded for small $\eps$, the minimal value $\me_\eps(\bar\rrho)$ is uniformly bounded as well. To prove the other claims, we write the Euler-Lagrange system using the functions $\theta_j$:
\begin{equation} \label{eq:el_theta}
\begin{split}
F_1'(\bar\rho_1) + \eps \theta_1(F_1'(\bar\rho_1), F_2'(\bar\rho_2)) = (C_1 - K \ast \bar\rho_2)_+ \\
F_2'(\bar\rho_2) + \eps \theta_2(F_1'(\bar\rho_1), F_2'(\bar\rho_2)) = (C_2 - K \ast \bar\rho_1)_+
\end{split}
\end{equation}
It follows from assumption \eqref{eq:thetabdd_swap} that $|\theta_j(\uu)| \leq \kappa_{j,i} u_i$ for all $\uu \in \Rnn^2$ and any $i,j = 1,2$. It thus follows from  \eqref{eq:el_theta} that
\begin{align} \label{eq:el_fj_lowerbd}
(1 - \eps \kappa_{1,1})\,F_1'(\bar\rho_1) \leq (C_1 - K \ast \bar\rho_2)_+ \leq (1 + \eps \kappa_{1,1})\,F_1'(\bar\rho_1),
\end{align}
and analogously for $F_2'(\bar\rho_2)$. For the integral of the right-hand side, we have by \eqref{eq:estim_dfj}
\begin{align*}
 (1 + \eps \kappa_{1,1}) \intrd F_1'(\bar\rho_1)\,\intd x \leq (1 + \eps \kappa_{1,1})\,\alpha_1\, \left(1 + 2\me_\eps(\bar\rrho)\right),
\end{align*}
which is $\eps$-uniformly bounded for $\eps \in (0, \bar\eps]$. Hence by \eqref{eq:el_fj_lowerbd}, the integral of $(C_1 - K \ast \bar\rho_2)_+$ is $\eps$-uniformly bounded as well. With $\bar x := \argmin_x K \ast \bar\rho_2$, we have, using the bound $C_K$ on the Hessian of $K \ast \bar\rho_2$ and inequality \eqref{eq:estim_fj_int_eeps}, 
\begin{align*}
K \ast \bar\rho_2(x) \leq K \ast \bar\rho_2(\bar x) + \frac{C_K}{2} |x - \bar x|^2 \leq \intrd \bar\rho_1 K \ast \bar\rho_2\,\intd x + \frac{C_K}{2} |x - \bar x|^2 \leq \me_\eps(\bar\rrho) + \frac{C_K}{2} |x - \bar x|^2
\end{align*}
for every $x \in \Rd$. Hence we can estimate the integral of $(C_1 - K \ast \bar\rho_2)_+$ from below by
\begin{align*}
\intrd (C_1 - K \ast \bar\rho_2)_+ \,\intd x \geq \intrd (C_1 - \me_\eps(\bar\rrho) - C_K |x-\bar x|^2)_+\,\intd x.
\end{align*}
Since the left-hand side of this estimate is $\eps$-uniformly bounded, the same is true for the right-hand side, which implies that $C_1$ is $\eps$-uniformly bounded for $\eps \in (0, \bar\eps]$. Since it holds $\frac{\lambda}{2}|x-\bar x|^2 \leq K \ast \bar\rho_2 \leq C_1$ in $\supp\,\bar\rho_1$, this directly implies that the diameter of $\supp\,\bar\rho_1$ is $\eps$-uniformly bounded for $\eps \in (0,\bar\eps]$.  If $\bar\eps  < 1 / \max\{\kappa_{1,1}, \kappa_{2,2}\} $, implying that $1 - \eps \kappa_{1,1}$ is bounded away from 0 uniformly for $\eps \in (0, \bar\eps]$, the uniform boundedness of $C_1$ and non-negativity of $K \ast \bar\rho_2$ together with \eqref{eq:el_fj_lowerbd} imply uniform boundedness of $F_1'(\bar\rho_1)$, and thus also for $\bar\rho_1$ by the growth of $F_1'$. The proof for the other component is analogous.
\end{proof}

In the next lemma, we prove an $\eps$-uniform regularity result for the minimizer $\bar\rrho$.
\begin{lemma} \label{lem:minim_fj_lip}
Assuming that $\bar\eps > 0$ is sufficiently small, the functions $F_j'(\bar\rho_j)$ are Lipschitz in $\Rd$, and the gradients $\nabla F_j'(\bar\rho_j)$ are Lipschitz in the common support $\{\bar\rho_1 > 0\} \cap \{\bar\rho_2 > 0\}$, both with Lipschitz constants independent of $\eps \in (0, \bar\eps]$.
\end{lemma}
\begin{proof}
We consider the function $\Gamma_\eps: \R^2 \to \R^2$, defined by
\begin{equation*}
\Gamma_\eps(\uu) := \left\{ \begin{array}{ll}
\uu + \eps \begin{pmatrix} \theta_1(\uu) \\ \theta_2(\uu)
\end{pmatrix} = \nabla \mf_\eps\left(\rr\right) & \text{ for } \uu \in \Rnn^2 \\
\uu & \text{ else} 
\end{array}\right.,
\end{equation*}
where $\rr =  ((F_1')^{-1}(u_1), (F_2')^{-1}(u_2))$ for $\uu \in \Rnn^2$. It follows from the form \eqref{eq:el_theta} of the Euler-Lagrange system that
\begin{equation} \label{eq:el_gamma}
\Gamma_\eps\left(F_1'(\bar\rho_1), F_2'(\bar\rho_2)\right) = \begin{pmatrix}
(C_1 - K \ast \bar\rho_2)_+ \\ (C_2 - K \ast \bar\rho_1)_+
\end{pmatrix}.
\end{equation}
By strict convexity of $\mf_\eps$, the map $\Gamma_\eps$ is injective. Since both $\theta_j$ vanish on $\partial \Rnn^2$, we have $\Gamma_\eps = \mathrm{id}$ there, and $\Gamma_\eps$ is continuous in $\R^2$. To prove that $\Gamma_\eps$ maps $\Rp^2$ to itself surjectively, take any $\vv = (v_1, v_2) \in \Rp^2$ and consider the minimizer $\bar\rr$ of the function $\rr \mapsto \mf_\eps(\rr) - \vv \cdot \rr$ over $\Rnn^2$. The minimizer exists by the superlinear growth of $\mf_\eps$, and it does not lie on the boundary $\partial \Rnn^2$, as for any $\rr$ with $r_i = 0$, it holds $\partial_{r_i} (\mf_\eps(\rr) - \vv \cdot \rr) = -v_i < 0$. We thus have $\bar\rr \in \Rp^2$, and thus by the necessary first-order minimality condition, it holds $\nabla \mf_\eps(\bar\rr) = \vv$, implying that $\Gamma_\eps(F_1'(\bar r_1), F_2'(\bar r_2)) = \vv$. This shows that $\Gamma_\eps$ maps $\R^2$, $\Rnn^2$ and $\Rp^2$ onto themselves bijectively, and we can rewrite \eqref{eq:el_gamma} as
\begin{equation} \label{eq:el_gammainv}
\begin{pmatrix}
F_1'(\bar\rho_1) \\ F_2'(\bar\rho_2)
\end{pmatrix} = \Gamma_\eps^{-1} \begin{pmatrix}
(C_1 - K \ast \bar\rho_2)_+ \\ (C_2 - K \ast \bar\rho_1)_+
\end{pmatrix}.
\end{equation}
We claim that for $\bar\eps$ sufficiently small, the inverse mapping $\Gamma_\eps^{-1}$ is of class $C^1$, and its Jacobian $\dff \Gamma_\eps^{-1}$ is globally bounded and locally Lipschitz in $\R^2$, uniformly with respect to $\eps \in (0, \bar\eps]$. We prove the claim by the inverse function theorem. In $\Rp^2$, we have with \eqref{eq:thetabdd_swap}
\begin{align*}
\det \dff \Gamma_\eps(\uu) = \det\left[ \mathrm{Id} + \eps \begin{pmatrix}
\theta_{1,1}(\uu) & \theta_{1,2}(\uu) \\ \theta_{2,1}(\uu) & \theta_{2,2}(\uu)
\end{pmatrix} \right] \geq 1 - \eps \left(\kappa_{1,1} + \kappa_{2,2}\right) - \eps^2 \left(\kappa_{1,1}\kappa_{2,2} + \kappa_{1,2}\kappa_{2,1}\right),
\end{align*}
which is $\eps$-uniformly bounded away from 0 for $\eps \in (0,\bar\eps]$, if $\bar\eps > 0$ is sufficiently small.
In $\R^2 \setminus \Rp^2$, we have $\dff \Gamma_\eps = \mathrm{Id}$. Since the $\theta_{j,i}$ vanish on $\partial \Rnn^2$ by hypothesis \eqref{eq:thetabdd_swap}, and since they are locally Lipschitz, $\dff \Gamma_\eps$ is locally Lipschitz on $\R^2$, uniformly for $\eps \in (0, \bar\eps]$. Hence the inverse function theorem is applicable and yields $C^1$-regularity of the inverse $\Gamma_\eps^{-1}$ on $\R^2$, with Jacobian $\dff \Gamma_\eps^{-1} = \left(\dff \Gamma_\eps \circ \Gamma_\eps^{-1}\right)^{-1}$. Since $\det \dff \Gamma_\eps$ is $\eps$-uniformly bounded away from 0 on $\R^2$, this implies that $D\Gamma_\eps^{-1}$ is $\eps$-uniformly bounded in $\R^2$, implying that $\Gamma_\eps^{-1}$ is globally Lipschitz on $\R^2$ with an $\eps$-independent Lipschitz constant for all $\eps \in (0, \bar\eps]$. Since the inverse mapping $A \mapsto A^{-1}$ is Lipschitz on any set on $2\times2$-matrices $A$ with $\det A$ bounded away from 0, this also implies that $\dff\Gamma_\eps^{-1}$ is locally Lipschitz on $\R^2$, uniformly in $\eps$, since it is a composition of locally Lipschitz functions. This finishes the proof of the claim.

By Lemma \ref{lem:minim_bdd}, the diameter of $\supp\,\bar\rho_j$ is $\eps$-uniformly bounded. Together with the bound $|\nabla K(z)| \leq C_K |z|$, this implies that $(C_1 - K \ast \bar\rho_2)_+$ and $(C_2 - K \ast \bar\rho_1)_+$ are globally Lipschitz in $\R^d$, uniformly in $\eps$. It thus follows from \eqref{eq:el_gammainv} and the $\eps$-uniform Lipschitz continuity of $\Gamma_\eps^{-1}$ that the $F_j'(\bar\rho_j)$ are $\eps$-uniformly Lipschitz in $\R^d$. On the set $\{\bar\rho_1 > 0\} \cap \{\bar\rho_2 > 0\}$, it holds $(C_1 - K \ast \bar\rho_2)_+ = C_1 - K \ast \bar\rho_2$ and $(C_2 - K \ast \bar\rho_1)_+ = C_2 - K \ast \bar\rho_1$ by Lemma \ref{lem:el_rhogr0}, which are of class $C^2$ by the regularity of $K$. Thus, their gradients are $\eps$-uniformly Lipschitz in $\{\bar\rho_1 > 0\} \cap \{\bar\rho_2 > 0\}$, and \eqref{eq:el_gammainv} together with the local Lipschitz-continuity of $\dff\Gamma_\eps^{-1}$ implies $\eps$-uniform Lipschitz-continuity of $\nabla F_j'(\bar\rho_j)$ in the common support of $\bar\rho_1$ and $\bar\rho_2$. Note that local Lipschitz-continuity of $\dff\Gamma_\eps^{-1}$ is sufficient, since the constants $C_1, C_2$ are $\eps$-uniformly bounded by Lemma \ref{lem:minim_bdd} and thus $\dff \Gamma_\eps^{-1}$ is $\eps$-uniformly Lipschitz on the compact set $[0,C_1] \times [0,C_2] \subset \R^2$. This finishes the proof.
\end{proof}
An important consequence of this lemma is the following semiconvexity result, which we shall use in section \ref{sec:exp_cvgce} to prove exponential convergence to the equilibrium for system \eqref{eq:system-intro}.
\begin{corollary} \label{cor:vj_semiconv}
For the same $\bar\eps > 0$ as in the previous lemma, there exists an $\eps$-independent constant $K_0 > 0$ such that $V_j := \partial_{r_j} h(\bar\rho_1, \bar\rho_2)$ is $(-K_0)$-semiconvex in $\Rd$ for $j = 1,2$ and every $\eps \in (0, \bar\eps]$.
\end{corollary}
\begin{proof}
We show that for some $\eps$-independent constant $K_0$, the gradient $\nabla V_j$ is $K_0$-Lipschitz in $\Rd$ for $j = 1,2$ and $\eps \in (0, \bar\eps]$. With $\bar\uu := (F_1'(\bar\rho_1), F_2'(\bar\rho_2))$, it holds
\begin{equation} \label{eq:gradvj_theta}
\nabla V_j = \nabla \theta_j(\bar\uu) =
\theta_{j,1}(\bar\uu) \nabla F_1'(\bar\rho_1) + \theta_{j,2}(\bar\uu) \nabla F_2'(\bar\rho_2),
\end{equation}
which is nonzero only in $\{\bar\rho_1 > 0\} \cap \{\bar\rho_2 > 0\}$, since the $\theta_{j,i}$ vanish on $\partial \Rnn^2$ by assumption \eqref{eq:thetabdd_swap}. Inside the common support of $\bar\rho_1$ and $\bar\rho_2$, the $\nabla F_j'(\bar\rho_j)$ are $\eps$-uniformly Lipschitz by Lemma \ref{lem:minim_fj_lip}. Since by Lemma \ref{lem:minim_bdd}, the $\bar\rho_j$ are $\eps$-uniformly bounded by a constant $H_0$, the function $\bar\uu$ only attains values in the compact set $[0,F_1'(H_0)] \times [0, F_2'(H_0)]$ for all $\eps \in (0, \bar\eps]$. By Hypothesis \ref{hyp:theta}, the $\theta_{j,i}$ are Lipschitz on this set, and $\bar\uu$ is globally Lipschitz by Lemma \ref{lem:minim_fj_lip}, uniformly in $\eps$, hence $\theta_{j,i}(\bar\uu)$ is $\eps$-uniformly Lipschitz in $\Rd$. This proves that the expression \eqref{eq:gradvj_theta} is Lipschitz in $\{\bar\rho_1 > 0\} \cap \{\bar\rho_2 > 0\}$. Additionally, it is continuous across the boundary of the common support. To see this, take any $x^*$ on the boundary, and any sequence $x_n \to x^*$ with $x_n \in \{\bar\rho_1 > 0\} \cap \{\bar\rho_2 > 0\}$. Since $x^*$ lies on the boundary and both $F_j'(\bar\rho_j)$ are continuous in $\Rd$, at least one of the two sequences $\bar u_k(x_n) = F_k'(\bar\rho_k(x_n))$ for $k = 1,2$ converges to 0, implying by \eqref{eq:thetabdd_swap} that $\theta_{j,i}(\bar\uu(x_n)) \to 0$ for both $i$. Since the $\nabla F_j'$-terms are bounded along the sequence $(x_n)$, this implies that the expression \eqref{eq:gradvj_theta} evaluated at $x_n$ converges to 0, yielding continuity across the boundary. This proves that the $\nabla V_j$ are $\eps$-uniformly Lipschitz in $\Rd$ with some Lipschitz constant $K_0$, yielding $(-K_0)$-semiconvexity on $V_j$ for every $\eps \in (0, \bar\eps]$.
\end{proof}

\section{Minimizing movement scheme for the energy functional} \label{sec:min_move_scheme}

Fixing any $\tau > 0$, we introduce the functional $\me_{\eps, \tau}$, defined for $\rrho, \hat{\rrho} \in \mptrd^2$ as
\begin{equation*}
\eepstau(\rrho | \hat\rrho) := \frac{1}{2\tau} \dst(\rrho, \hat\rrho)^2 + \me_\eps(\rrho).
\end{equation*}
Our goal is to get a time-discrete variational approximation for a weak solution of system \eqref{eq:system-intro} by iteratively minimizing the functional $\eepstau(\cdot | \hat\rrho)$ over $\mptrd^2$, with $\hat\rrho$ being the minimizing pair obtained in the previous iteration.

\subsection{Existence of minimizers for the modified functional}
We start by proving existence of global minimizers of the functional $\eepstau(\cdot | \hat\rrho)$.

\begin{lemma} \label{lem:yosida_min_exist}
For every pair $\hat\rrho \in \mptrd^2$, the functional $\mptrd^2 \ni \rrho \mapsto \eepstau(\rrho | \hat\rrho)$ possesses a global minimizer $\rrho^+$. Every minimizer $\rrho^+$ satisfies the inequality
\begin{equation} \label{eq:step_estimate}
\frac{1}{2\tau} \dst(\rrho^+, \hat\rrho)^2 \leq \me_\eps(\hat\rrho) - \me_\eps(\rrho^+).
\end{equation}
If $\hat\rrho \in \sspc$, then also $\rrho^+ \in \sspc$ for every global minimizer $\rrho^+$.
\end{lemma}
\begin{proof}
We first claim that for any $\hat\rrho \in \mptrd^2$, the minimization problem of $\eepstau(\cdot|\hat\rrho)$ over $\mptrd^2$ is equivalent to the restricted minimization problem over the set of all $\rrho \in \mptrd^2$ which have the same combined center of mass as $\hat\rrho$, i.e. $\mom_1[\rho_1] + \mom_1[\rho_2] = \mom_1[\hat\rho_1] + \mom_1[\hat\rho_2]$.

To prove the claim, fix any $\rrho \in \mptrd^2$ with $\me_\eps(\rrho) < +\infty$, so in particular, the $\rho_j$ for $j = 1,2$ are absolutely continuous with respect to the Lebesgue measure. Denote by $\Phi_j: \Rd \to \R$ convex functions such that $\nabla \Phi_j: \Rd \to \Rd$ are optimal transport maps from $\rho_j$ to $\hat\rho_j$, which exist by Brenier's theorem. Now consider the pair of translated densities $\sigma_v\rrho = \left(\rho_1(\cdot + v), \rho_2(\cdot + v)\right)$, for any direction $v \in \Rd$. Clearly, the maps $\nabla \Phi_j(\cdot + v)$ push $\sigma_v \rho_j$ forward to $\hat\rho_j$, and since they are the gradients of the convex functions $\Phi_j(\cdot + v)$, they are optimal in the transport problem from $\sigma_v \rho_j$ to $\hat\rho_j$. This implies
\begin{align*}
\wass(\sigma_v \rho_j, \hat\rho_j)^2 &= \intrd |\nabla\Phi_j(x+v) - x|^2\, \rho_j(x+v)\, \intd x = \intrd |\nabla\Phi_j(y) - y + v|^2\, \rho_j(y)\, \intd y \\
&= \intrd |\nabla\Phi_j(y) - y|^2\, \rho_j(y)\, \intd y + |v|^2 + 2\, \left(\intrd \left(\nabla\Phi_j(y) - y\right)\,\rho_j(y) \,\intd y\right) \cdot v \\
&= \wass(\rho_j, \hat\rho_j)^2 + |v|^2 + 2\, \left( \mom_1[\hat\rho_j] - \mom_1[\rho_j] \right) \cdot v\,.
\end{align*}
Combining this fact with the translation invariance of the energy $\me_\eps$, we obtain
\begin{align*}
&\eepstau(\sigma_v\rrho|\hat\rrho) = \frac{\wass(\sigma_v \rho_1, \hat\rho_1)^2 + \wass(\sigma_v \rho_2, \hat\rho_2)^2}{2 \tau} + \me_\eps(\sigma_v\rrho) \\
&= \frac{\wass(\rho_1, \hat\rho_1)^2 + \wass(\rho_2, \hat\rho_2)^2 + 2|v|^2 + 2\,(\mom_1[\hat\rho_1] + \mom_1[\hat\rho_2] - \mom_1[\rho_1] - \mom_1[\rho_2])\cdot v}{2 \tau} + \me_\eps(\rrho) \\
&= \eepstau(\rrho|\hat\rrho) + \frac{ |v|^2 + \,(\mom_1[\hat\rho_1] + \mom_1[\hat\rho_2] - \mom_1[\rho_1] - \mom_1[\rho_2])\cdot v}{\tau}\,.
\end{align*}
This expression is quadratic in $v \in \Rd$, and is uniquely minimized by $\bar v$ given by
\begin{align*}
\bar v := -\frac{1}{2}\left(\mom_1[\hat\rho_1] + \mom_1[\hat\rho_2] - \mom_1[\rho_1] - \mom_1[\rho_2]\right).
\end{align*}
For the combined center of mass of this optimally shifted pair $\sigma_{\bar v}\rrho$, we obtain
\begin{align*}
\mom_1[\sigma_{\bar v}\rho_1] + \mom_1[\sigma_{\bar v}\rho_2] = \mom_1[\rho_1] + \mom_1[\rho_2] - 2\bar v = \mom_1[\hat\rho_1] + \mom_1[\hat\rho_2].
\end{align*}
This proves the claimed equivalence of the minimization problems over the whole $\mptrd^2$ and the restricted set of pairs which preserve $\hat\rrho$'s combined center of mass.

To show existence of a global minimizer, we can thus w.l.o.g. assume $\hat\rrho \in \sspc$ and prove existence of a minimizer $\rrho^+$ over $\sspc$. Similarly as in the previous section, we shall use the direct method from the calculus of variations. Since the distance term is non-negative, pre-compactness of the sublevel sets of $\eepstaurhat{\cdot}$ follows directly from Lemma \ref{lem:eeps_coer}. Weak lower semicontinuity follows from Lemma \ref{lem:eeps_lsc} and lower semi-continuity of the Wasserstein distance with respect to weak $L^1$-convergence, proving the existence of a minimizer $\rrho^+ \in \sspc$.

From the minimality of $\rrho^+$, we directly obtain the inequality
\begin{align*}
\frac{1}{2\tau} \dst(\rrho^+, \hat\rrho^+)^2 + \me_\eps(\rrho^+) = \eepstaurhat{\rrho^+} \leq \eepstaurhat{\hat\rrho} = \me_\eps(\hat\rrho),
\end{align*}
which implies \eqref{eq:step_estimate}, thus finishing the proof.
\end{proof}

As a first regularity result, we prove a global $L^2$-estimate for $\nabla F_j'(\rho_j^+)$, restricted to a set where $\rho_j^+$ is smaller than some fixed height $\alpha > 0$.
\begin{lemma} \label{lem:heat_comparison}
Let $\hat\rrho \in \mptrd^2$ be any fixed pair, and let $\alpha > 0$ be a constant. Denote by $\mh_c(\rrho) := \mh(\rho_1) + \mh(\rho_2) :=  \intrd [ \rho_1 \log \rho_1 + \rho_2 \log \rho_2]\,\intd x $ the combined (negative) entropy functional for $\rho \in \mptrd$, and let $[\rho]_\alpha := \min\{\rho, \alpha\}$ denote a density $\rho$ cut off at height $\alpha$. Then any minimizer $\rrho^+$ of the functional $\eepstaurhat{\cdot}$ satisfies
\begin{equation} \label{eq:h1_estimate_min}
\frac{1}{2A}\intrd \left[\left|\nabla F_1'\left([\rho_1^+]_\alpha\right)\right|^2 + \left|\nabla F_2'\left([\rho_2^+]_\alpha\right)\right|^2 \right]\,\intd x \leq \frac{\mh_c(\hat\rrho) - \mh_c(\rrho^+)}{\tau} + 2 d C_K\,,
\end{equation}
where $A := \max_{j = 1,2} \max_{r \leq \alpha} F_j''(r) < + \infty$.
\end{lemma}
\begin{proof}
Let $\rrho^+$ be any minimizer of $\eepstaurhat{\cdot}$ over $\mptrd^2$. We prove \eqref{eq:h1_estimate_min} by analyzing the behaviour of the functional $\eepstaurhat{\cdot}$ along the heat flow with initial datum $\rrho^+$: Denote by $\mk^s$ the heat kernel in $\Rd$, and let $\rho_j^s$ be the solution of the heat equation $\partial_s \rho_j^s = \Delta \rho_j^s$ for $(s,x) \in [0, \infty) \times \Rd$ with initial condition $\rho_j^0 = \rho_j^+$ for $j = 1,2$. Explicitly, we have
\begin{align*}
\mk^s(x) = \frac{1}{(4\pi s)^{d/2}} e^{-\frac{|x|^2}{4s}}, \qquad \rho_j^s = \mk^s \ast \rho_j^+.
\end{align*}
Clearly, $\rrho^s = (\rho_1^s, \rho_2^s) \in \mptrd$ is again a feasible pair of probability distributions, hence by minimality of $\rrho^+$, it holds
\begin{equation} \label{eq:liminf_eepstau_split}
0 \leq \liminf_{s \to 0^+} \frac{\eepstaurhat{\rrho^s} - \eepstaurhat{\rrho^+}}{s} \leq \limsup_{s \to 0^+} \frac{\dst(\rrho^s, \hat\rrho)^2 - \dst(\rrho^+, \hat\rrho)^2}{2\tau s} + \liminf_{s \to 0^+} \frac{\me_\eps(\rrho^s) - \me_\eps(\rrho^+)}{s}.
\end{equation}
It is well known that the heat flow is the metric gradient flow in the Wasserstein space of the (negative) entropy functional $\mh(\rho) = \intrd \rho \log \rho\,\intd x$, which is convex along generalized geodesics. By \cite[Theorem 4.0.4]{AGS}, it thus satisfies the evolutional variational inequality
\begin{equation*}
\frac{1}{2}\limsup_{s \to 0^+}\, \frac{\wass(\rho_j^s, \hat\rho_j)^2 - \wass(\rho_j^+, \hat\rho_j)^2}{s} \leq \mh(\hat\rho_j) - \mh(\rho_j^+).
\end{equation*}
By combining the two components $j = 1,2$ and dividing by $\tau$, this yields
\begin{equation} \label{eq:evi_heat}
\limsup_{s \to 0^+}\, \frac{\dst(\rrho^s, \hat\rrho)^2 - \dst(\rrho^+, \hat\rrho)^2}{2 \tau s} \leq \frac{\mh_c(\hat\rrho) - \mh_c(\rrho^+)}{\tau}.
\end{equation}
We now analyze the change in the energy term. We split it in the following way:
\begin{align*}
\me_\eps(\rrho^s) - \me_\eps(\rrho^+) = \intrd [\mf_\eps(\rrho^s) - \mf_\eps(\rrho^+)]\,\intd x + \intrd [\rho_1^s K \ast \rho_2^s - \rho_1^+ K \ast \rho_2^+]\,\intd x.
\end{align*}
For the interaction term, we use the representation $\rho_j^s = \mk^s\ast \rho_j^+$ together with the symmetry and semigroup property of the heat kernel $\mk^s$ and the properties of convolutions to obtain
\begin{align*}
&\intrd \left[\rho_1^s K \ast \rho_2^s - \rho_1^+ K \ast \rho_2^+ \right]\,\intd x = \intrd [(\mk^s \ast \rho_1^+) K \ast (\mk^s \ast \rho_2^+) - \rho_1^+ K \ast \rho_2^+]\,\intd x  \\
&= \intrd \left[\rho_1^+ K \ast \mk^s \ast \mk^s \ast \rho_2^+ - \rho_1^+ K \ast \rho_2^+\right]\,\intd x = \intrd \rho_1^+ [K \ast \mk^{2s} - K] \ast \rho_2^+\,\intd x.
\end{align*}
Hence by dominated convergence and the equation $\partial_s [K \ast \mk^s] = \Delta [K \ast \mk^s]$, it holds
\begin{equation} \label{eq:k_heat_diff}
\begin{split}
\lim_{s \to 0} \frac{1}{s}\intrd \left[\rho_1^s K \ast \rho_2^s - \rho_1^+ K \ast \rho_2^+ \right]\,\intd x &= \intrd \rho_1^+ \left[\lim_{s\to 0} \frac{K\ast \mk^{2s} - K}{s} \right] \ast  \rho_2^+\,\intd x \\
& = 2\intrd \rho_1^+ \Delta K \ast \rho_2^+\,\intd x \leq 2dC_K.
\end{split}
\end{equation}
The dominated convergence is justified, since from the assumption that $K$ is convex with $\|\nabla^2 K\| \leq C_K$ and thus $0 \leq \Delta K \leq dC_K$ pointwise in $\Rd$, it follows $0 \leq \Delta[K \ast \mk^s] \leq dC_K$ pointwise in $\Rd$ for every $s \geq 0$, which implies that the difference quotients are globally bounded by a constant.

For the $\mf_\eps$-term, we use that $\mf_{2\eps}$ is convex by Proposition \ref{prop:asmpt_corollaries}, and thus $s \mapsto \intrd \mf_{2\eps}(\rrho^s)\,\intd x$ is non-increasing along the heat flow $\rrho^s$: By Jensen and Fubini, we have
\begin{equation} \label{eq:heat_conv_nonincr}
\begin{split}
\intrd \mf_{2\eps}(\rrho^s)\,\intd x &= \intrd \mf_{2\eps}(\mk^s \ast \rrho^+)\,\intd x = \intrd \mf_{2\eps}\left(\intrd \rrho^+(y) \mk^s(x-y)\,\intd y\right)\,\intd x \\
&\leq \iintrdrd \mf_{2\eps}(\rrho^+(y)) \mk^s(x-y) \,\intd x \intd y = \intrd \mf_{2\eps}(\rrho^+)\,\intd y.
\end{split}
\end{equation}
Together with the identity $\mf_\eps(\rr) = \frac{1}{2}\mf_{2\eps}(\rr) + \frac{1}{2} F_1(r_1) + \frac{1}{2} F_2(r_2)$, this implies
\begin{align*}
\intrd [\mf_\eps(\rrho^s) - \mf_\eps(\rrho^+)]\,\intd x \leq \frac{1}{2} \intrd [F_1(\rho_1^s) - F_1(\rho_1^+) ]\,\intd x + \frac{1}{2}\intrd[ F_2(\rho_2^s) - F_2(\rho_2^+)]\,\intd x.
\end{align*}
We estimate the first integral (the second one is done in the same way) by replacing $\rho_1^+$ with $\rho_1^{s_0}$ for some arbitrarily small $s_0 > 0$ and then using a limiting argument. We use the fact that the $\rho_j^\sigma$ solve the heat equation and Fubini to obtain for any $s_0 > 0$:
\begin{equation} \label{eq:fubini_f_heat}
\frac{1}{2} \intrd [F_1(\rho_1^s) - F_1(\rho_1^{s_0}) ]\,\intd x = \frac{1}{2} \intrd \int_{s_0}^s F_1'(\rho_1^\sigma)\Delta\rho_1^\sigma \, \intd \sigma \intd x = \frac{1}{2}\int_{s_0}^s \intrd F_1'(\rho_1^\sigma)\Delta\rho_1^\sigma \,\intd x \intd \sigma.
\end{equation}
To justify the Fubini, it suffices to show that $F_1'(\rho_1^\sigma)$ is uniformly bounded in $L^\infty(\Rd)$, and that $\Delta\rho_1^\sigma$ is uniformly bounded in $L^1(\Rd)$ for $\sigma \in [s_0, s]$. The first claim follows from global uniform boundedness of $\rho_1^\sigma = \mk^\sigma \ast \rho_1^+$, which holds since $\mk^\sigma(x) \leq \mk^{s_0}(0) < +\infty$ for all $x \in \Rd$ and all $\sigma \in [s_0, s]$. The second claim follows in a similar way from the fact that $\Delta \mk^\sigma$ is uniformly bounded in $L^1(\Rd)$ for $\sigma \in [s_0, s]$.

By integrating \eqref{eq:fubini_f_heat} by parts, using that $\nabla [\rho_1^\sigma]_\alpha = \nabla \rho_1^\sigma$ in $\{\rho_1^\sigma \leq \alpha\}$ and $0$ in $\{\rho_1^\sigma > \alpha\}$ for every $\sigma \geq s_0$ by regularity, and that $0 \leq F_1''([\rho_1^\sigma]_\alpha) \leq A$ in $\Rd$, we obtain
\begin{equation} \label{eq:s0_heat_estim}
\begin{split}
&\frac{1}{2} \intrd [F_1(\rho_1^s) - F_1(\rho_1^{s_0}) ]\,\intd x = \frac{1}{2}\int_{s_0}^s \intrd F_1'(\rho_1^\sigma)\Delta\rho_1^\sigma \,\intd x \intd \sigma\\
&= -\frac{1}{2} \int_{s_0}^s \intrd \nabla[ F_1'(\rho_1^\sigma)] \cdot \nabla \rho_1^\sigma \,\intd x \intd \sigma
= -\frac{1}{2}  \int_{s_0}^s \intrd F_1''(\rho_1^\sigma) \left|\nabla \rho_1^\sigma\right|^2\,\intd x\intd \sigma \\&\leq -\frac{1}{2A}  \int_{s_0}^s \intrd F_1''\left([\rho_1^\sigma]_\alpha\right)^2 \left|\nabla [\rho_1^\sigma]_\alpha\right|^2\,\intd x \intd \sigma
= -\frac{1}{2A} \int_{s_0}^s \intrd \left| \nabla F_1'([\rho_1^\sigma]_\alpha)\right|^2\,\intd x\intd \sigma.
\end{split}
\end{equation}
We shall now pass to the limit $s_0 \to 0$ in this estimate. On the right-hand side, the integrand is non-negative, hence by monotone convergence, it holds
\begin{equation*}
-\frac{1}{2A} \int_{s_0}^s \intrd \left| \nabla F_1'([\rho_1^\sigma]_\alpha)\right|^2\,\intd x\intd \sigma \to -\frac{1}{2A} \int_0^s \intrd \left| \nabla F_1'([\rho_1^\sigma]_\alpha)\right|^2\,\intd x\intd \sigma.
\end{equation*}
On the left-hand side, we observe that by convexity of $F_1$, the map $s \mapsto \intrd F_1(\rho_1^s)\,\intd x$ is non-increasing, which follows from Jensen's inequality as in \eqref{eq:heat_conv_nonincr}. In particular, 
\begin{equation*}
\limsup_{s_0 \to 0} \intrd F_1(\rho_1^{s_0})\,\intd x \leq \intrd F_1(\rho_1^+)\,\intd x.
\end{equation*}
Passing to the limit $s_0 \to 0$ in \eqref{eq:s0_heat_estim} thus yields the estimate
\begin{align*}
\frac{1}{2} \intrd [F_1(\rho_1^s) - F_1(\rho_1^{*}) ]\,\intd x \leq -\frac{1}{2A} \int_0^s \intrd \left| \nabla F_1'([\rho_1^\sigma]_\alpha)\right|^2\,\intd x\intd \sigma.
\end{align*}
By adding the term for the other component and dividing by $-s < 0$, this yields
\begin{equation} \label{eq:f_heat_diff}
\frac{1}{2As} \int_0^s \intrd \left[ \left| \nabla F_1'([\rho_1^\sigma]_\alpha)\right|^2 + \left| \nabla F_2'([\rho_2^\sigma]_\alpha)\right|^2 \right] \,\intd x\intd \sigma \leq - \frac{1}{s} \intrd [\mf_\eps(\rrho^s) - \mf_\eps(\rrho^+)]\,\intd x 
\end{equation}
By inserting \eqref{eq:evi_heat} and \eqref{eq:k_heat_diff} into \eqref{eq:liminf_eepstau_split}, we can control the right hand side of \eqref{eq:f_heat_diff} from above as $s \to 0$:
\begin{equation} \label{eq:limsup_control}
\begin{split}
\limsup_{s \to 0^+} \left[- \frac{1}{s} \intrd [\mf_\eps(\rrho^s) - \mf_\eps(\rrho^+)]\,\intd x \right] &= - \liminf_{s \to 0^+} \frac{1}{s} \intrd [\mf_\eps(\rrho^s) - \mf_\eps(\rrho^+)]\,\intd x \\
&\leq \frac{\mh_c(\hat\rrho) - \mh_c(\rrho^+)}{\tau} + 2dC_K
\end{split}
\end{equation}
By \eqref{eq:f_heat_diff}, this implies that $\nabla F_j'([\rho_j^s]_\alpha)$ is bounded in $L^2(\Rd;\Rd)$, uniformly along a subsequence $s \to 0$. Hence by Alaoglu's theorem, there exist $v_j \in L^2(\Rd;\Rd)$ such that $\nabla F_j'([\rho_j^s]_\alpha) \weakto v_j$ weakly in $L^2$ globally up to subsequence.  Additionally, since $F_j'([\rho_j^s]_\alpha)$ is uniformly bounded in $L^\infty$ by $F_j'(\alpha)$, we get that $F_j'([\rho_j^s]_\alpha)$ is locally bounded in $H^1$, uniformly in $s$. Applying again Alaoglu's theorem gives  $u_j \in H^1_{loc}(\Rd)$ such that up to subsequence, $F_j'([\rho_j^s]_\alpha) \rightharpoonup u_j$  weakly in $H^1_{loc}$ as $s \to 0$. In particular, $\nabla F_j'([\rho_j^s]_\alpha) \rightharpoonup \nabla u_j$ weakly in $L^2_{loc}$, thus we can identify $v_j = \nabla u_j$. Also, we have $F_j'([\rho_j^s]_\alpha) \to u_j$ strongly in $L^2$ locally, which implies pointwise a.e. convergence to $u_j$ along a further subsequence. It thus holds $u_j = F_j'([\rho_j^+]_\alpha)$, as $\rho_j^s \to \rho_j^+$ a.e. pointwise. This identification shows that the whole sequence $\nabla F_j'([\rho_j^s]_\alpha)$ converges weakly in $L^2(\Rd)$. All together, we have shown $\nabla F_j'([\rho_j^s]_\alpha) \weakto \nabla F_j'([\rho_j^+]_\alpha)$ weakly in $L^2(\Rd)$ as $s \to 0$.

 By lower semicontinuity of the $L^2$-norm with respect to weak convergence, this implies
\begin{align*}
&\frac{1}{2A}\intrd \left[\left|\nabla F_1'\left([\rho_1^+]_\alpha\right)\right|^2 + \left|\nabla F_2'\left([\rho_2^+]_\alpha\right)\right|^2 \right]\,\intd x \leq \liminf_{s \to 0^+} \frac{1}{2A}\intrd \left[\left|\nabla F_1'\left([\rho_1^s]_\alpha\right)\right|^2 + \left|\nabla F_2'\left([\rho_2^s]_\alpha\right)\right|^2 \right]\,\intd x \\ 
&\leq \limsup_{s \to 0^+} \frac{1}{2As} \int_0^s \intrd \left[ \left| \nabla F_1'([\rho_1^\sigma]_\alpha)\right|^2 + \left| \nabla F_2'([\rho_2^\sigma]_\alpha)\right|^2 \right] \,\intd x\intd \sigma \leq \frac{\mh_c(\hat\rrho) - \mh_c(\rrho^+)}{\tau} + 2dC_K,
\end{align*}
where the last estimate follows from \eqref{eq:f_heat_diff} and \eqref{eq:limsup_control}. This finishes the proof of \eqref{eq:h1_estimate_min}.
\end{proof}

\subsection{Euler-Lagrange system for the minimizing movement step}
As a next step, we derive an Euler-Lagrange system that minimizers $\rrho^+$ satisfy. For simplicity, we prove it under additional assumptions on $\hat\rrho$, and later use approximation arguments to generalize our results.
\begin{lemma} \label{lem:el_yosida}  Assume in addition that both components of $\hat\rrho$ are absolutely continuous with respect to the Lebesgue measure, and that there exist positive radii $R_j > 0$ such that $\hat\rho_j > 0$ a.e. in $B_{R_j}$ and $\hat\rho_j = 0$ a.e. outside of $B_{R_j}$. Fix any minimizer $\rrho^+$ of $\eepstau(\cdot|\hat\rrho)$ over $\mptrd^2$ obtained from Lemma \ref{lem:yosida_min_exist}. Then there exist constants $C_1, C_2$ and optimal pairs $(\varphi_1, \psi_1)$ and $(\varphi_2, \psi_2)$ of $c$-conjugate Kantorovich potentials for the (dual) transport problems of $\rho_1^+$ to $\hat\rho_1$ and of $\rho_2^+$ to $\hat\rho_2$, respectively, such that $\rrho^+$ satisfies the Euler-Lagrange system
\begin{equation} \label{eq:el_yosida}
\begin{split}
F_1'(\rho_1^+) + \eps \partial_{r_1} h(\rho_1^+, \rho_2^+) &= (C_1-\frac{\varphi_1}{\tau} - K \ast \rho_2^+)_+\\
F_2'(\rho_2^+) + \eps \partial_{r_2} h(\rho_1^+, \rho_2^+) &= (C_2-\frac{\varphi_2}{\tau} - K \ast \rho_1^+)_+  \,.
\end{split}
\end{equation}
 
The optimal maps $T_1, T_2: \R^d \to \R^d$, transporting $\rho_1^+$ to $\hat\rho_1$ and $\rho_2^+$ to $\hat\rho_2$, respectively, satisfy
\begin{equation} \label{eq:el_map_yosida}
\begin{split}
\frac{T_1(x) - x}{\tau} &= \nabla\left[F_1'(\rho_1^+) + \eps \partial_{r_1} h(\rho_1^+, \rho_2^+) + K \ast \rho_2^+ \right] \qquad \ml^d\text{-a.e. in } \{\rho_1^+ > 0 \} \\
\frac{T_2(x) - x}{\tau} &= \nabla\left[F_2'(\rho_2^+) + \eps \partial_{r_2} h(\rho_1^+, \rho_2^+) + K \ast \rho_1^+ \right] \qquad \ml^d\text{-a.e. in } \{\rho_2^+ > 0 \}.
\end{split}
\end{equation}
\end{lemma}
\begin{proof}
Take $\hat\rrho \in \sspc$ and $\rrho^+$ as above. We only need to show the first equation in \eqref{eq:el_yosida} and \eqref{eq:el_map_yosida}, since the other one is analogous in both cases.

Fix any $\eta_0 \in L^1(\Rd)$ which is compactly supported, bounded, satisfies $\intrd \eta_0\, \intd x = 0$ and $\rho_1^s := \rho_1^+ + s \eta_0 \geq 0$ for every sufficiently small $s > 0$, and there exists a constant $H > 0$ such that $\eta_0$ vanishes almost everywhere on $\{\rho_1^+ > H\}$, similarly as in the proof of Lemma \ref{lem:el_intineq}. Then for all sufficiently small $s > 0$, it holds $\rho_1^s \in \mptrd$, and thus by minimality of $\rho_1^+$ for the functional $\eepstau((\cdot, \rho_2^+)|\hat\rrho)$ over $\mptrd$, we have with $\rrho^s := (\rho_1^s, \rho_2^+)$:
\begin{align}
0 &\leq \liminf_{s \to 0^+}\, \frac{\eepstaurhat{\rrho^s} - \eepstaurhat{\rrho^+}}{s} = \frac{1}{2\tau} \liminf_{s \to 0^+} \,\frac{\wass(\rho_1^s, \hat\rho_1)^2 - \wass(\rho_1^+, \hat\rho_1)^2}{s} + \ddszerop \me_\eps(\rrho^s) \nonumber \\ \label{eq:deriv_eepstau}
&= \frac{1}{2\tau} \liminf_{s \to 0^+} \,\frac{\wass(\rho_1^s, \hat\rho_1)^2 - \wass(\rho_1^+, \hat\rho_1)^2}{s} + \intrd \left[ F_1'(\rho_1^+) + \eps \partial_{r_1}h(\rho_1^+, \rho_2^+) + K \ast \rho_2^+\right]\,\eta_0\,\intd x.
\end{align}
The last equality follows as in the proof of Lemma \ref{lem:el_intineq} from the dominated convergence theorem. Our goal is to construct a pair $(\varphi_1^0, \psi_1^0)$ of optimal Kantorovich potentials for the transport problem from $\rho_1^+$ to $\hat\rho_1$, such that
\begin{equation} \label{eq:deriv_wass}
\frac{1}{2} \liminf_{s \to 0^+} \,\frac{\wass(\rho_1^s, \hat\rho_1)^2 - \wass(\rho_1^+, \hat\rho_1)^2}{s} \leq \intrd \varphi_1^0\, \eta_0\,\intd x.
\end{equation}
Denote by $(\varphi_1^s, \psi_1^s)$ pairs of optimal Kantorovich potentials for the transport from $\rho_1^s$ to $\hat\rho_1$. We can fix some point $x_0 \in \supp\,\rho_1^+$ and without loss of generality assume $\varphi_1^s(x_0) = 0$ for every $s$. Since $\{\rho_1^+ > 0 \} \subset \{\rho_1^s > 0\}$ for every $s \in [0, \bar s]$ with some $\bar s > 0$, the pair $(\varphi_1^s, \psi_1^s)$ is admissible for the (dual) transport problem from $\rho_1^+$ to $\hat\rho_1$. Hence for every $s \in [0, \bar s]$
\begin{align*}
\frac{1}{2} \left[\wass(\rho_1^s, \hat\rho_1)^2 - \wass(\rho_1^+, \hat\rho_1)^2 \right] \leq  \intrd \left[ \varphi_1^s \, \rho_1^s + \psi_1^s\,\hat\rho_1\right]\,\intd x - \intrd \left[\varphi_1^s \, \rho_1^+ + \psi_1^s\,\hat\rho_1 \right]\,\intd x = s \intrd \varphi_1^s\,\eta_0\,\intd x.
\end{align*}
The proof of \eqref{eq:deriv_wass} thus reduces to the construction of an optimal potential $\varphi_1^0$ for the transport problem from $\rho_1^+$ to $\hat\rho_1$ such that
\begin{equation} \label{eq:pot_liminf}
\liminf_{s \to 0^+} \intrd \varphi_1^s\,\eta_0\,\intd x \leq \intrd \varphi_1^0\,\eta_0\,\intd x.
\end{equation}
Denote $\Phi_1^s := \frac{1}{2}|\cdot|^2 - \varphi_1^s$. Since $(\varphi_1^s, \psi_1^s)$ is a pair of optimal $c$-conjugate potentials, $\Phi_1^s$ is a convex function and satisfies $|\nabla \Phi_1^s(x)| = |T^s(x)|$ for almost every $x \in \{\rho_1^s > 0 \}$, where $T^s\#\rho_1^s = \hat\rho_1$ is the optimal transport map from $\rho_1^s$ to $\hat\rho_1$. Since by assumption, $\hat\rho_1$ vanishes a.e. outside of $B_{R_1}$, this implies $|\nabla \Phi_1^s| \leq R_1$ a.e. on $ \{\rho_1^s > 0 \}$. Since $\Phi_1^s$ is convex on $\Rd$, this implies existence of a convex function $\tilde \Phi_1^s: \Rd \to \R$ with $\tilde \Phi_1^s = \Phi_1^s$ in $\{\rho_1^s > 0\}$ and $|\nabla \tilde\Phi_1^s| \leq R_1$ in $\Rd$. Hence the family of potentials $\{\tilde\Phi_1^s\}_{s \in [0, \bar s]}$ is equi-Lipschitz in $\Rd$, and because $\tilde\Phi_1^s(x_0) = \Phi_1^s(x_0) = 0$ for all $s$, it is bounded pointwise. The Arzelà-Ascoli Theorem thus yields existence of a locally uniformly converging subsequence $(\tilde\Phi_1^{s_k})_k \to \Phi_1^0$ in $\Rd$. 

Clearly, $\varphi_1^0 := \frac{1}{2}|\cdot|^2 - \Phi_1^0$ satisfies \eqref{eq:pot_liminf}. It remains to show that $\varphi_1^0$, together with its $c$-conjugate $\psi_1^0$, forms an optimal pair for the transport from $\rho_1^+$ to $\hat\rho_1$. Since $\Phi_1^0$ is convex, it is sufficient to show $\nabla \Phi_1^0\#\rho_1^+ = \hat\rho_1$. This however follows from the fact that the $\tilde\Phi_1^s$ are convex and $\tilde\Phi_1^{s_k} \to \Phi_1^0$ locally uniformly, which implies $\nabla\tilde\Phi_1^{s_k} \to \nabla \Phi_1^0$ almost everywhere: For any $\zeta \in C^0_c(\Rd)$, it holds by dominated convergence
\begin{align*}
\intrd \zeta\, \nabla \Phi_1^0\#\rho_1^+\, \intd x = \intrd \zeta(\nabla \Phi_1^0(x))\, \rho_1^+(x)\,\intd x = \lim_{k} \intrd \zeta(\nabla \tilde\Phi_1^{s_k}(x))\, \rho_1^+(x)\,\intd x = \intrd \zeta\,\hat\rho_1\,\intd x.
\end{align*}
All together, we have proven the inequality \eqref{eq:deriv_wass} for the optimal potential $\varphi_1^0$.

We can insert this result into \eqref{eq:deriv_eepstau} to obtain
\begin{equation*}
0 \leq \intrd \left[ F_1'(\rho_1^+) + \eps \partial_{r_1}h(\rho_1^+, \rho_2^+) + K \ast \rho_2^+ + \frac{\varphi_1^0}{\tau}\right]\,\eta_0\,\intd x.
\end{equation*}
In order to obtain the Euler-Lagrange system \eqref{eq:el_yosida}, we have to show that for some optimal Kantorovich potential $\varphi_1$, this inequality holds true for all $\eta$ which satisfy our assumptions at once, and not just for the specific $\eta_0$ used above to construct the potentials $\varphi_1^0$ and $\psi_1^0$. We define, using the notations $\Phi_1^0 = \frac{1}{2}|\cdot|^2 - \varphi_1^0$ and $\Psi_1^0 = \frac{1}{2}|\cdot|^2 - \psi_1^0$, the functions
\begin{align*}
\Psi_1 := \left\{ \begin{array}{ll}
\Psi_1^0 & \text{ in } B_{R_1} \\
+\infty & \text{ else}
\end{array} \right.,\quad \Phi_1 := (\hat\Psi_1^0)^*,
\end{align*}
where $(\cdot)^*$ denotes the convex conjugate. Since $\hat\rho_1$ is supported in $B_{R_1}$, it follows that $\Psi_1$ is convex with $\nabla \Psi_1 \# \hat\rho_1 = \nabla \Psi_1^0 \# \hat\rho_1 = \rho_1^+$, implying that the pair of $c$-conjugate potentials $(\varphi_1, \psi_1) := ( \frac{1}{2}|\cdot|^2 - \Phi_1, \frac{1}{2}|\cdot|^2 - \Psi_1)$ is optimal. We claim that for all $\eta$ that have the properties of $\eta_0$ specified above, it holds
\begin{equation} \label{eq:el_intineq_general}
0 \leq \intrd \left[ F_1'(\rho_1^+) + \eps \partial_{r_1}h(\rho_1^+, \rho_2^+) + K \ast \rho_2^+ + \frac{\varphi_1}{\tau}\right]\,\eta\,\intd x.
\end{equation}
To prove this, take any such $\eta$. By the same construction as for $\eta_0$, there exists a pair of optimal potentials $(\tilde\varphi_1, \tilde\psi_1)$, such that \eqref{eq:el_intineq_general} holds with $\tilde \varphi_1$ instead of $\varphi_1$. Now, observe that $\tilde\Psi_1 := \frac{1}{2}|\cdot|^2 - \tilde\psi_1$ is a convex function whose gradient agrees with the gradient of $\Psi_1^0$ a.e. on the connected set $\{\hat\rho_1 > 0\} = B_{R_1}$, since by absolute continuity of $\hat\rho_1$ with respect to the Lebesgue measure, the optimal transport map from $\hat\rho_1$ to $\rho_1^+$ is unique and given by the gradient of any one of the convex potentials $\Psi_1^0$ and $\tilde \Psi_1$. Hence there exists a global constant $C$ such that $\Psi_1^0 = \tilde\Psi_1 + C$ in $B_{R_1}$. By definition of $\Psi_1$, this yields $\Psi_1 \geq \tilde\Psi_1 + C$ in $\Rd$ with equality in $B_{R_1}$. By the properties of the convex conjugate, this implies
\begin{align*}
\frac{1}{2} |\cdot|^2 - \varphi_1 = \Phi_1 = \Psi_1^* \leq \tilde\Psi_1^* - C = \frac{1}{2}|\cdot|^2 - \tilde\varphi_1 - C,
\end{align*}
or equivalently, $\varphi_1 \geq \tilde\varphi_1 + C$ in $\Rd$, with equality almost everywhere on the image of the map $\nabla \Psi_1$ over $B_{R_1}$, which is $\{\rho_1^+ > 0\}$ since $\nabla \Psi_1$ pushes $\hat\rho_1$ to $\rho_1^+$. Using that $\eta \geq 0$ in $\{\rho_1^+ = 0\}$ and $\intrd \eta = 0$, this gives
\begin{align*}
\intrd \varphi_1\eta\,\intd x = \int_{\{\rho_1^+ > 0\}} \left(\tilde\varphi_1 + C\right)\, \eta\,\intd x + \int_{\{\rho_1^+ = 0\}} \varphi_1 \eta\,\intd x \geq \intrd \left(\tilde\varphi_1 + C\right)\,\eta\,\intd x = \intrd \tilde\varphi_1 \eta\,\intd x,
\end{align*}
proving \eqref{eq:el_intineq_general}, since we have already shown that it holds with $\tilde \varphi_1$ instead of $\varphi_1$.

By the same argument as in the proof of Lemma \ref{lem:el_cineq}, \eqref{eq:el_intineq_general} yields the existence of some constant $C_1$, such that everywhere in $\Rd$, it holds
\begin{equation*}
F_1'(\rho_1^+) + \eps \partial_{r_1} h(\rho_1^+, \rho_2^+) + K \ast \rho_2^+ + \frac{\varphi_1}{\tau} \geq C_1,
\end{equation*}
with equality almost everywhere on $\{\rho_1^+ > 0\}$. The form \eqref{eq:el_yosida} then follows in the same way as Lemma \ref{lem:el_rhogr0} from the properties of $\mf_\eps$.
Hence the optimal potential $\varphi_1$ satisfies the first equation in \eqref{eq:el_yosida}. The equation for the other component is derived analogously.

Since the left-hand side of \eqref{eq:el_yosida} is strictly positive on $\{\rho_j^+ > 0\}$, so is the right-hand side. The Kantorovich potentials are differentiable almost  everywhere in $\{\rho_j^+ > 0\}$ with $-\nabla \varphi_j = T_j(x) - x$, hence the gradient of the right-hand side exists in $\{\rho_j^+ > 0\}$, and the equations \eqref{eq:el_map_yosida} follow from \eqref{eq:el_yosida} by taking the gradient on both sides.
\end{proof}

As a consequence of Lemma \ref{lem:el_yosida}, we prove that under the same assumptions, the transition from $\hat\rrho$ to $\rrho^+$ satisfies a time-discrete version of the weak formulation of the evolution equations \eqref{eq:system-intro}.
\begin{corollary} \label{cor:weaksol_conditions}
Let $\hat\rrho$ satisfy the assumptions of Lemma \ref{lem:el_yosida}. Then for all $\zeta \in C^\infty_c(\Rd)$, there hold the estimates
\begin{equation} \label{eq:discrete_weaksol_estim}
\begin{split}
\left| \intrd \frac{\rho_1^+ - \hat\rho_1}{\tau}\,\zeta\,\intd x + \intrd \rho_1^+\, \nabla\left[F_1'(\rho_1^+) + \eps\partial_{r_1}h(\rho_1^+, \rho_2^+) + K \ast \rho_2^+\right]\cdot\nabla \zeta\,\intd x\right| &\leq \frac{\wass(\rho_1^+, \hat\rho_1)^2}{2\tau}\, \|\zeta\|_{C^2}\\
\left| \intrd \frac{\rho_2^+ - \hat\rho_2}{\tau}\,\zeta\,\intd x + \intrd \rho_2^+\, \nabla\left[F_2'(\rho_2^+) + \eps\partial_{r_2}h(\rho_1^+, \rho_2^+) + K \ast \rho_1^+\right]\cdot\nabla \zeta\,\intd x\right| &\leq \frac{\wass(\rho_2^+, \hat\rho_2)^2}{2\tau}\, \|\zeta\|_{C^2}.
\end{split}
\end{equation}
Denoting by $\mr_1$ and $\mr_2$ the respective left-hand sides of these estimates, it holds
\begin{equation} \label{eq:weaksol_estim_error}
\mr_1 + \mr_2 \leq \left( \me_\eps(\hat\rrho) - \me_\eps(\rrho^+) \right)\, \|\zeta\|_{C^2}.
\end{equation}
\end{corollary}
\begin{proof}
Let $T_1 \# \rho_1^+ = \hat\rho_1$ be the optimal transport map, and let $\zeta\in C^\infty_c(\Rd)$ be arbitrary. We obtain for the difference quotient on the left-hand side
\begin{align*}
\intrd \frac{\rho_1^+ - \hat\rho_1}{\tau}\,\zeta\,\intd x &= \frac{1}{\tau}\left[\intrd \rho_1^+\,\zeta\,\intd x - \intrd T_1 \# \rho_1^+\,\zeta\,\intd x\right] = \intrd \rho_1^+(x)\, \frac{\zeta(x) - \zeta\left(T_1(x)\right)}{\tau}\,\intd x  \\
&= -\intrd \rho_1^+(x)\, \left[ \nabla\zeta(x)\cdot\frac{T_1(x) - x}{\tau} + \frac{(T_1(x) - x)\cdot \nabla^2 \zeta(\xi_x) (T_1(x) - x)}{2\tau} \right]\,\intd x
\end{align*}
for some $\xi_x \in [x, T_1(x)]$. Applying \eqref{eq:el_map_yosida} and $\wass(\rho_1^+, \hat\rho_1)^2= \intrd \rho_1^+(x)\,\left| T_1(x)-x\right|^2\,\intd x$ thus yields
\begin{align*}
&\left| \intrd \frac{\rho_1^+ - \hat\rho_1}{\tau}\,\zeta\,\intd x + \intrd \rho_1^+\, \nabla\left[F_1'(\rho_1^+) + \eps\partial_{r_1}h(\rho_1^+, \rho_2^+) + K \ast \rho_2^+\right]\cdot\nabla \zeta\,\intd x\right| \\
&= \left| \intrd \frac{\rho_1^+ - \hat\rho_1}{\tau}\,\zeta\,\intd x + \intrd \rho_1^+(x)\,\nabla \zeta(x)\cdot \frac{T_1(x) - x}{\tau} \,\intd x\right| \\
&\leq \intrd \rho_1^+(x)\, \left|\frac{(T_1(x) - x)\cdot \nabla^2 \zeta(\xi_x) (T_1(x) - x)}{2\tau}\right|\,\intd x \leq \frac{\wass(\rho_1^+, \hat\rho_1)^2}{2\tau}\, \|\zeta\|_{C^2}.
\end{align*}
This proves the first inequality in \eqref{eq:discrete_weaksol_estim}, the second one follows analogously. By adding the two inequalities and applying \eqref{eq:step_estimate}, we obtain \eqref{eq:weaksol_estim_error}.
\end{proof}
\subsection{Regularity estimates for the minimizers}
In order to prove convergence of the minimizing movement scheme to a weak solution of \eqref{eq:system-intro}, we need to prove the estimates from Corollary \ref{cor:weaksol_conditions} for more general $\hat\rrho$, which will be done using an approximating argument. For the convergence proof, we require further regularity results for the minimizers $\rrho^+$, in particular, a global $H^1$-bound for $F_j'(\rho_j^+)$. This will follow from the next results, together with the bound on $\nabla F_j'(\rho_j^+)$ from Lemma \ref{lem:heat_comparison}.

\begin{lemma} \label{lem:h1est_bigrho}
Let $\hat\rrho$ satisfy the assumptions of Lemma \ref{lem:el_yosida}. Then it holds
\begin{equation} \label{eq:h1est_bigrho}
\sum_{j = 1,2} \intrd \left| \nabla [F_j'(\rho_j^+) + \eps \partial_{r_j}h(\rho_1^+, \rho_2^+)]\right|^2\,\rho_j^+\,\intd x \leq 4\, \frac{\me_\eps(\hat\rrho) - \me_\eps(\rrho^+)}{\tau} + \frac{8 C_K^2}{\lambda}\, \me_\eps(\rrho^+)\,.
\end{equation}
\end{lemma}
\begin{proof}
Since $K \ast \rho_j^+$ is differentiable everywhere, the Euler-Lagrange equations in the form \eqref{eq:el_map_yosida} imply that the gradients on the left-hand side are well-defined a.e. in $\{\rho_j^+ > 0 \}$. For $j = 1$ (the other term is analogous), it holds
\begin{align*}
&\intrd \left| \nabla [F_1'(\rho_1^+) + \eps \partial_{r_1}h(\rho_1^+, \rho_2^+)]\right|^2\,\rho_1^+\,\intd x \\
&\leq 2 \intrd  \left| \nabla [F_1'(\rho_1^+) + \eps \partial_{r_1}h(\rho_1^+, \rho_2^+) + K \ast \rho_2^+]\right|^2\,\rho_1^+\,\intd x + 2 \intrd \left| \nabla [K \ast \rho_2^+]\right|^2\,\rho_1^+\,\intd x \\
&\leq \frac{2}{\tau^2} \intrd |T_1(x) - x|^2\,\rho_1^+\,\intd x + \frac{4C_K^2}{\lambda} \intrd \rho_1^+ K \ast \rho_2^+\,\intd x \leq \frac{2}{\tau^2}\, \wass(\rho_1^+, \hat\rho_1)^2 + \frac{4C_K^2}{\lambda}\, \me_\eps(\rrho^+),
\end{align*}
where the second to last inequality uses \eqref{eq:el_map_yosida} for the first integral, and the fact that $K$, and thus $K \ast \rho_2^+$, is non-negative and $\lambda$-convex with its second derivative bounded by $C_K$ for the second integral. Adding the analogous term for $j = 2$ yields
\begin{align*}
\sum_{j = 1,2} \intrd \left| \nabla [F_j'(\rho_j^+) + \eps \partial_{r_j}h(\rho_1^+, \rho_2^+)]\right|^2\,\rho_j^+\,\intd x \leq \frac{2}{\tau^2}\, \dst(\rrho^+, \hat\rrho)^2 + \frac{8C_K^2}{\lambda}\, \me_\eps(\rrho^+),
\end{align*}
which together with \eqref{eq:step_estimate} yields the claimed estimate \eqref{eq:h1est_bigrho}.
\end{proof}
In order to obtain a global $H^1$-estimate for the functions $F_j'(\rho_j^+)$, we first need to prove existence of their gradients. This follows directly from the following lemma, which is a consequence of the Euler-Lagrange system \eqref{eq:el_yosida}.

\begin{lemma} \label{lem:min_lip}
Under the assumptions of Lemma \ref{lem:el_yosida}, the $F_j'(\rho_j^+)$ are locally Lipschitz in $\Rd$.
\end{lemma}
\begin{proof}
This follows from \eqref{eq:el_yosida} in the same way as in Lemma \ref{lem:minim_fj_lip}, since the right-hand side of \eqref{eq:el_yosida} is locally Lipschitz in $\Rd$. Note that the Kantorovich potentials $\varphi_j$ are of the form $\frac{1}{2}|\cdot|^2 - \Phi_j$, where $\Phi_j$ are convex functions which are finite in $\Rd$, since the left-hand side of \eqref{eq:el_yosida} is finite everywhere.
\end{proof}
With this regularity result, we can prove the following estimate for $\nabla F_j'(\rho_j^+)$:
\begin{lemma} \label{lem:h1est_bigrho_2}
Let $\hat\rrho$ satisfy the assumptions of Lemma \ref{lem:el_yosida}, and assume that $\eps$ is such that $3 \kappa^2 \eps^2 < 1$, where $\kappa^2 := \max\{\kappa_{1,1}, \kappa_{1,2}\}^2 + \max\{\kappa_{2,1},  \kappa_{2,2}\}^2$ with the constants $\kappa_{j,i}$ from Hypothesis \ref{hyp:theta}. Then there exists a constant $B$, independent of $\hat\rrho$ and $\rrho^+$, such that
\begin{equation*}
\sum_{j = 1,2} \intrd \left|\nabla F_j'(\rho_j^+)\right|^2\,\rho_j^+\,\intd x \leq B\, \left(\frac{\me_\eps(\hat\rrho) - \me_\eps(\rrho^+)}{\tau} + \frac{2 C_K^2}{\lambda}\, \me_\eps(\rrho^+)\right).
\end{equation*}
\end{lemma}
\begin{proof}
The gradient of $F_j'(\rho_j^+)$ exists almost everywhere by local Lipschitz-continuity, and the same is true for $F_j'(\rho_j^+) + \eps \partial_{r_j}h(\rho_1^+, \rho_2^+)$, see the proof of Lemma \ref{lem:h1est_bigrho}. Thus, the gradient of $\partial_{r_j}h(\rho_1^+, \rho_2^+)$ is also well-defined in $\Rd$. The claim will now follow from Lemma \ref{lem:h1est_bigrho} and Hypothesis \ref{hyp:theta}: With $\uu^+ := (F_1'(\rho_1^+), F_2'(\rho_2^+))$, we have
\begin{align*}
\nabla \partial_{r_j}h(\rho_1^+, \rho_2^+) = \nabla \theta_j(\uu^+) = \theta_{j,1}(\uu^+) \nabla F_1'(\rho_1^+) + \theta_{j,2}(\uu^+) \nabla F_2'(\rho_2^+).
\end{align*}
By splitting the integral, we thus obtain for $j = 1,2$ the estimate
\begin{align*}
&\intrd \left|\nabla F_j'(\rho_j^+)\right|^2\,\rho_j^+\,\intd x = \intrd \left| \nabla \left[F_j'(\rho_j^+) + \eps \partial_{r_j}h(\rho_1^+, \rho_2^+)\right] - \eps \sum_{i = 1,2} \theta_{j,i}(\uu^+) \nabla F_i'(\rho_i^+) \right|^2\,\rho_j^+\,\intd x \\ &\leq 3 \intrd \left| \nabla [F_j'(\rho_j^+) + \eps \partial_{r_j}h(\rho_1^+, \rho_2^+)]\right|^2\,\rho_j^+\,\intd x + 3 \eps^2 \sum_{i = 1,2} \intrd \theta^2_{j,i}(\uu^+) \left| \nabla F_i'(\rho_i^+)\right|^2\,\rho_j^+ \,\intd x.
\end{align*}
Summing over $j = 1,2$ and applying \eqref{eq:h1est_bigrho} for the first term and \eqref{eq:thetabdd_swap} for the other two yields
\begin{align*}
\sum_{j = 1,2} \intrd \left|\nabla F_j'(\rho_j^+)\right|^2\,\rho_j^+\,\intd x &\leq 12\, \frac{\me_\eps(\hat\rrho) - \me_\eps(\rrho^+)}{\tau} + \frac{24 C_K^2}{\lambda}\, \me_\eps(\rrho^+) \\ &+ 3\kappa^2\eps^2 \sum_{i = 1,2} \intrd \left|\nabla F_i'(\rho_i^+)\right|^2\,\rho_i^+\,\intd x.
\end{align*}
Subtracting the last term and using $3 \kappa^2\eps^2 < 1$ yields the claim with $B = \frac{12}{1 - 3 \kappa^2\eps^2}$.
\end{proof}
What makes Lemma \ref{lem:h1est_bigrho_2} important is that it controls $|\nabla F_j'(\rho_j^+)|^2$ on the set where $\rho_j^+$ is large. Together with Lemma \ref{lem:heat_comparison}, which controls the same gradient when $\rho_j^+$ is small, this implies an explicit global $L^2$-bound for $\nabla F_j(\rho_j^+)$:
\begin{corollary} \label{cor:l2bound_gradF}
Let $\hat\rrho$ satisfy the assumptions of Lemma \ref{lem:el_yosida}. Let $\alpha > 0$ be arbitrary, and denote by $A$ and $B$ the constants from Lemmas \ref{lem:heat_comparison} and \ref{lem:h1est_bigrho_2}, respectively. Then it holds
\begin{equation} \label{eq:l2bound_gradF}
\sum_{j = 1,2} \intrd \left|\nabla F_j'(\rho_j^+)\right|^2\,\intd x \leq C\, \left(\frac{\me_\eps(\hat\rrho) - \me_\eps(\rrho^+)}{\tau} + \frac{2 C_K^2}{\lambda}\, \me_\eps(\rrho^+) + \frac{\mh_c(\hat\rrho) - \mh_c(\rrho^+)}{\tau} + 2 d C_K \right),
\end{equation}
with the constant $C := \max\{\frac{B}{\alpha}, 2A\}$.
\end{corollary}
\begin{proof}
For $j = 1,2$, we split the integral into one part where $\rho_j^+$ is large and one where it is small. By local Lipschitz-continuity, it holds $\nabla F_j'(\rho_j^+) = \nabla F_j'([ \rho_j^+ ]_\alpha)$ on $\{\rho_j^+ < \alpha \}$, hence
\begin{align*}
\intrd \left|\nabla F_j'(\rho_j^+)\right|^2\,\intd x &= \int_{\{\rho_j^+ \geq \alpha\}} \left|\nabla F_j'(\rho_j^+)\right|^2\,\intd x + \int_{\{\rho_j^+ < \alpha\}} \left|\nabla F_j'(\rho_j^+)\right|^2\,\intd x \\
&\leq \frac{1}{\alpha} \intrd \left|\nabla F_j'(\rho_j^+)\right|^2\,\rho_j^+\,\intd x + \intrd \left|\nabla F_j'([\rho_j^+]_\alpha)\right|^2\,\intd x.
\end{align*}
The claim now follows by summing over $j = 1,2$ and applying Lemmas \ref{lem:h1est_bigrho_2} and \ref{lem:heat_comparison}.
\end{proof}

\subsection{Generalization of the time-discrete error estimates}
With the help of the estimate from Corollary \ref{cor:l2bound_gradF}, we can now prove the time-discrete error estimate \eqref{eq:weaksol_estim_error} for much more general $\hat\rrho$. Concretely, we prove the following theorem:
\begin{theorem} \label{thm:discrete_error}
Let $\hat\rrho \in \mptrd^2$ be such that $\me_\eps(\hat\rrho) < +\infty$. Then there exists a minimizer $\rrho^+ \in \mptrd^2$ of the functional $\eepstaurhat{\cdot}$, such that for all $\zeta \in C^\infty_c(\Rd)$:
\begin{equation*}
\mr_1 + \mr_2 \leq \left( \me_\eps(\hat\rrho) - \me_\eps(\rrho^+) \right)\, \|\zeta\|_{C^2}\,,
\end{equation*}
where $\mr_1$ and $\mr_2$ are the integral errors defined as in Corollary \ref{cor:weaksol_conditions}.
\end{theorem}
\begin{proof}
In Corollary \ref{cor:weaksol_conditions}, the claim has already been shown under the more restrictive assumptions that both components of $\hat\rrho$ are strictly positive almost everywhere inside balls $B_{R_j}$ and vanish almost everywhere else. We prove the more general case by approximating $\hat\rrho$ by a sequence of pairs that satisfy these stronger assumptions.

Let $\hat\rrho$ be as above. Without loss of generality, we can assume $\hat\rrho \in \sspc$. We start by showing that the assumption $\me_\eps(\hat\rrho) <+\infty$ implies that $\mh_c(\hat\rrho) < +\infty$ as well. To see this, observe that by comparing the asymptotic growths of the derivatives, it holds $r \log r \leq cF_j(r)$ for every $r \geq 0$ with some constant $c$. It thus holds for every $\rrho \in \mptrd$
\begin{align*}
\mh_c(\rrho) = \intrd \left[\rho_1 \log \rho_1 + \rho_2 \log \rho_2\right]\,\intd x \leq c\intrd \left[ F_1(\rho_1) + F_2(\rho_2)\right]\,\intd x \leq 2c\,\me_\eps(\rrho),
\end{align*}
proving that $\mh_c(\hat\rrho)$ is finite.

For $j = 1,2$, we now take a sequence of positive radii $0 < R_{j,1} < R_{j, 2} < \dots$ such that $R_{j,n} \nearrow$  $R_{j, \infty} := \sup\,\{R > 0: \supp\, \hat\rho_j \nsubseteq B_R\} \in (0, +\infty]$.  Further, we fix constants $H_j,C_j>0$ such that for every $n$, it holds
\begin{equation*}
|B_{R_{j,n}}'| \geq C_j\,, \quad \text{where } B_{R_{j,n}}':= \{\hat\rho_j \leq H_j\} \cap B_{R_{j,n}}.
\end{equation*}
Define for every $n$ the probability densities
\begin{equation*}
\hat\rho_{j,n} := \mathds{1}_{B_{R_{j,n}}}\,\hat\rho_j + \frac{m_{j,n}}{|B_{R_{j,n}}'|} \mathds{1}_{B_{R_{j,n}}'}\,, \quad \text{with } m_{j,n}:= 1 - \int_{B_{R_{j,n}}} \hat\rho_j\,\intd x \in (0, 1].
\end{equation*}
For the densities $\hat\rho_{j,n}$, we have $\{\hat\rho_{j,n} > 0\} = B_{R_{j,n}}$. Hence, the pairs $\hat\rrho_n := (\hat\rho_{1,n}, \hat\rho_{2,n})$ satisfy the assumptions of Corollary \ref{cor:weaksol_conditions}, and thus for every minimizer $\rrho_n^+$ of $\eepstau(\cdot|\hat\rho_n)$, the estimate \eqref{eq:weaksol_estim_error}. The same estimate for $\hat\rrho$, which is the claim, thus follows if we prove for $n \to \infty$ the convergences
\begin{equation} \label{eq:integral_1_cvgce}
\intrd \frac{\rho_{1,n}^+ - \hat\rho_{1,n}}{\tau}\,\zeta\,\intd x \rightarrow \intrd \frac{\rho_1^+ - \hat\rho_1}{\tau}\,\zeta\,\intd x,
\end{equation}
\begin{equation} \label{eq:integral_2_cvgce}
\begin{split}
&\intrd \rho_{1,n}^+\, \nabla\left[F_1'(\rho_{1,n}^+) + \eps\partial_{r_1}h(\rho_{1,n}^+, \rho_{2,n}^+) + K \ast \rho_{2,n}^+\right]\cdot\nabla \zeta\,\intd x \\
&\rightarrow \intrd \rho_1^+\, \nabla\left[F_1'(\rho_1^+) + \eps\partial_{r_1}h(\rho_1^+, \rho_2^+) + K \ast \rho_2^+\right]\cdot\nabla \zeta\,\intd x\,,
\end{split}
\end{equation}

\begin{equation} \label{eq:energy_cvgce}
\limsup_{n \to \infty}\, \left( \me_\eps(\hat\rrho_n) - \me_\eps(\rrho^+_n) \right) \leq  \me_\eps(\hat\rrho) - \me_\eps(\rrho^+) 
\end{equation}
for some minimizer $\rrho^+$ of $\eepstaurhat{\cdot}$, up to subsequence. The convergence of the other component's terms will follow analogously.

We start by observing that by construction, $\hat\rho_{j,n} \to \hat\rho_j$ in $L^1(\Rd)$, and with respect to the $\wass$-metric. Additionally, it holds that $\me_\eps(\hat\rrho_n) \to \me_\eps(\hat\rrho)$ and $\mh_c(\hat\rrho_n) \to \mh_c(\hat\rrho)$, which can be shown from the construction by dominated convergence. In particular, since $\me_\eps(\hat\rrho)$ and $\mh_c(\hat\rrho)$ are finite, $(\me_\eps(\hat\rrho_n))_n$ and $(\mh_c(\hat\rrho_n))_n$ are bounded sequences. Since it holds $\me_\eps(\rrho_n^+) \leq \me_\eps(\hat\rrho_n)$, the sequence $(\me_\eps(\rrho^+_n))_n$ is thus bounded as well. By Lemma \ref{lem:yosida_min_exist}, it holds $\mom_1[\rho_1^+] + \mom_1[\rho_2^+] = \mom_1[\hat\rho_1^n] + \mom_1[\hat\rho_2^n] \to \mom_1[\hat\rho_1] + \mom_1[\hat\rho_2] = 0$. Lemma \ref{lem:eeps_coer} thus yields $\rrho_n^+ \weakto \rrho^+ \in \mptrd^2$ weakly in $L^1$ up to subsequence. We check that $\rrho^+$ is a global minimizer of the functional $\eepstaurhat{\cdot}$. Indeed, it holds
\begin{equation} \label{eq:energy_lsc}
\me_\eps(\rrho^+) \leq \liminf_{n \to \infty} \me_\eps(\rrho_n^+)
\end{equation}
by weak $L^1$-convergence of the $\rrho_n^+$ and Lemma \ref{lem:eeps_lsc}. By lower semicontinuity of the distance with respect to weak $L^1$-convergence, the triangle inequality and $\dst(\hat\rrho_n, \hat\rrho) \to 0$, we also get
\begin{align*}
\dst(\rrho^+, \hat\rrho)^2 \leq \liminf_{n \to \infty} \dst(\rrho_n^+, \hat\rrho)^2 \leq \liminf_{n \to \infty} \dst(\rrho_n^+, \hat\rrho_n)^2.
\end{align*}
Together with the minimality of $\rrho_n^+$ for $\eepstau(\cdot|\hat\rrho_n)$, this implies for every $\rrho \in \mptrd$:
\begin{align*}
&\eepstaurhat{\rrho^+} = \frac{1}{2\tau} \dst(\rrho^+, \hat\rrho)^2 + \me_\eps(\rrho^+) \leq \liminf_{n \to \infty} \left( \frac{1}{2\tau} \dst(\rrho_n^+, \hat\rrho_n)^2 + \me_\eps(\rrho_n^+) \right) \\
&\leq \liminf_{n \to \infty} \left( \frac{1}{2\tau} \dst(\rrho, \hat\rrho_n)^2 + \me_\eps(\rrho) \right) =  \frac{1}{2\tau} \dst(\rrho, \hat\rrho)^2 + \me_\eps(\rrho) = \eepstaurhat{\rrho}\,,
\end{align*}
proving that $\rrho^+$ is a minimizer of $\eepstaurhat{\cdot}$.

The convergence \eqref{eq:integral_1_cvgce} is clear by narrow convergence of $\hat\rho_{1,n} \to \hat\rho_1$ and $\rho_{1,n}^+ \to \rho_1^+$. Since \eqref{eq:energy_lsc} and $\me_\eps(\hat\rrho_n) \to \me_\eps(\hat\rrho)$ imply the inequality \eqref{eq:energy_cvgce}, it remains to prove \eqref{eq:integral_2_cvgce}. We show this convergence by proving that in $L^2(S)$ with $S := \supp\,\zeta \subset \subset \Rd$, it holds for $j = 1,2$:
\begin{align*}
\begin{split}
\nabla F_j'(\rho_{j,n}^+) \weakto \nabla F_j'(\rho_j^+), \quad  \nabla \partial_{r_j}h(\rho_{1,n}^+, \rho_{2,n}^+) &\weakto \nabla \partial_{r_j}h(\rho_1^+, \rho_2^+), \quad \nabla K \ast \rho_{j,n}^+ \weakto  \nabla K \ast \rho_j^+\text{ weakly,} \\
\rho_{j,n}^+ &\to \rho_j^+ \text{ strongly.}
\end{split}
\end{align*}
We start by showing that $(F_j'(\rho_{j,n}^+))_n$ is a bounded sequence in $H^1(S)$. In order to apply Corollary \ref{cor:l2bound_gradF}, note that the sequences $(\me_\eps(\hat\rrho_n))_n$, $(\me_\eps(\rrho^+_n))_n$ and $(\mh_c(\hat\rrho_n))_n$ are bounded from above, and that $(\me_\eps(\rrho^+_n))_n$ and $(\mh_c(\rrho^+_n))_n$ are bounded from below, which follows from non-negativity of $\me_\eps$ and the general inequality
\begin{align} \label{eq:entropy_lower_bd}
\intrd \rho \log \rho \,\intd x \geq 1 - \mom_2[\rho] - \pi^{d/2} \qquad \text{ for all } \rho \in \mptrd,
\end{align}
and from the boundedness of the second moments of $\rho_{j,n}^+$. Hence, the right hand sides of \eqref{eq:l2bound_gradF} are bounded from above by a constant for all $n$, which implies uniform boundedness of $\nabla F_1'(\rho^+_{1,n})$ in $L^2(\Rd)$. In order to also get uniform boundedness of $F_j'(\rho_{j,n}^+)$ in $L^2(S)$, note that $F_j'(\rho_{j,n}^+)$ is uniformly bounded in $L^1(\Rd)$ by inequality \eqref{eq:estim_fjp_eeps}, thus uniform boundedness in $L^2(S)$ follows from the Poincaré inequality in the smooth and bounded set $S$.
Hence, the sequence $(F_j'(\rho_{j,n}^+))_n$ is bounded in $H^1(S)$. Alaoglu's and Rellich's theorems thus imply $F_j'(\rho_{j,n}^+) \to u_j^+$ strongly in $L^2(S)$, weakly in $H^1(S)$, for some $u_j^+ \in H^1(S)$. The identification $u_j^+ = F_j'(\rho_j^+)$ now follows from the narrow convergence $\rho_{j,n}^+ \to \rho_j^+$, since $\rho_{j,n}^+ \to (F_j')^{-1}(u_j^+)$ pointwise almost everywhere along a subsequence. This finishes the proof of the weak $L^2(S)$-convergence $\nabla F_j'(\rho_{j,n}^+) \weakto \nabla F_j'(\rho_j^+)$.

In order to obtain weak $L^2(S)$-convergence $\nabla \partial_{r_j}h(\rho_{1,n}^+, \rho_{2,n}^+) \weakto \nabla \partial_{r_j}h(\rho_1^+, \rho_2^+)$, we write
\begin{align*}
\nabla \partial_{r_j}h(\rho_{1,n}^+, \rho_{2,n}^+) = \nabla \theta_j\left(F_1'(\rho_{1,n}^+), F_2'(\rho_{2,n}^+)\right) = \sum_{i=1,2} \theta_{j,i}\left(F_1'(\rho_{1,n}^+), F_2'(\rho_{2,n}^+)\right) \nabla F_i'(\rho_{i,n}^+).
\end{align*}
Denoting $\uu_n^+:= (u_{1,n}^+, u_{2,n}^+) := \left(F_1'(\rho_{1,n}^+), F_2'(\rho_{2,n}^+)\right)$, it thus follows from assumption \eqref{eq:thetabdd_swap}:
\begin{align*}
\|\nabla \partial_{r_j}h(\rho_{1,n}^+, \rho_{2,n}^+)\|_{L^2(S)} \leq \sum_{i = 1,2} \left( \int_S \theta^2_{j,i}\left(\uu_n^+\right) |\nabla u_{i,n}^+|^2\,\intd x \right)^{\frac{1}{2}} \leq \sum_{i = 1,2} \kappa_{j,i}\, \|u_{i,n}^+\|_{L^2(S)},
\end{align*}
which is uniformly bounded for all $n$, since $u_{i,n}^+ = F_i'(\rho_{i,n}^+)$ is uniformly bounded in $L^2(S)$. Hence up to subsequence, $\nabla \partial_{r_j}h(\rho_{1,n}^+, \rho_{2,n}^+) \weakto v$ for some $v \in L^2(S; \Rd)$. In order to identify the limit $v = \nabla \partial_{r_j}h(\rho_{1}^+, \rho_{2}^+)$, notice that for every $\psi \in C^\infty_c(S; \Rd)$, it holds
\begin{align*}
\int_S \nabla \partial_{r_j}h(\rho_{1,n}^+, \rho_{2,n}^+) \cdot \psi\,\intd x &= - \int_S \partial_{r_j}h(\rho_{1,n}^+, \rho_{2,n}^+)\ \dv\, \psi\,\intd x \\
&\to - \int_S \partial_{r_j}h(\rho_{1}^+, \rho_{2}^+)\ \dv\,\psi\,\intd x = \int_S \nabla\partial_{r_j}h(\rho_{1}^+, \rho_{2}^+)\cdot \psi\,\intd x,
\end{align*}
where we have used that $\partial_{r_j}h(\rho_{1,n}^+, \rho_{2,n}^+) \to \partial_{r_j}h(\rho_1^+, \rho_2^+)$ strongly in $L^2(S)$, which follows from $\partial_{r_j}h(\rho_{1,n}^+, \rho_{2,n}^+) = \theta_j\left(F_1'(\rho_{1,n}^+), F_2'(\rho_{2,n}^+)\right)$ and the $L^2(S)$-convergence $F_i'(\rho_{i,n}^+) \to F_i'(\rho_i^+)$, since $\theta_j$ is globally $(\kappa_{j,1} + \kappa_{j,2})$-Lipschitz by \eqref{eq:thetabdd_swap}. This implies $v = \nabla\partial_{r_j}h(\rho_{1}^+, \rho_{2}^+)$.

To prove $\nabla K \ast \rho_{j,n}^+ \weakto  \nabla K \ast \rho_j^+$ weakly in $L^2(S)$, fix any $\varphi \in L^2(\Rd; \Rd)$ with $\varphi \equiv 0$ in $\Rd \setminus S$. It now holds
\begin{align*}
\int_S \varphi \cdot \nabla K \ast \rho_{j,n}^+\,\intd x = -\intrd (\nabla K \ast \varphi)\, \rho_{j,n}^+\,\intd x \to -\intrd (\nabla K \ast \varphi)\, \rho_j^+\,\intd x = \int_S \varphi \cdot \nabla K \ast \rho_j^+\,\intd x,
\end{align*}
since $\rho_{j,n}^+ \to \rho_j^+$ narrowly with uniformly bounded second moments, and since $\nabla K \ast \varphi$ is continuous and grows at most linearly by the growth assumption on $K$.

It remains to prove $\rho_{j,n}^+ \to \rho_j^+$ strongly in $L^2(S)$, which we derive from the strong $L^2(S)$-convergence of $F_j'(\rho_{j,n}^+) \to F_j'(\rho_j^+)$. Since for non-negative functions $u_n$ and $u$, the convergence $u_n \to u$ in $L^2$ is equivalent to the convergence $u_n^2 \to u^2$ in $L^1$, it is sufficient to show that every subsequence of $((\rho_{j,n}^+)^2)_n$ contains a further subsequence that converges to $(\rho_j^+)^2$ strongly in $L^1(S)$. Since strong $L^1(S)$-convergence follows from equi-integrability together with pointwise convergence almost everywhere, and since $F_j'(\rho_{j,n}^+) \to F_j'(\rho_j^+)$ strongly in $L^2(S)$, we only need to prove equi-integrability of the sequence $((\rho_{j,n}^+)^2)_n$ in $S$. By the $L^1(S)$-convergence $F_j'(\rho_{j,n}^+)^2 \to F_j'(\rho_j^+)^2$, the sequence $(F_j'(\rho_{j,n}^+)^2)_n$ is equi-integrable in $S$. It thus follows that
\begin{align*}
\rho_{j,n}^2 = g\left(F_j'(\rho_{j,n}^+)^2\right) \quad \text{ with } g(r):= (F_j')^{-1}(\sqrt r)^2
\end{align*}
is also equi-integrable in $S$, since the non-decreasing function $g$ grows at most linearly for large $r$ by assumption \eqref{eq:fjp_growth} on $F_j''$. This finishes the proof of $\rho_{j,n}^+ \to \rho_j^+$ strongly in $L^2(S)$, thus concluding the proof of the theorem.
\end{proof}
\begin{remark}
From the approximation argument in the proof above, we also obtain the following generalization of Corollary \ref{cor:l2bound_gradF}: If $\hat\rrho$ is such that $\me_\eps(\hat\rrho) < +\infty$, it holds
\begin{equation} \label{eq:l2bound_gradF_gen}
\sum_{j = 1,2} \intrd \left|\nabla F_j'(\rho_j^+)\right|^2\,\intd x \leq C\, \left(\frac{\me_\eps(\hat\rrho) - \me_\eps(\rrho^+)}{\tau} + \frac{2 C_K^2}{\lambda}\, \me_\eps(\hat\rrho) + \frac{\mh_c(\hat\rrho) - \mh_c(\rrho^+)}{\tau} + 2 d C_K \right),
\end{equation}
where the only difference to \eqref{eq:l2bound_gradF} is the appearance of $\me_\eps(\hat\rrho)$ instead of $\me_\eps(\rrho^+)$ in the second term. To see this, notice that \eqref{eq:l2bound_gradF} holds for every $\hat\rrho_n$ and $\rrho_n^+$ from the proof of Theorem \ref{thm:discrete_error}. On the left-hand side, we can pass to the limit $n \to \infty$ since $\nabla F_j'(\rho_{j,n}^+) \weakto \nabla F_j'(\rho_j^+)$ weakly in $L^2(\Rd)$ and the $L^2$-norm is weakly lower semicontinuous. The weak $L^2(\Rd)$-convergence follows from the uniform boundedness of $\nabla F_j'(\rho_{j,n}^+)$ that is obtained from \eqref{eq:l2bound_gradF}, Alaoglu's theorem and the weak $L^2_{loc}$-convergence proven above. On the right-hand side, we use that $\mh_c(\hat\rrho_n) \to \mh_c(\hat\rrho)$, $\me_\eps(\rrho^+) \leq \liminf_{n} \me_\eps(\rrho_n^+)$, $\mh_c(\rrho^+) \leq \liminf_{n} \mh_c(\rrho_n^+)$ and $\me_\eps(\hat\rrho) = \lim_n \me_\eps(\hat\rrho_n) \geq \limsup_n \me_\eps(\rrho_n^+)$, to pass to the limit and obtain \eqref{eq:l2bound_gradF_gen}.
\end{remark}
\subsection{Interpolation and convergence to weak solution} \label{subsec:weaksol}
We fix any initial data $\rrho^0 \in \sspc$ such that $\me_\eps(\rrho^0)$ is finite. For every $\tau > 0$, we iteratively define a sequence of feasible pairs $(\hat\rrho_n)_n = (\hat\rho_{1,n}, \hat\rho_{2,n})_n \subset \sspc$ in the following way:
\begin{equation} \label{eq:interpol_def}
\hat\rrho_0 := \rrho^0, \qquad \hat\rrho_n := \rrho_{n-1}^+ \in \argmin_{\rrho \in \sspc} \eepstau(\rrho|\hat\rrho_{n-1}) \ \text{ for all } n \geq 1
\end{equation}
where the minimizer $\rrho_{n-1}^+$ is chosen to be the one constructed in Theorem \ref{thm:discrete_error}. Notice that the assumption $\me_\eps(\hat\rrho_{n-1}) < +\infty$ for the theorem is fulfilled in every iteration, as $\me_\eps(\hat\rrho_n) \leq \me_\eps(\hat\rrho_{n-1})$ by Lemma \ref{lem:yosida_min_exist}.
We now define the piecewise constant interpolating curve $\rrho^\tau = (\rho_1^\tau, \rho_2^\tau): [0,+\infty) \to \sspc$ as
\begin{align*}
\rrho^\tau(t) := \hat\rrho_{n+1} = \rrho_n^+ \quad \text{ for all } t \text{ such that } n\tau < t \leq (n+1)\tau.
\end{align*}
The goal is to prove that for some subsequence $\tau \downarrow 0$, the curves $\rrho^\tau$ converge in some topology to a continuous limit curve $\rrho: [0,+\infty) \to \sspc$, which is a weak solution to the evolution equations \eqref{eq:system-intro}. We first prove that the piecewise constant interpolations $\rrho^\tau$ satisfy a weak formulation of the system up to a small error, which is a direct consequence of Theorem \ref{thm:discrete_error}.
\begin{corollary} \label{cor:weaksol_tau_error}
Let $\zeta \in C^\infty_c(\R \times \Rd)$ be arbitrary. Define for $\tau > 0$ the time-discrete approximation $\zeta^\tau$ on $[0, \infty) \times \Rd$ as $\zeta^\tau(t,x) := \zeta_n(x) := \zeta(n\tau, x)$ for $n\tau \leq t < (n+1)\tau$. Then it holds
\begin{equation} \label{eq:weaksol_tau_error}
\left| \int_0^\infty\intrd \left[  \rho_1^\tau\, \partial_t \zeta - \rho_1^\tau\,\nabla\left[F_1'(\rho_1^\tau) + \eps \partial_{r_1}h(\rho_1^\tau, \rho_2^\tau) + K \ast \rho_2^\tau\right]\cdot\nabla \zeta^\tau \right]\,\intd x\intd t + \intrd \hat\rho_1 \zeta_0\,\intd x \right| \leq C\tau,
\end{equation}
and analogously for the other component, where $C = \me_\eps(\hat\rrho)\,\|\zeta\|_{C^2}$ is independent of $\tau$.
\end{corollary}
\begin{proof}
Since $\rrho^\tau$ and $\zeta^\tau$ do not depend on $t$ in every interval $(n\tau, (n+1)\tau)$, it holds
\begin{align*}
\int_0^\infty\intrd \rho_1^\tau\, \partial_t\zeta\,\intd x\intd t &= \sum_{n = 0}^\infty \int_{n\tau}^{(n+1)\tau} \intrd \rho_{1,n}^+\,\partial_t\zeta\,\intd x\intd t
= \sum_{n = 0}^\infty \intrd \rho_{1,n}^+\,\left(\zeta_{n+1} -\zeta_n\right)\,\intd x \\
&= \tau\sum_{n= 0}^\infty \intrd \frac{\hat\rho_{1,n} - \rho_{1,n}^+}{\tau}\, \zeta_n\,\intd x - \intrd \hat\rho_1 \zeta_0\,\intd x
\end{align*}
and
\begin{align*}
&\int_0^\infty \intrd \rho_1^\tau\,\nabla\left[F_1'(\rho_1^\tau) + \eps \partial_{r_1}h(\rho_1^\tau, \rho_2^\tau) + K \ast \rho_2^\tau\right]\cdot\nabla \zeta^\tau \,\intd x\intd t \\
&= \tau \sum_{n=0}^\infty \intrd \rho_{1,n}^+\,\nabla\left[F_1'(\rho_{1,n}^+) + \eps \partial_{r_1}h(\rho_{1,n}^+, \rho_{2,n}^+) + K \ast \rho_{2,n}^+\right]\cdot\nabla \zeta_n \,\intd x.
\end{align*}
By applying Theorem \ref{thm:discrete_error}, we thus obtain the estimate
\begin{align*}
&\left| \int_0^\infty\intrd \left[  \rho_1^\tau\, \partial_t \zeta - \rho_1^\tau\,\nabla\left[F_1'(\rho_1^\tau) + \eps \partial_{r_1}h(\rho_1^\tau, \rho_2^\tau) + K \ast \rho_2^\tau\right]\cdot\nabla \zeta^\tau \right]\,\intd x\intd t + \intrd \hat\rho_1 \zeta_0\,\intd x \right| \\
&\leq \tau \sum_{n=0}^\infty \left(\me_\eps(\hat\rrho_n) - \me_\eps(\rrho^+_n)\right)\,\|\zeta_n\|_{C^2} \leq \me_\eps(\rrho^0)\,\|\zeta\|_{C^2} \tau,
\end{align*}
where we have used $\me_\eps(\rrho_n^+) \geq 0$, $\rrho_n^+ = \hat\rrho_{n+1}$ and $\hat\rrho_0 = \rrho^0$. The proof for the other component is analogous.
\end{proof}
In order to establish convergence of $\rrho^\tau$ to some limiting curve $\rrho$, we first prove the following a priori estimate on the jump sizes of the piecewise constant curves $\rrho^\tau$:
\begin{lemma} \label{lem:rtau_quasicont}
For every $\tau > 0$ and every $s,t \geq 0$, it holds
\begin{equation} \label{eq:rtau_quasicont}
\dst\left(\rrho^\tau(s), \rrho^\tau(t)\right)^2 \leq 2 \me_\eps(\rrho^0)\left(|s-t| + \tau\right)
\end{equation}
\end{lemma}
\begin{proof}
Fix $\tau > 0$. By Lemma \ref{lem:yosida_min_exist}, it holds $\dst(\hat\rrho, \rrho^+)^2 \leq 2\tau ( \me_\eps(\hat\rrho) - \me_\eps(\rrho^+))$ for all $\hat\rrho \in \mptrd^2$. Thus for all integers $m \geq n \geq 0$ and every $s,t \geq 0$ such that $\rrho^\tau(s) = \rrho_n^+$ and $ \rrho^\tau(t) = \rrho_m^+$, the triangle inequality for the metric $\dst$ and the identity $\rrho_k^+ = \hat\rrho_{k+1}$ yield
\begin{align*}
&\dst\left(\rrho^\tau(s), \rrho^\tau(t)\right)^2 = \dst\left(\rrho_n^+, \rrho_m^+\right)^2 \leq \left( \sum_{k = n}^{m-1} \dst(\rrho_k^+, \rrho_{k+1}^+)\right)^2 \leq (m-n) \sum_{k = n}^{m-1} \dst(\rrho_k^+, \rrho_{k+1}^+)^2 \\
&\leq 2\tau(m-n)\sum_{k = n}^{m-1} \left(\me_\eps(\rrho_k^+) - \me_\eps(\rrho_{k+1}^+) \right) = 2\tau(m-n)\left(\me_\eps(\rrho_n^+) - \me_\eps(\rrho_{m}^+) \right) \leq 2\tau(m-n)\, \me_\eps(\rrho^0).
\end{align*}
By construction of $\rrho^\tau$, we have $s \leq (n + 1)\tau$ and $t > m\tau$, which imply $(m-n)\tau \leq t-s+\tau$. Inserting into the above estimate yields \eqref{eq:rtau_quasicont}. The case $n > m$ is analogous.
\end{proof}

With the help of Lemma \ref{lem:rtau_quasicont}, we prove the following convergence result for $\rrho^\tau$:
\begin{lemma} \label{lem:interpol_cvgce}
There exists a subsequence $\tau \downarrow 0$ and a curve $\rrho^*:[0,+\infty) \to \sspc$, which is Hölder-$\frac{1}{2}$-continuous with respect to the metric $\dst$, such that $\rho_j^\tau(t) \weakto \rho_j^*(t)$ weakly in $L^1(\Rd)$ at every $t \geq 0$.
\end{lemma}
\begin{proof}
For every $t \geq 0$ and $\tau > 0$, it holds $\rrho^\tau(t) \in \sspc$ and
$\me_\eps(\rrho^\tau(t)) \leq \me_\eps(\hat\rrho) < +\infty$.
Lemma \ref{lem:eeps_coer} thus implies weak $L^1$-compactness of the sequence at every $t \geq 0$. By a diagonal argument, this implies existence of a single subsequence $\tau \downarrow 0$, along which $\rho_j^\tau(t)$ is weakly $L^1$-convergent to some $\rho_j^*(t)$ at every $t \in \Q \cap [0, +\infty)$ for both $j$. By lower semicontinuity of the distance $\dst$ with respect to weak $L^1$-convergence and Lemma \ref{lem:rtau_quasicont}, it holds for every $s,t \in \Q \cap [0,+\infty)$:
\begin{align*}
\dst(\rrho^*(s), \rrho^*(t))^2 \leq \liminf_\tau \dst\left(\rrho^\tau(s), \rrho^\tau(t)\right)^2 \leq \liminf_\tau 2 \me_\eps(\hat\rrho)\left(|s-t| + \tau\right) = 2 \me_\eps(\hat\rrho) |s-t|.
\end{align*}
Hence, the function $\rrho^*: \Q \cap [0,+\infty) \to \sspc$ is Hölder-$\frac{1}{2}$-continuous. Since $(\sspc, \dst)$ is a complete metric space, there exists a unique (Hölder-$\frac{1}{2}$-)continuous extension $\rrho^*: [0, +\infty) \to \sspc$. The claim that $\rho_j^\tau(t) \weakto \rho_j^*(t)$ weakly at every $t \geq 0$ now follows directly from weak compactness, since $\rho_j^\tau(t) \weakto \rho_j \in \mptrd$ implies by the same arguments for every $s \in \Q \cap [0,+\infty)$ the estimate
\begin{align*}
\dst(\rrho, \rrho^*(s))^2 \leq 2 \me_\eps(\rrho^0) |s-t|
\end{align*}
and thus $\rrho = \rrho^*(t)$ by continuity of $\rrho^*$. This finishes the proof.
\end{proof}

In the following theorem, we prove that along some further subsequence, we have $\rrho^\tau \to \rrho^*$ in some stronger topology than proven in the previous lemma, which in combination with Corollary \ref{cor:weaksol_tau_error} allows us to show that the limit curve $\rrho^*$ satisfies the weak formulation of the system of evolution equations stated in Theorem \ref{thm:main}.

\begin{theorem} \label{thm:cvgce_to_weak_sol}
Let $\rrho^*: [0, +\infty) \to \sspc$ be the limit curve constructed in the previous lemma. There exists a further subsequence $\tau \downarrow 0$ along which
\begin{equation} \label{eq:cvgces_tau}
\begin{split}
\left.
\begin{array}{c}
\rho_j^\tau \to \rho_j^* \\ [1ex]
F_j'(\rho_j^\tau) \to F_j'(\rho_j^*)
\end{array}
\right\}
\quad \text{ strongly in } L^2_{loc}([0, +\infty) \times \Rd), \\
\left.
\begin{array}{c}
\nabla F_j'(\rho_j^\tau) \weakto \nabla F_j'(\rho_j^*) \\[1ex]
\nabla \partial_{r_j} h(\rho_1^\tau, \rho_2^\tau) \weakto\nabla \partial_{r_j} h(\rho_1^*, \rho_2^*)\\[1ex]
\nabla K \ast \rho_j^\tau \weakto \nabla K \ast \rho_j^*
\end{array}
\right\}
\quad \text{ weakly in } L^2_{loc}([0, +\infty) \times \Rd).
\end{split}
\end{equation}
The limit curve $\rrho^*$ is a weak solution to the system of evolution PDEs \eqref{eq:system-intro}. More precisely, for every $\zeta \in C^\infty_c([0, +\infty) \times \Rd)$, it holds
\begin{equation} \label{eq:weak_pde_lim}
\begin{split}
\int_0^\infty\intrd \left[  \rho_1^*\, \partial_t \zeta - \rho_1^*\,\nabla\left[F_1'(\rho_1^*) + \eps \partial_{r_1}h(\rho_1^*, \rho_2^*) + K \ast \rho_2^*\right]\cdot\nabla \zeta \right]\,\intd x\intd t + \intrd \rho_1^0 \zeta_0\,\intd x &= 0 \\
\int_0^\infty\intrd \left[  \rho_2^*\, \partial_t \zeta - \rho_2^*\,\nabla\left[F_2'(\rho_2^*) + \eps \partial_{r_2}h(\rho_1^*, \rho_2^*) + K \ast \rho_1^*\right]\cdot\nabla \zeta \right]\,\intd x\intd t + \intrd \rho_2^0 \zeta_0\,\intd x &= 0.
\end{split}
\end{equation}
\end{theorem}
\begin{proof}
We fix some $T > 0$ and some smooth open bounded subset $\Omega \subset \Rd$, and prove the existence of a subsequence $\tau \downarrow 0$ such that the set of convergences \eqref{eq:cvgces_tau} holds in $L^2([0,T]\times\Omega)$. Since $T$ and $\Omega$ are arbitrary, the $L^2_{loc}$-convergence then follows directly.

We start by integrating \eqref{eq:l2bound_gradF_gen} from $0$ to $T$ to obtain a $\tau$-independent $L^2([0,T]\times\Omega)$-bound on $\nabla F_j'(\rho_j^\tau)$. With $N = \left\lceil T/\tau - 1 \right\rceil$, we have by non-negativity of $\me_\eps$ and the inequality \eqref{eq:entropy_lower_bd}:
\begin{equation} \label{eq:gradfjtau_bound}
\begin{split}
&\sum_{j = 1,2} \int_0^T \int_\Omega \left| \nabla F_j'(\rho_j^\tau)\right|^2\,\intd x \intd t \leq \sum_{n=0}^N \int_{n\tau}^{(n+1)\tau} \sum_{j = 1,2} \intrd \left| \nabla F_j'(\rho_{j,n}^+)\right|^2\,\intd x \intd t \\
&\leq  \sum_{n=0}^N \tau C\, \left(\frac{\me_\eps(\hat\rrho_n) - \me_\eps(\rrho_n^+)}{\tau} + \frac{2 C_K^2}{\lambda}\, \me_\eps(\hat\rrho_n) + \frac{\mh_c(\hat\rrho_n) - \mh_c(\rrho_n^+)}{\tau} + 2 d C_K \right) \\
&\leq C \left( \me_\eps(\rrho^0) - \me_\eps(\rrho_n^+) + \left(N + 1\right)\tau\frac{2C_K^2}{\lambda} \me_\eps(\rrho^0) + \mh_c(\rrho^0) - \mh_c(\rrho_n^+) + 2dC_K\tau \right) \\
&\leq C\left(\left(1 + (T + \tau)\frac{2C_K^2}{\lambda}\right) \,\me_\eps(\rrho^0) + \mh_c(\rrho^0) + \mom_2[\rho_{1,n}^+] + \mom_2[\rho_{2,n}^+] - 2 + 2\pi^{d/2} + 2dC_K\right) \\
&\leq \hat C\, \left(\me_\eps(\rrho^0) + \mh_c(\rrho^0) + 1\right)
\end{split}
\end{equation}
for some constant $\hat C$ independent of $\tau > 0$ sufficiently small. At the end, we have used $\mom_2[\rho_{1,n}^+] + \mom_2[\rho_{2,n}^+] \leq \frac{2}{\lambda} \me_\eps(\rrho_n^+) \leq \frac{2}{\lambda} \me_\eps(\rrho^0)$, which follows from \eqref{eq:estim_fj_int_eeps} and \eqref{eq:estim_inter_2mom}.

The Poincaré inequality in $\Omega$ yields a $\tau$-independent constant $C_\Omega$ such that
\begin{equation} \label{eq:fjtau_poinc}
\int_0^T \int_\Omega F_j'(\rho_j^\tau)^2\,\intd x\intd t \leq C_\Omega \int_0^T \left( \int_\Omega \left|\nabla F_j'(\rho_j^\tau)\right|^2\,\intd x + \left( \int_\Omega F_j'(\rho_j^\tau)\,\intd x \right)^2 \right)\,\intd t.
\end{equation}
The first integral on the right is uniformly bounded in $\tau$ as seen above, while the second one is uniformly bounded by using \eqref{eq:estim_fjp_eeps} and the uniform bound $\me_\eps(\rrho^\tau(t)) \leq \me_\eps(\rrho^0)$.
Hence, the whole expression on the right-hand side of \eqref{eq:fjtau_poinc} is uniformly bounded in $\tau$. Together, \eqref{eq:gradfjtau_bound} and \eqref{eq:fjtau_poinc} imply existence of a $\tau$-independent constant $M < +\infty$ such that
\begin{equation} \label{eq:fjtau_h1bound}
\int_0^T \|F_j'(\rho_j^\tau)\|_{H^1(\Omega)}^2 \,\intd t \leq M
\end{equation}
for all $\tau \in (0, \bar\tau]$ for some $\bar\tau > 0$. We shall use this estimate together with the generalized version \cite[Theorem 2]{RS} of the Aubin-Lions theorem to prove existence of a subsequence $\tau \downarrow 0$ along which $\rho_j^\tau(t) \to \rho_j^*(t)$ strongly in $L^2(\Omega)$ at almost every $t \in [0,T]$. Specifically, we construct maps $\mf: L^2(\Omega) \to [0, +\infty]$ and $g: L^2(\Omega) \times L^2(\Omega) \to [0, +\infty]$ such that $\mf$ is lower semicontinuous and coercive and $g$ is lower semicontinuous, both with respect to the strong topology on $L^2(\Omega)$, and it holds
\begin{gather} \label{eq:Fnormcoerciveint}
\int_0^T \mf(\rho_j^\tau(t))\,\intd t \leq M \quad \text{ for all } \tau \in (0, \bar\tau], \\ \label{eq:gcompmap}
\lim_{h \downarrow 0} \sup_{\tau \in (0, \bar\tau]} \int_0^{T-h} g\left(\rho_j^\tau(t+h), \rho_j^\tau(t)\right)\,\intd t = 0.
\end{gather}
Since \eqref{eq:fjtau_h1bound} and the at least linear growth of $F_j'$ imply $\rho_j^\tau(t) \in L^2(\Omega)$ at almost every $t \in [0,T]$, the expressions in \eqref{eq:Fnormcoerciveint} and \eqref{eq:gcompmap} are well-defined for every $\tau \in (0, \bar\tau]$. We now claim that
\begin{align*}
\mf(u) &:= \left\{
\begin{array}{ll}
\|F_j'(u)\|_{H^1(\Omega)}^2 & \text{ if } F_j'(u) \in H^1(\Omega) \\
+\infty & \text{ otherwise}
\end{array}
\right.,
 \\
g(u,v) &:= \inf\left\{\wass(\rho, \eta)^2: \rho, \eta \in \mptrd, \ \mom_2[\rho],\mom_2[\eta] \leq \frac{2}{\lambda}\me_\eps(\rrho^0),\ \rho|_\Omega = u, \eta|_\Omega = v\right\}
\end{align*}
fulfill all desired properties. \eqref{eq:Fnormcoerciveint} is just \eqref{eq:fjtau_h1bound} by definition of $\mf$. To prove \eqref{eq:gcompmap}, we use that $\mom_2[\rho_j^\tau(s)] \leq \frac{2}{\lambda} \me_\eps(\rrho^0)$ for all $\tau > 0$ and $s \geq 0$. By definition of $g$, this implies
\begin{align*}
g\left(\rho_j^\tau(t+h), \rho_j^\tau(t)\right) \leq \wass\left(\rho_j^\tau(t+h), \rho_j^\tau(t)\right)^2 \leq \dst\left(\rrho^\tau(t+h), \rrho^\tau(t)\right)^2
\end{align*}
for all $h > 0$ and $t \in [0, T-h]$. Together with estimate \eqref{eq:rtau_quasicont}, this yields for every fixed $\hat\tau \in (0, \bar\tau]$ the estimate
\begin{align*}
\sup_{\tau \in (0, \hat\tau]} \int_0^{T-h} g\left(\rho_j^\tau(t+h), \rho_j^\tau(t)\right)\,\intd t \leq \sup_{\tau \in (0, \hat\tau]} \int_0^{T-h} \dst\left(\rrho^\tau(t+h), \rrho^\tau(t)\right)^2\,\intd t \leq 2T \me_\eps(\rrho^0)(h + \hat\tau)
\end{align*}
Additionally, we prove that for every $\hat\tau \in (0, \bar\tau]$, it holds
\begin{align*}
\lim_{h \downarrow 0} \sup_{\tau \in [\hat\tau, \bar\tau]} \int_0^{T-h} g\left(\rho_j^\tau(t+h), \rho_j^\tau(t)\right)\,\intd t = 0.
\end{align*}
Indeed, for every $\tau \in [\hat\tau, \bar\tau]$ and every $h < \hat\tau$, it follows from the fact that $\rho_j^\tau$ is constant on intervals of length $\tau > h$ and inequality \eqref{eq:step_estimate} that
\begin{align*}
\int_0^{T-h} g\left(\rho_j^\tau(t+h), \rho_j^\tau(t)\right)\,\intd t \leq \sum_{n = 1}^{\left\lceil T/\tau \right\rceil} h\dst(\hat\rrho_n, \rrho_{n}^+)^2 \leq 2\tau h \sum_{n=1}^{\left\lceil T/\tau \right\rceil} \left(\me_\eps(\hat\rrho_n) - \me_\eps(\rrho_n^+)\right) \leq 2\bar\tau h \me_\eps(\rrho^0).
\end{align*}
Together, these two results yield \eqref{eq:gcompmap}, since $\hat\tau$ can be arbitrarily small:
\begin{align*}
\limsup_{h \downarrow 0} \sup_{\tau \in (0, \bar\tau]} \int_0^{T-h} g\left(\rho_j^\tau(t+h), \rho_j^\tau(t)\right)\,\intd t \leq \lim_{h \downarrow 0} 2T \me_\eps(\rrho^0)(h + \hat\tau) = 2T \me_\eps(\rrho^0)\hat\tau.
\end{align*}

In order to prove lower semicontinuity of $g$ with respect to strong $L^2(\Omega)$-convergence, and compatibility of $g$ in the sense that $g(u, v) = 0\ \Rightarrow\ u = v$ for $u,v \in L^2(\Omega)$, we first observe that the infimum in the definition of $g$ is attained as a minimum, unless $g(u,v) = +\infty$. To see this, note that in this case, the set
\begin{align*}
\left\{\rho, \eta \in \mptrd: \ \mom_2[\rho],\mom_2[\eta] \leq \frac{2}{\lambda}\me_\eps(\rrho^0),\ \rho|_\Omega = u, \eta|_\Omega = v\right\}
\end{align*}
is nonempty and compact with respect to narrow convergence, since pre-compactness follows from the bound on the second moments of $\rho$ and $\eta$, and closedness from the fact that $\Omega$ is an open set and the map $\rho \mapsto \mom_2[\rho]$ is lower semicontinuous with respect to narrow convergence. This implies attainment of the minimum by lower semicontinuity of the Wasserstein distance in the narrow topology.

With this observation, the property $g(u, v) = 0\ \Rightarrow\ u = v$ is clear. To prove lower semicontinuity, let $u_n \to u$ and $v_n \to v$ be two convergent sequences in $L^2(\Omega)$. Without loss of generality, we can assume $g(u_n, v_n) < +\infty$ for all $n$ and $g(u_n, v_n) \to C$ with some $C < +\infty$. Hence for every $n$, there are $\rho_n, \eta_n \in \mptrd$ such that $\mom_2[\rho_n], \mom_2[\eta_n] \leq \frac{2}{\lambda}\me_\eps(\rrho^0)$, $\rho_n|_\Omega = u_n$, $ \eta_n|_\Omega = v_n$ and $g(u_n, v_n) = \wass(\rho_n, \eta_n)^2$. By the same compactness argument as above, it holds $\rho_n \to \rho$ and $\eta_n \to \eta$ narrowly up to subsequence, where $\rho$ and $\eta$ are admissible in the minimization problem in the definition of $g(u,v)$. Hence
\begin{align*}
g(u,v) \leq \wass(\rho, \eta)^2 \leq \liminf_{n \to \infty} \wass(\rho_n, \eta_n)^2 = \lim_{n \to \infty} g(u_n, v_n) = C,
\end{align*}
proving lower semicontinuity of $g$.

To show coercivity of the functional $\mf$, observe that for any sequence $(u_n) \subset L^2(\Omega)$ along which $\mf(u_n)$ is uniformly bounded for all $n$, the sequence $(F_j'(u_n))$ is weakly compact in $H^1(\Omega)$, hence by Rellich's theorem, $F_j'(u_n) \to v$ strongly in $L^2(\Omega)$ up to subsequence. This implies $u_n \to u := (F_j')^{-1}(v)$ strongly in $L^2(\Omega)$ along the same subsequence by the at least linear growth of $F_j'(r)$ for large $r$, as seen in the proof of Theorem \ref{thm:discrete_error}, thus showing compactness of sublevel sets of $\mf$.

In order to prove lower semicontinuity for $\mf$, let $u_n \to u$ in $L^2(\Omega)$. Since without loss of generality, we can assume that $(\mf(u_n))$ is bounded, we again have that $(F_j'(u_n))$ is bounded and thus weakly compact in $H^1(\Omega)$. Together with the $L^2$-convergence $u_n \to u$, this implies $F_j'(u_n) \weakto F_j'(u)$ weakly in $H^1(\Omega)$. The claim thus follows from lower semicontinuity of the norm with respect to weak convergence.

All together, this shows that we can apply \cite[Theorem 2]{RS} to the sequence of functions $\rho_j^\tau \in L^2([0,T]; L^2(\Omega))$ to obtain that along a subsequence $\tau \downarrow 0$, it holds $\rho_j^\tau(t) \to \rho_j^*(t)$ strongly in $L^2(\Omega)$ at almost every $t \geq 0$.

In order to get from this the claimed convergence $\rho_j^\tau \to \rho_j^*$ in $L^2([0,T] \times \Omega)$, which is equivalent to $\rho_j^\tau \to \rho_j^*$ strongly in $L^2([0,T]; L^2(\Omega))$, we additionally show that the family of functions $[0,T] \ni t \mapsto \|F_j'(\rho_j^\tau(t))\|_{L^2(\Omega)}^2$ is equi-integrable. It follows from the Sobolev embedding theorem that there exists an exponent $2^* > 2$ and a constant $C_\Omega$ such that $\|u\|_{L^{2^*}(\Omega)} \leq C_\Omega \|u\|_{H^1(\Omega)}$ for all $u \in H^1(\Omega)$. We denote
\begin{align*}
r := \frac{2^*}{2} > 1, \quad r' := \frac{r}{r-1}, \quad q := \frac{3r-1}{2r} > 1.
\end{align*}
Notice that $2(q-1)r'=1$. We use Hölder twice to obtain at any $t \geq 0$ with $u := F_j'(\rho_j^\tau(t))$:
\begin{align*}
\|u\|_{L^2(\Omega)}^{2q} &= \left( \int_\Omega 1 \cdot u^2\,\intd x \right)^q \leq |\Omega|^{q - 1}  \int_\Omega u^{2q}\,\intd x = |\Omega|^{q - 1}  \int_\Omega u^2 u^{2(q-1)} \,\intd x \\
&\leq |\Omega|^{q - 1} \left(\int_\Omega u^{2r}\,\intd x\right)^\frac{1}{r}\left(\int_\Omega u^{2(q-1)r'}\,\intd x \right)^\frac{1}{r'} = |\Omega|^{q-1} \|u\|_{L^{2^*}(\Omega)}^2 \|u\|_{L^1(\Omega)}^\frac{1}{r'} \\
&\leq |\Omega|^{q - 1} C_\Omega^2  \left(\alpha_j(1 + 2 \me_\eps(\rrho^0))\right)^\frac{1}{r'}\,\|u\|_{H^1(\Omega)}^2 =:  C \|u\|_{H^1(\Omega)}^2.
\end{align*}
In the last estimate, we used the Sobolev inequality and \eqref{eq:estim_fjp_eeps}. With \eqref{eq:fjtau_h1bound}, this yields
\begin{align*}
\int_0^T \|F_j'(\rho_j^\tau)\|_{L^2(\Omega)}^{2q} \,\intd t \leq C \int_0^T \|F_j'(\rho_j^\tau)\|_{H^1(\Omega)}^{2} \,\intd t \leq C M
\end{align*}
for all $\tau > 0$. Since $q > 1$, this implies equi-integrability of $\|F_j'(\rho_j^\tau)\|_{L^2(\Omega)}^2$ which also implies equi-integrability of $\|\rho_j^\tau\|_{L^2(\Omega)}^2$ by the at least linear growth of $F_j'$. Together with the pointwise almost everywhere convergence $\rho_j^\tau(t) \to \rho_j^*(t)$ in $L^2(\Omega)$, Vitali's convergence theorem thus yields $\rho_j^\tau \to \rho_j^*$ strongly in $L^2([0,T] \times \Omega)$.

The convergence $F_j'(\rho_j^\tau) \to F_j'(\rho_j^*)$ in $L^2([0,T] \times \Omega)$ also follows immediately, since pointwise almost everywhere convergence along a subsequence follows from continuity of $F_j'$ and the convergence $\rho_j^\tau \to \rho_j^*$ in $L^2([0,T] \times \Omega)$, and equi-integrability has been shown above.

By the bound \eqref{eq:gradfjtau_bound}, the sequence $\nabla F_j'(\rho_j^\tau)$ is uniformly bounded in $\tau$ in $L^2([0,T] \times \Omega; \Rd)$. Alaoglu's theorem thus yields $\nabla F_j'(\rho_j^\tau) \weakto v \in L^2([0,T] \times \Omega; \Rd)$ weakly. Now it holds for every $\psi \in C^\infty_c([0,T] \times \Omega; \Rd)$, using the $L^2([0,T] \times \Omega)$-convergence $F_j'(\rho_j^\tau) \to F_j'(\rho_j^*)$:
\begin{align*}
&\int_0^T \int_\Omega F_j'(\rho_j^*) \ \dv \, \psi\,\intd x\intd t = \lim_\tau \int_0^T \int_\Omega F_j'(\rho_j^\tau) \ \dv \, \psi\,\intd x\intd t \\
&= -\lim_\tau \int_0^T \int_\Omega \nabla F_j'(\rho_j^\tau) \cdot \psi\,\intd x = -\int_0^T \int_\Omega v \cdot \psi\,\intd x,
\end{align*}
implying $v = \nabla F_j'(\rho_j^\tau)$.

The weak $L^2([0,T] \times \Omega)$-convergence $\nabla \partial_{r_j} h(\rho_1^\tau, \rho_2^\tau) \weakto\nabla \partial_{r_j} h(\rho_1^*, \rho_2^*)$ follows from the strong $L^2$-convergence $F_j'(\rho_j^\tau) \to F_j'(\rho_j^*)$ and assumption \eqref{eq:thetabdd_swap} by the same proof as in Theorem \ref{thm:discrete_error}. It thus remains to prove $\nabla K \ast \rho_j^\tau \weakto \nabla K \ast \rho_j^*$ weakly in $L^2([0,T] \times \Omega; \Rd)$. Let $\varphi \in L^2([0,T] \times \Omega; \Rd)$ be arbitrary. Again as in the proof of Theorem \ref{thm:discrete_error}, the narrow convergence $\rho_j^\tau(t) \to \rho_j^*(t)$ in $\Rd$ at every $t$ and boundedness of the second moments imply
\begin{align*}
\int_\Omega \varphi_t \cdot \nabla K \ast \rho_j^\tau(t)\,\intd x \to \int_\Omega \varphi_t \cdot \nabla K \ast \rho_j^*(t)\,\intd x 
\end{align*}
at every $t \in [0,T]$ where $\varphi_t = \varphi(t, \cdot) \in L^2(\Omega)$, which holds at almost every $t$. The estimates $|\nabla K (x-y)|^2 \leq C_K^2 |x-y|^2 \leq 2C_K^2(|x|^2 + |y|^2)$ and $\mom_2[\rho_j^\tau] \leq \frac{2}{\lambda} \me_\eps(\rrho^0)$ further yield
\begin{align*}
\left| \int_\Omega \varphi_t \cdot \nabla K \ast \rho_j^\tau(t)\,\intd x \right| &\leq \|\varphi_t\|_{L^2(\Omega)} \left(\int_\Omega \intrd \left| \nabla K(x-y) \right|^2 \rho_j^\tau(y) \,\intd y \intd x  \right)^{\frac{1}{2}} \\
&\leq \|\varphi_t\|_{L^2(\Omega)}  \sqrt 2\, C_K \left(\frac{2}{\lambda}|\Omega|\, \me_\eps(\rrho^0) + \int_\Omega |x|^2\, \intd x  \right)^{\frac{1}{2}}.
\end{align*}
which is integrable from $0$ to $T$ since $\varphi \in L^2([0,T] \times \Omega)$. Hence by dominated convergence
\begin{align*}
\int_0^T \int_\Omega \varphi \cdot \nabla K \ast \rho_j^\tau \,\intd x \intd t \to \int_0^T \int_\Omega \varphi \cdot \nabla K \ast \rho_j^* \,\intd x \intd t,
\end{align*}
proving weak $L^2$-convergence $\nabla K \ast \rho_j^\tau \weakto \nabla K \ast \rho_j^*$ in $[0,T] \times \Omega$.

The claim that $\rrho^*$ solves the weak form \eqref{eq:weak_pde_lim} of the evolution equations now follows directly from the set of convergences \eqref{eq:cvgces_tau} and Corollary \ref{cor:weaksol_tau_error}: By the regularity of $\zeta$, it holds $\nabla\zeta^\tau \to \nabla \zeta$ uniformly in $[0, \infty) \times \Rd$, and since $\zeta$ has compact support in $[0, \infty) \times \Rd$, the local convergences shown in \eqref{eq:cvgces_tau} are sufficient to prove convergence of the integrals on the left-hand side of \eqref{eq:weaksol_tau_error}.
\end{proof}

\section{Exponential convergence to equilibrium} \label{sec:exp_cvgce}
In this section, we prove that the weak solution $\rrho^*$ to \eqref{eq:system-intro} constructed above converges exponentially in time to a global minimizer $\bar\rrho = (\bar\rho_1, \bar\rho_2) \in \sspc$ of the energy functional $\me_\eps$. En passant, it follows that this minimizer is unique.

As explained in the introduction, we split the energy as follows:
\[ \me_\eps(\rrho)=\ml_\eps(\rrho)+\eps \mn_\eps(\rrho) \] 
with functionals $\me_\eps$ and $\mn_\eps$ given in \eqref{eq:L_def} and \eqref{eq:N_def}, respectively. For the definition of the auxiliary potentials $V_1$ and $V_2$ by \eqref{eq:V_def}, any of the a priori many minimizers $\bar\rrho$ of $\me_\eps$ can be used.
We shall prove that $\ml_\eps$ is geodesically $\lambda_\eps$-convex with some $0 < \lambda_\eps < \lambda$, and that $\bar\rrho$ is its unique global minimizer over $\sspc$. Furthermore, we show that it defines a Lyapunov functional for the minimizing movement scheme from the previous section, which will lead to exponential convergence to the equilibrium.

\subsection{Geodesic convexity of the auxiliary functional}
We start by proving  geodesic $\lambda$-convexity of the interaction energy.
\begin{lemma} \label{lem:inter_geod_conv}
The interaction energy $\rrho \mapsto \intrd \rho_1 K \ast \rho_2\,\intd x$ is geodesically $\lambda$-convex on $\sspc$.
\end{lemma}
\begin{proof}
 This is a special case of Zinsl's result \cite{Zinsl} on uniform geodesic convexity of $n$-component interaction functionals. In the case at hand, the proof simplifies as follows:

Take any $\rrho^0, \rrho^1 \in \sspc$. For simplicity, we assume that $\rrho^0$ and $\rrho^1$ are absolutely continuous with respect to the Lebesgue measure. The general case then follows by density arguments and weak lower semi-continuity of the interaction energy, see e.g. \cite[Prop. 9.2.10]{AGS}. Denote by $T_j$ the optimal transport maps pushing $\rho^0_j$ to $\rho^1_j$, and denote by $(\rrho^s)_{0 \leq s \leq 1}$ the geodesic between $\rrho^0$ and $\rrho^1$, given by $\rho^s_j = T_j^s\#\rho_j^0$ with $T_j^s = (1-s)\mathrm{id} + s T_j$. From the properties of the push-forward and $\lambda$-convexity of $K$, we obtain for the interaction energy of $\rrho^s$:
\begin{align*}
\intrd \rho^s_1 K \ast \rho_2^s\,\intd x &= \iintrdrd \rho_1^0(x) \rho_2^0(y) \,K\left(T_1^s(x) - T_2^s(y)\right)\,\intd x \intd y \\
&= \iintrdrd \rho_1^0(x) \rho_2^0(y) \,K\left((1-s)(x-y) + s(T_1(x) - T_2(y)\right)\,\intd x \intd y \\
&\leq (1-s)\intrd \rho_1^0 K \ast \rho_2^0 \,\intd x + s \intrd \rho_1^1 K \ast \rho_2^1 \,\intd x \\ &- \frac{\lambda}{2}s(1-s)\iintrdrd \rho_1^0(x) \rho_2^0(y) \,|x-y-T_1(x)+T_2(y)|^2\,\intd x \intd y.
\end{align*}
For the last integral, $\wass(\rho^0_j,\rho^1_j)^2 = \intrd \rho_j^0\, |x-T_j(x)|^2\,\intd x$ and $\rrho^0, \rrho^1 \in \sspc$ yield
\begin{align*}
&\iintrdrd \rho_1^0(x) \rho_2^0(y) \,|x-y-T_1(x)+T_2(y)|^2\,\intd x \intd y \\
&= \dst(\rrho^0, \rrho^1)^2 - 2\left(\intrd \rho_1^0(x)\, (x-T_1(x))\,\intd x\right)\cdot \left(\intrd \rho_2^0(y)\,(y-T_2(y)) \,\intd y\right) \\
&= \dst(\rrho^0, \rrho^1)^2 - 2 \left(\mom_1[\rho_1^0] - \mom_1[\rho_1^1]\right)\cdot\left(\mom_1[\rho_2^0] - \mom_1[\rho_2^1]\right) \\
&= \dst(\rrho^0, \rrho^1)^2 + 2\left|\mom_1[\rho_1^0] - \mom_1[\rho_1^1]\right|^2 \geq \dst(\rrho^0, \rrho^1)^2.
\end{align*}
Inserting above, this proves the inequality
\begin{align*}
\intrd \rho^s_1 K \ast \rho_2^s\,\intd x \leq (1-s)\intrd \rho_1^0 K \ast \rho_2^0 \,\intd x + s \intrd \rho_1^1 K \ast \rho_2^1 \,\intd x - \frac{\lambda}{2}s(1-s)\,\dst(\rrho^0, \rrho^1)^2,
\end{align*}
which is geodesic $\lambda$-convexity of the interaction energy.
\end{proof}
We assume in all of the following that $\eps$ is
small enough such that $\eps K_0 < \lambda$ with the constant $K_0$ from Corollary \ref{cor:vj_semiconv}.
\begin{corollary} \label{cor:l_conv_min}
The functional $\ml_\eps$ is $\lambda_\eps$-convex along geodesics in $\sspc$ with the positive modulus $\lambda_\eps = \lambda - \eps K_0$.
\end{corollary}
\begin{proof}
The $F$-terms are geodesically convex, since the $F_j$ satisfy the McCann condition \eqref{eq:mccann_cond}. Since $V_j$ is $(-K_0)$-semiconvex in $\Rd$, the functionals $\eps\int \rho_j V_j$ are geodesically $(-\eps K_0)$-semiconvex, which then holds for their sum as well, since they only depend on one of the components each. Together with the previous lemma, this implies the $\lambda_\eps$-convexity of $\ml_\eps$ in $\sspc$.
\end{proof}
As a direct consequence of geodesic $\lambda_\eps$-convexity, we prove the following technical inequality for $\ml_\eps$:

\begin{lemma} \label{lem:above_tangent}
Let $\rrho^0, \rrho^1 \in \sspc$ be pairs and assume that both components of $\rrho^0$ are absolutely continuous with respect to the Lebesgue measure. Additionally, we assume that $\nabla F_j'(\rho_j^0)$ exists almost everywhere in $\Rd$ and lies in $L^2(\intd \rho_j^0)$. Then, denoting by $T_j$ the (unique) optimal transport maps pushing $\rho_j^0$ to $\rho_j^1$, it holds
\begin{equation} \label{eq:above_tangent}
\begin{split}
\ml_\eps(\rrho^1) - \ml_\eps(\rrho^0) &\geq \intrd \rho_1^0\, \nabla\left[F_1'(\rho_1^0) + \eps V_1 + K \ast \rho_2^0\right] \cdot (T_1(x) - x) \,\intd x \\
 &+ \intrd \rho_2^0\, \nabla\left[F_2'(\rho_2^0) + \eps V_2 + K \ast \rho_1^0\right] \cdot (T_2(x) - x)\,\intd x + \frac{\lambda_\eps}{2} \dst(\rrho^1, \rrho^0)^2.
\end{split}
\end{equation}
\end{lemma}
\begin{proof}
Denote by $(\rrho^s)_{0 \leq s \leq 1}$ the geodesic between $\rrho^0$ and $\rrho^1$, as in the proof of Lemma \ref{lem:inter_geod_conv}. Since $\ml_\eps$ is geodesically $\lambda_\eps$-convex, there holds the above-tangent formula
\begin{equation} \label{eq:ab_tan_general}
\ml_\eps(\rrho^1) - \ml_\eps(\rrho^0) \geq \ddszerop^+ \ml_\eps(\rrho^s) + \frac{\lambda_\eps}{2} \dst(\rrho^1, \rrho^0)^2,
\end{equation}
see \cite[Proposition 5.29]{Vill}. To compute the derivative, we use \cite[Theorem 5.30]{Vill} to obtain
\begin{align*}
\ddszerop^+ \intrd \rho_j^s V_j\,\intd x \geq \intrd \rho_1^0\, \nabla V_j \cdot (T_j(x)-x)\,\intd x,
\end{align*}
using that the $V_j$ are $(-K_0)$-semiconvex in $\Rd$. For the $F_j$-terms, we apply \cite[Theorem 10.4.6]{AGS}: Using our assumption $\nabla F_j'(\rho_j^0) \in L^2(\intd \rho_j^0)$, it follows that $\nabla F_j'(\rho_j^0)$ lies in the subdifferential of the geodesically convex functional $\rho_j \mapsto \intrd F_j(\rho_j)\,\intd x$ at $\rho_j^0$. Thus,
\begin{align*}
\ddszerop^+ \intrd F_j(\rho_j^s)\,\intd x \geq \intrd \rho_j^0\,\nabla F_j'(\rho_j^0)\cdot \left(T_j(x) - x\right)\,\intd x.
\end{align*}
For the interaction term, we use $\rho_j^s = T_j^s \# \rho_j^0$ to obtain
\begin{align*}
&\intrd \left[ \rho_1^s K \ast \rho_2^s - \rho_1^0 K \ast \rho_2^0 \right]\,\intd x = \iintrdrd \rho_1^0(x) \rho_2^0(y)\, \left[ K\left(T_1^s(x) - T_2^s(y)\right) - K(x-y) \right]\,\intd x \intd y \\
&= \iintrdrd \rho_1^0(x) \rho_2^0(y) \int_0^1 \frac{\intd }{\intd t}\, K\left(T_1^{st}(x) - T_2^{st}(y)\right) \,\intd t\,\intd x \intd y \\
& = s \iintrdrd \rho_1^0(x) \rho_2^0(y) \int_0^1 \nabla K\left(T_1^{st}(x) - T_2^{st}(y)\right)\cdot \left(T_1(x) - x + y - T_2(y)\right)\,\intd t \,\intd x \intd y
\end{align*}
To apply dominated convergence after dividing by $s$, observe that
\begin{align*}
|\nabla K\left(T_1^{st}(x) - T_2^{st}(y)\right)| \leq C_K |T_1^{st}(x) - T_2^{st}(y)| \leq C_K \left(|x-y| + |T_1(x) - x| + |T_2(y) - y|\right)
\end{align*}
for every $s$ and $t$, and that $\iint \rho_1^0 \rho_2^0\, |T_j - \mathrm{id}|^2 \leq \dst(\rrho^0, \rrho^1)^2 < +\infty$, and that the $\rho_j^0$ have bounded second moments. Dominated convergence and the continuity of $\nabla K$ thus yield
\begin{align*}
\ddszerop \intrd \rho_1^s K \ast \rho_2^s \,\intd x &= \iintrdrd \rho_1^0(x) \rho_2^0(y) \nabla K(x-y)\cdot \left(T_1(x) - x + y - T_2(y)\right)\,\intd x \intd y \\
&= \intrd \left[ \rho_1^0\, \nabla K \ast \rho_2^0 \cdot (T_1(x) - x) + \rho_2^0\, \nabla K \ast \rho_1^0 \cdot (T_2(x) - x) \right]\,\intd x.
\end{align*}
Inserting these derivatives into \eqref{eq:ab_tan_general} yields \eqref{eq:above_tangent}.
\end{proof}
\begin{corollary} \label{cor:l_min_rhobar}
With the global minimizer $\bar\rrho \in \sspc$ of $\me_\eps$ chosen in the construction of $\ml_\eps$, every $\rrho \in \sspc$ satisfies the inequality
\begin{equation} \label{eq:l_conv_estim}
\ml_\eps(\rrho) - \ml_\eps(\bar\rrho) \geq \frac{\lambda_\eps}{2}\dst(\rrho, \bar\rrho)^2.
\end{equation}
In particular, $\bar\rrho$ is the unique global minimizer of $\ml_\eps$ over $\sspc$.
\end{corollary}
\begin{proof}
Take any $\rrho \in \sspc$. We apply the previous lemma with $\rrho^0 = \bar\rrho$ and $\rrho^1 = \rrho$. Note that by Lemma \ref{lem:minim_fj_lip}, $F_j'(\bar\rho_j)$ is globally Lipschitz, so in particular $\nabla F_j'(\bar\rho_j) \in L^2(\intd \bar\rho_j)$, hence the previous lemma is applicable. By Lemma \ref{lem:el_cineq}, the gradient of $F_j'(\bar\rho_j) + \eps\partial_{r_j} h(\bar\rho_1, \bar\rho_2) + K \ast \bar\rho_{j'}$ vanishes on the set $\{\bar\rho_j > 0\}$, hence the integrals in \eqref{eq:above_tangent} vanish, yielding \eqref{eq:l_conv_estim}. The second claim follows immediately by non-negativity of the distance.
\end{proof}

In order to analyze the functionals $\ml_\eps$ and $\mn_\eps$ further, we derive the following representation for the difference $\ml_\eps(\rrho) - \ml_\eps(\bar\rrho)$:
\begin{lemma}
For any $\rrho \in \sspc$ which is absolutely continuous with respect to the Lebesgue measure, it holds
\begin{equation} \label{eq:ldiff_rep}
\ml_\eps(\rrho) - \ml_\eps(\bar\rrho) = \mi_F(\rrho) + \mi_K(\rrho) + \mk(\rrho)
\end{equation}
with the three functionals
\begin{align*}
&\mi_F(\rrho) := \intrd [\dst_{F_1}(\rho_1| \bar\rho_1) + \dst_{F_2}(\rho_2|\bar\rho_2) ] \,\intd x \\
&\mi_K(\rrho):= \intrd [(K \ast \bar\rho_2 - C_1)_+\, \rho_1 + (K \ast \bar\rho_1 - C_2)_+\, \rho_2] \,\intd x \\
&\mk(\rrho) := \intrd (\rho_1 - \bar\rho_1) K \ast (\rho_2 - \bar\rho_2) \,\intd x.
\end{align*}
\end{lemma}
\begin{proof}
 From the definition of $\ml_\eps$ and the fact that $\eps V_j = (C_j - K \ast \bar\rho_{j'})_+ - F_j'(\bar\rho_j)$ by the Euler-Lagrange system \eqref{eq:euler-lagrange}, we obtain
\begin{align} \nonumber
&\ml_\eps(\rrho) - \ml_\eps(\bar\rrho) \\ \nonumber
&= \intrd [F_1(\rho_1) - F_1(\bar\rho_1) + F_2(\rho_2) - F_2(\bar\rho_2) + (\rho_1 - \bar\rho_1) \eps V_1 + (\rho_2 - \bar\rho_2) \eps V_2 \\ \nonumber
&\quad + \rho_1 K \ast \rho_2 - \bar\rho_1 K \ast \bar\rho_2 ]\, \intd x \\ \nonumber
&= \intrd [\dst_{F_1}(\rho_1| \bar\rho_1) + \dst_{F_2}(\rho_2|\bar\rho_2) \\ \label{eq:ldiff_1}
&\quad  + (C_1 - K \ast \bar\rho_2)_+ (\rho_1 - \bar\rho_1) + (C_2 - K \ast \bar\rho_1)_+ (\rho_2 - \bar\rho_2) + \rho_1 K \ast \rho_2 - \bar\rho_1 K \ast \bar\rho_2] \, \intd x.
\end{align}
Using the identity $(C_1 - K \ast \bar\rho_2)_+ = (K \ast \bar\rho_2 - C_1)_+ + C_1 - K \ast \bar\rho_2$, together with the equality of mass of $\rho_1$ and $\bar\rho_1$, and the fact that $\bar\rho_1 \equiv 0$ in $\{K \ast \bar\rho_2 > C_1 \}$ by Lemma \ref{lem:el_rhogr0}, we have
\begin{align*}
\intrd (C_1 - K \ast \bar\rho_2)_+ (\rho_1 - \bar\rho_1) \, \intd x &= \intrd \left[(K \ast \bar\rho_2 - C_1)_+ (\rho_1 - \bar\rho_1) + (C_1 - K \ast \bar\rho_2) (\rho_1 - \bar\rho_1)  \right] \, \intd x \\
&= \intrd \left[ (K \ast \bar\rho_2 - C_1)_+ \rho_1 -\rho_1 K \ast \bar\rho_2 + \bar\rho_1 K \ast \bar\rho_2 \right]\,\intd x
\end{align*}
and similarly for $\intrd (C_2 - K \ast \bar\rho_1)_+ (\rho_2 - \bar\rho_2) \, \intd x$. Inserting into \eqref{eq:ldiff_1} yields
\begin{align*}
&\ml_\eps(\rrho) - \ml_\eps(\bar\rrho) \\ 
&= \intrd [\dst_{F_1}(\rho_1| \bar\rho_1) + \dst_{F_2}(\rho_2| \bar\rho_2) + (K \ast \bar\rho_2 - C_1)_+ \rho_1 + (K \ast \bar\rho_1 - C_2)_+ \rho_2 \\ \nonumber
&\quad + \rho_1 K \ast \rho_2 + \bar\rho_1 K \ast \bar\rho_2 - \rho_1 K \ast \bar\rho_2 - \rho_2 K \ast \bar\rho_1  ] \, \intd x \\
&=  \mi_F(\rrho) + \mi_K(\rrho) + \mk(\rrho),
\end{align*}
proving the claim.
\end{proof}

Note that $\mi_F(\rrho)$ is always non-negative by convexity of $F_1$ and $F_2$, and the same is true for $\mi_K(\rrho)$. However, if $K$ is not the quadratic potential $\frac{\lambda}{2} |\cdot|^2$, the term $\mk(\rrho)$ can be negative. However, there hold the following general estimates:
\begin{lemma}
Take any $\rrho \in \sspc$. There holds the pointwise estimate
\begin{equation} \label{eq:pw_estim_nablaK}
\left| \nabla \left[ K \ast (\bar\rho_j - \rho_j) \right]\right| \leq C_K\, \wass(\rho_j, \bar\rho_j)
\end{equation}
for $j = 1,2$
everywhere in $\Rd$. The functional $\mk$ from the previous lemma satisfies
\begin{equation} \label{eq:estim_intK}
\mk(\rrho) \geq -\frac{C_K}{2} \dst(\rrho, \bar\rrho)^2.
\end{equation}
Moreover, it holds
\begin{equation} \label{eq:I_estim_diffL}
\mi_F(\rrho) + \mi_K(\rrho) \leq \left(1 + \frac{C_K}{\lambda_\eps}\right) (\ml_\eps(\rrho) - \ml_\eps(\bar\rrho)).
\end{equation}
\end{lemma}
\begin{proof}
To prove \eqref{eq:pw_estim_nablaK}, let $T:\Rd \to \Rd$ be the optimal map transporting $\bar\rho_j$ to $\rho_j$, given by Brenier's Theorem. This allows us to write, fixing some $x \in \Rd$:
\begin{align*}
\nabla \left[K \ast (\bar\rho_j - \rho_j) \right](x) &= \intrd \nabla K(x - y) \left(\bar\rho_j(y) - \rho_j(y)\right) \,\intd y \\
&= \intrd \nabla K(x - y) \bar\rho_j(y)\,\intd y - \intrd \nabla K(x - y) (T \# \bar\rho_j)(y)\,\intd y \\
&= \intrd \left[ \nabla K(x-y) - \nabla K(x-T(y)) \right]\, \bar\rho_j(y) \,\intd y
\end{align*}
The global estimate $\| \nabla^2 K \| \leq C_K$ implies $|\nabla K(x-y) - \nabla K(x-T(y))| \leq C_K |y - T(y)|$ for every $y \in \Rd$. Together with Jensen's inequality, this yields
\begin{align*}
\left| \nabla \left[ K \ast (\bar\rho_j - \rho_j) \right](x)\right|^2 &= \left|\intrd \left[ \nabla K(x-y) - \nabla K(x-T(y)) \right]\, \bar\rho_j(y) \,\intd y \right|^2 \\
&\leq \intrd \left| \nabla K(x-y) - \nabla K(x-T(y)) \right|^2\, \bar\rho_j(y) \,\intd y \\
&\leq C_K^2 \intrd |y-T(y)|^2\bar\rho_j(y) \,\intd y = C_K^2\, \wass(\rho_j, \bar\rho_j)^2,
\end{align*}
proving the claimed estimate \eqref{eq:pw_estim_nablaK}. To prove  \eqref{eq:estim_intK}, let $T_1$ be an optimal transport map pushing $\bar\rho_1$ to $\rho_1$. For any $s \in [0,1]$, denote $T_s(x) := (1-s)x + sT_1(x)$. Using $\rho_1 = T_1 \# \bar\rho_1$ and Fubini, which is applicable by boundedness of the second moments of $\rho_j$ and $\bar\rho_j$, we obtain
\begin{align*}
\mk(\rrho) &= \iintrdrd (\rho_1(x) - \bar\rho_1(x))\,(\rho_2(y) - \bar\rho_2(y))\,K(x-y)\,\intd x \intd y \\
&= \iintrdrd \bar\rho_1(x)\, (\rho_2(y) - \bar\rho_2(y))\,[K(T_1(x) - y) - K(x-y)]\, \intd x \intd y \\
&= \iintrdrd \bar\rho_1(x)\, (\rho_2(y) - \bar\rho_2(y)) \left[ \int_0^1 \nabla K(T_s(x) - y) \cdot (T_1(x) - x) \, \intd s \right] \, \intd x \intd y \\
&= \int_0^1 \iintrdrd \bar\rho_1(x)\, (\rho_2(y) - \bar\rho_2(y))\, \nabla K(T_s(x) - y) \cdot (T_1(x) - x)\, \intd x \intd y\, \intd s \\
&= \int_0^1 \intrd \bar\rho_1(x) \, \nabla [K \ast (\rho_2 - \bar\rho_2)](T_s(x)) \cdot (T_1(x) - x)\, \intd x\,\intd s.
\end{align*}
This integral can be estimated with the help of the pointwise estimate \eqref{eq:pw_estim_nablaK} for $\rho_2$:
\begin{align*}
|\mk(\rrho)| &\leq \int_0^1 \intrd \bar\rho_1(x)\, \left|\nabla [K \ast (\rho_2 - \bar\rho_2)](T_s(x)) \cdot (T_1(x) - x)\right|\,\intd x\,\intd s \\
&\leq \frac{1}{2C_K} \int_0^1 \intrd \bar\rho_1(x)\, \left|\nabla [K \ast (\rho_2 - \bar\rho_2)](T_s(x))\right|^2\,\intd x\,\intd s \\
&\quad + \frac{C_K}{2} \int_0^1 \intrd \bar\rho_1(x) \, \left|T_1(x) - x\right|^2\,\intd x\,\intd s \\
&\leq \frac{1}{2C_K} C_K^2 \wass(\rho_2, \bar\rho_2)^2 + \frac{C_K}{2} \wass(\rho_1, \bar\rho_1)^2 = \frac{C_K}{2} \dst(\rrho, \bar\rrho)^2,
\end{align*}
implying in particular $\mk(\rrho) \geq -\frac{C_K}{2} \dst(\rrho, \bar\rrho)^2$, which is \eqref{eq:estim_intK}. 

Combining \eqref{eq:ldiff_rep} with the estimates \eqref{eq:estim_intK} and \eqref{eq:l_conv_estim} yields
\begin{align*}
\mi_F(\rrho) + \mi_K(\rrho) = \ml_\eps(\rrho) - \ml_\eps(\bar\rrho) - \mk(\rrho) \leq \left(1 + \frac{C_K}{\lambda_\eps}\right) (\ml_\eps(\rrho) - \ml_\eps(\bar\rrho)),
\end{align*}
proving \eqref{eq:I_estim_diffL}, thus concluding the proof.
\end{proof}
\subsection{Controlling the non-convex terms}

The next goal, which is a main step in our proof of exponential convergence, is to prove the following proposition that estimates the influence of the non-convex terms $\mn_\eps$.

\begin{proposition} \label{thm:key_estimate}
There are $\eps$-independent constants $c_1, c_2 > 0$ such that for $j = 1,2$ and every sufficiently small $\eps > 0$:
\begin{equation}
\begin{split} \label{eq:key_estimate}
&\intrd \left| \nabla \left[\partial_{r_j} h(\rrho) - V_j \right] \right|^2 \rho_j \, \intd x \leq c_1 (\ml_\eps(\rrho) - \ml_\eps(\bar\rrho)) \\
&+ c_2 \intrd \left[ \left| \nabla \left[ F_1'(\rho_1) + \eps V_1 + K \ast \rho_2 \right] \right|^2 \rho_1 + \left| \nabla \left[ F_2'(\rho_2) + \eps V_2 + K \ast \rho_1 \right] \right|^2 \rho_2 \right] \intd x
\end{split}
\end{equation}
for all $\rrho \in \sspc$ with $F_j'(\rho_j)$ locally Lipschitz.
\end{proposition}
\begin{proof}
Fix any pair $\rrho \in \sspc$ as above, and write $\uu = (U_1, U_2) := (F_1'(\rho_1), F_2'(\rho_2))$, and similarly $\bar\uu := (F_1'(\bar\rho_1), F_2'(\bar\rho_2))$. Note that $\partial_{r_j} h(\rrho) = \theta_j(\uu)$ inherits local Lipschitz-continuity from $F_j'(\rho_j)$, since $\theta_j$ is Lipschitz on $\Rnn^2$ by \eqref{eq:thetabdd_swap}. Recall also that $F_j'(\bar\rho_j)$ is Lipschitz on $\Rd$ by Lemma \ref{lem:minim_fj_lip}. Taking the gradient with respect to $x \in \Rd$ is thus well-defined almost everywhere for all functions considered. We have
\begin{equation*}
\nabla \theta_j(\uu) = \sum_{i = 1,2} \theta_{j,i}(\uu)\, \nabla U_i = \sum_{i = 1,2} \theta_{j,i}(\uu)\, \nabla F_i'(\rho_i).
\end{equation*}
Hence, we can rewrite the left-hand side of \eqref{eq:key_estimate}:
\begin{align}
\nonumber
&\intrd \left| \nabla \left[\partial_{r_j} h(\rrho) - V_j \right] \right|^2 \rho_j \, \intd x = \intrd \left| \nabla \left[ \theta_j(\uu) - \theta_j(\bar\uu) \right] \right|^2 \rho_j \, \intd x \\ \label{eq:int_innersum}
&= \intrd \left| \sum_{i = 1,2} \left[ \theta_{j,i}(\uu)\, \nabla F_i'(\rho_i) - \theta_{j,i}(\bar\uu)\, \nabla F_i'(\bar\rho_i) \right] \right|^2 \rho_j \, \intd x.
\end{align}
We split this sum into eight parts as follows:
\begin{align*}
&\sum_{i = 1,2} \left[ \theta_{j,i}(\uu)\, \nabla F_i'(\rho_i) - \theta_{j,i}(\bar\uu)\, \nabla F_i'(\bar\rho_i) \right] \\
&= \theta_{j,1}(\uu)\, \nabla \left[ F_1'(\rho_1)  + \eps V_1 + K \ast \rho_2 \right] + \theta_{j,2}(\uu)\,\nabla \left[ F_2'(\rho_2) + \eps V_2 + K \ast \rho_1 \right] \\
&\quad + \theta_{j,1}(\uu) \, \nabla \left[K \ast (\bar\rho_2 - \rho_2) \right] + \theta_{j,2}(\uu) \, \nabla \left[K \ast (\bar\rho_1 - \rho_1) \right] \\
&\quad - \theta_{j,1}(\uu)\, \nabla \left[ F_1'(\bar\rho_1)  + \eps V_1 + K \ast \bar\rho_2 \right] - \theta_{j,2}(\uu)\, \nabla \left[ F_2'(\bar\rho_2) + \eps V_2 + K \ast \bar\rho_1 \right] \\
&\quad + \left( \theta_{j,1}(\uu) - \theta_{j,1}(\bar\uu) \right)\,\nabla F_1'(\bar\rho_1) + \left( \theta_{j,2}(\uu) - \theta_{j,2}(\bar\uu) \right)\,\nabla F_2'(\bar\rho_2) \\
&=: B_{11} + B_{12} + B_{21} + B_{22} + B_{31} + B_{32} + B_{41} + B_{42}
\end{align*}
Inserting into \eqref{eq:int_innersum}, we obtain
\begin{equation} \label{eq:int_outersum}
\intrd \left| \nabla \left[\partial_{r_j} h(\rrho) - V_j \right] \right|^2 \rho_j \, \intd x = \intrd \left| \sum_{k = 1}^4 \sum_{l = 1}^2 B_{kl} \right|^2 \rho_j \, \intd x
\leq 8 \sum_{k = 1}^4 \sum_{l = 1}^2 \intrd |B_{kl}|^2 \rho_j \, \intd x.
\end{equation}
We shall prove the claim by showing that each of the eight summands in \eqref{eq:int_outersum} can be controlled by the right-hand side of \eqref{eq:key_estimate} separately. Explicitly, we need estimates of the form
\begin{equation*}
\intrd |B_{kl}|^2 \rho_j\, \intd x \leq c_{1} (\ml_\eps(\rrho) - \ml_\eps(\bar\rrho)) + c_{2} A(\rrho)
\end{equation*}
with constants $c_1, c_2$, where
\begin{equation*}
A(\rrho) := \intrd \left[ \left| \nabla \left[ F_1'(\rho_1) + \eps V_1 + K \ast \rho_2 \right] \right|^2 \rho_1 + \left| \nabla \left[ F_2'(\rho_2) + \eps V_2 + K \ast \rho_1 \right] \right|^2 \rho_2 \right]\, \intd x.
\end{equation*}

We may assume $j=1$, the other case is analogous. First, we observe that the function $\theta^2_{1,1}(\uu)$ is globally bounded on $\Rd$ by $\kappa^2_{1,1}$ according to \eqref{eq:thetabdd_swap}, and thus
\begin{equation*}
\intrd |B_{11}|^2 \rho_1\, \intd x \leq \kappa_{1,1}^2 \intrd \left| \nabla \left[ F_1'(\rho_1)  + \eps V_1 + K \ast \rho_2 \right] \right|^2 \rho_1 \, \intd x \leq \kappa_{1,1}^2 A(\rrho).
\end{equation*}
In order to get a similar estimate for the $B_{12}$-term, we use again \eqref{eq:thetabdd_swap} to see that the inequality $\theta_{1,2}^2(\uu)\rho_1 \leq \kappa_{1,2}^2 \rho_2$ holds in $\Rd$, hence
\begin{equation*}
\intrd |B_{12}|^2 \rho_1\, \intd x \leq \kappa_{1,2}^2 \intrd \left| \nabla \left[ F_2'(\rho_2)  + \eps V_2 + K \ast \rho_1 \right] \right|^2 \rho_2 \, \intd x \leq \kappa_{1,2}^2 A(\rrho).
\end{equation*}
For the $B_{21}$- and $B_{22}$-terms, we use again \eqref{eq:thetabdd_swap} and the estimates \eqref{eq:pw_estim_nablaK} and \eqref{eq:l_conv_estim} to obtain
\begin{align*}
\intrd |B_{21}|^2 \rho_1\, \intd x &= \intrd |\theta_{1,1}(\uu) |^2 |\nabla \left[K \ast (\bar\rho_2 - \rho_2) \right]|^2 \rho_1 \, \intd x  \\
&\leq \kappa_{1,1}^2 C_K^2 \wass(\rho_2, \bar\rho_2)^2 \leq \frac{2\kappa_{1,1}^2 C_K^2}{\lambda_\eps} (\ml_\eps(\rrho) - \ml_\eps(\bar\rrho)).
\end{align*}
The term $\intrd |B_{22}|^2 \rho_1\, \intd x$ is estimated analogously, using the property $|\theta_{1,2}(\uu)|^2\rho_1 \leq \kappa_{1,2}^2 \rho_2$.

For $B_{31}$, we use that $F_1'(\bar\rho_1)  + \eps V_1 = ( C_1- K \ast \bar\rho_2)_+$ by the Euler-Lagrange equation, yielding
\begin{align*}
|B_{31}|^2 &= \left| \theta_{1,1}(\uu)\, \nabla \left[ F_1'(\bar\rho_1)  + \eps V_1 + K \ast \bar\rho_2 \right] \right|^2 = \theta_{1,1}^2(\uu)\, \left| \nabla \left[ ( C_1- K \ast \bar\rho_2)_+ + K \ast \bar\rho_2 \right] \right|^2 \\
&= \theta_{1,1}^2(\uu) \, | \nabla [(K \ast \bar\rho_2 - C_1)_+]|^2 = \theta_{1,1}^2(\uu) \, | \nabla [K \ast \bar\rho_2 - C_1]|^2 \, \mathds{1}_{\{K \ast \bar\rho_2 > C_1\}}.
\end{align*}
The function $K \ast \bar\rho_2 - C_1$ inside the gradient is $\lambda$-convex with its second derivative bounded by $C_K$. With $x_0 := \argmin_{x \in \Rd} [K \ast \bar\rho_2(x) - C_1]$ and $c_0 :=   C_1 - K \ast \bar\rho_2(x_0)$, it thus holds
\begin{equation*}
|\nabla [K \ast \bar\rho_2(x) - C_1]|^2 \leq C_K^2 |x-x_0|^2 \leq \frac{2C_K^2}{\lambda} (K \ast \bar\rho_2(x) - C_1 + c_0).
\end{equation*}
Inserting above and using $(K \ast \bar\rho_2 - C_1)\,\mathds{1}_{\{K \ast \bar\rho_2 > C_1\}} = (K \ast \bar\rho_2 - C_1)_+$ yields
\begin{equation} \label{eq:b31_rep}
|B_{31}|^2 \leq \theta_{1,1}^2(\uu)\, \frac{2C_K^2}{\lambda} (K \ast \bar\rho_2 - C_1)_+ + \theta_{1,1}^2(\uu)\, \frac{2C_K^2 c_0}{\lambda}\, \mathds{1}_{\{K \ast \bar\rho_2 > C_1\}}
\end{equation}
Integrated against $\rho_1$, the first term is easily controlled by $\mi_K(\rrho)$:
\begin{equation} \label{eq:first_leq_ikr}
\intrd \theta_{1,1}^2(\uu)\, \frac{2C_K^2}{\lambda} (K \ast \bar\rho_2 - C_1)_+\, \rho_1\,\intd x \leq \frac{2C_K^2\kappa_{1,1}^2}{\lambda} \mi_K(\rrho).
\end{equation}
In the second term, note that the constant $c_0$ is non-negative, as otherwise $C_1 - K \ast \bar\rho_2$ would be negative everywhere in $\Rd$, contradicting $\{\bar\rho_1 > 0 \} = \{C_1 - K \ast \bar\rho_2 > 0 \}$ from Lemma \ref{lem:el_rhogr0}. Using the fact that on $\{K \ast \bar\rho_2 > C_1\}$, it is $\bar\rho_1 = 0$ and thus $\theta_{1,1}(\bar\uu)= 0$ by \eqref{eq:thetabdd_swap}, and inequality \eqref{eq:thetabregmanbd} together with $\bar\rho_1, \bar\rho_2 \leq H_0$ from Lemma \ref{lem:minim_bdd}, yields
\begin{equation} \label{eq:sec_leq_ifr}
\begin{split}
&\intrd \theta^2_{1,1}(\uu)\, \frac{2C_K^2 c_0}{\lambda}\, \mathds{1}_{\{K \ast \bar\rho_2 > C_1\}}\, \rho_1\,\intd x = \frac{2C_K^2 c_0}{\lambda} \int_{\{K \ast \bar\rho_2 > C_1\}} |\theta_{1,1}(\uu) - \theta_{1,1}(\bar\uu)|^2\,\rho_1 \,\intd x \\
&\leq \frac{2C_K^2 c_0 \beta_{H_0}}{\lambda} \int_{\{K \ast \bar\rho_2 > C_1\}} [\dst_{F_1}(\rho_1| \bar\rho_1) + \dst_{F_2}(\rho_2|\bar\rho_2) ] \,\intd x \leq \frac{2C_K^2 c_0 \beta_{H_0}}{\lambda}\, \mi_F(\rrho).
\end{split}
\end{equation}
The last inequality follows from the non-negativity of $\dst_{F_1}$ and $\dst_{F_2}$.
Combining the estimates \eqref{eq:first_leq_ikr} and \eqref{eq:sec_leq_ifr} yields with \eqref{eq:b31_rep}:
\begin{equation*}
\intrd |B_{31}|^2 \rho_1\, \intd x \leq \frac{2C_K^2}{\lambda} \max\{\kappa_{1,1}^2, c_0 \beta_{H_0}\}\, \left(\mi_K(\rrho)  + \mi_F(\rrho)\right).
\end{equation*}
With \eqref{eq:I_estim_diffL}, this implies that up to a constant factor, $\intrd |B_{31}|^2 \rho_1\, \intd x$ can be bounded from above by $\ml_\eps(\rrho) - \ml_\eps(\bar\rrho)$, which is the claimed estimate. Note that by Lemma \ref{lem:minim_bdd}, the constant $c_0 \leq C_1$ is $\eps$-uniformly bounded for small enough $\eps$.

For the $B_{32}$-term, we have the estimate
\begin{equation*}
|B_{32}|^2 \leq \theta_{1,2}^2(\uu)\, \frac{2C_K^2}{\lambda} (K \ast \bar\rho_1 - C_2)_+ + \theta_{1,2}^2(\uu)\, \frac{2C_K^2 \tilde{c}_0}{\lambda}\, \mathds{1}_{\{K \ast \bar\rho_1 > C_2\}}
\end{equation*}
with $\tilde c_0 = \max_{x} \left[ C_2 - K \ast \bar\rho_1(x) \right] > 0$, which is derived analogously to \eqref{eq:b31_rep}. Integrated against $\rho_1$, the two terms are now estimated like above, using $\theta^2_{1,2}(\uu) \rho_1 \leq \rho_2$ from \eqref{eq:thetabdd_swap}, and $|\theta_{1,2}(\uu) - \theta_{1,2}(\bar\uu)|^2\,\rho_1 \leq \beta_{H_0}[\dst_{F_1}(\rho_1| \bar\rho_1) + \dst_{F_2}(\rho_2|\bar\rho_2)]$ from \eqref{eq:thetabregmanbd}.

The last two terms $B_{4l}$ for $l = 1,2$ can be estimated using the Lipschitz-continuity of $F_l'(\bar\rho_l)$ from Lemma \ref{lem:minim_fj_lip}, implying $\left|\nabla F_l'(\bar\rho_l)\right|^2 \leq L_0^2$ with an $\eps$-independent constant $L_0$. Together with \eqref{eq:thetabregmanbd}, we obtain
\begin{align*}
\intrd |B_{4l}|^2 \rho_1 \, \intd x &= \intrd \left|\theta_{1,l}(\uu) - \theta_{1,l}(\bar\uu)\right|^2 \left|\nabla F_l'(\bar\rho_l)\right|^2\,\rho_1 \, \intd x \\
&\leq L_0^2\beta_{H_0} \intrd [\dst_{F_1}(\rho_1| \bar\rho_1) + \dst_{F_2}(\rho_2|\bar\rho_2) ] \,\intd x = L_0^2\beta_{H_0} \,\mi_F(\rrho).
\end{align*}
Together with \eqref{eq:I_estim_diffL} and the fact that $\mi_K(\rrho)$ is non-negative, this yields an estimate of the desired form, finishing the proof.
\end{proof}
In addition to Proposition \ref{thm:key_estimate}, which formally controls the slope of $\mn_\eps$ by the one of $\ml_\eps$, we prove the following lemma, which estimates the non-convex functional $\mn_\eps$ itself by $\ml_\eps$.
\begin{lemma} \label{lem:neps_estimate_leps}
There is an $\eps$-independent constant $C_\mn > 0$ such that for all sufficiently small $\eps > 0$ and all $\rrho \in \mptrd$, it holds
\begin{align*}
\left| \mn_\eps(\rrho) - \mn_\eps(\bar\rrho)\right| &\leq C_\mn\, \left( \ml_\eps(\rrho) - \ml_\eps(\bar\rrho)\right) \\
(1-\eps C_\mn)\left( \ml_\eps(\rrho) - \ml_\eps(\bar\rrho)\right) \leq \me_\eps(\rrho) &- \me_\eps(\bar\rrho) \leq (1 + \eps C_\mn) \left( \ml_\eps(\rrho) - \ml_\eps(\bar\rrho)\right).
\end{align*}
\end{lemma}
\begin{proof}
We may assume that both components of $\rrho$ are absolutely continuous with respect to the Lebesgue measure. From the definition of $\mn_\eps$ and the potentials $V_j$, it follows
\begin{equation} \label{eq:ndiff_rep}
\mn_\eps(\rrho) - \mn_\eps(\bar\rrho) = \intrd \left[h(\rrho) - h(\bar\rrho) - \partial_{r_1}h(\bar\rrho)(\rho_1 - \bar\rho_1) - \partial_{r_2}h(\bar\rrho)(\rho_2 - \bar\rho_2)\right]\,\intd x.
\end{equation}
We shall derive a pointwise estimate for this integrand by means of $\dst_{F_1}$ and $\dst_{F_2}$. With $\rr_s = (r_{1,s}, r_{2,s}) = (\rr + s(\bar\rr - \rr))$ and $\uu_s = (F_1'(r_{1,s}), F_2'(r_{2,s}))$, it holds for all $\rr, \bar\rr \in \Rnn^2$ by Taylor's formula for $h$:
\begin{equation} \label{eq:taylor_h}
\begin{split}
&h(\rr) - h(\bar\rr) - \nabla_\rr h(\bar\rr)\cdot(\rr - \bar\rr) = \int_0^1(\bar\rr - \rr)\cdot \nabla_\rr^2 h(\rr_s)(\bar\rr - \rr)\, s\,\intd s \\
&= \int_0^1 (\bar\rr - \rr)\cdot\begin{pmatrix}
\theta_{1,1}(\uu_s) & \theta_{1,2}(\uu_s) \\ \theta_{2,1}(\uu_s) & \theta_{2,2}(\uu_s)
\end{pmatrix} \begin{pmatrix}
F_1''(r_{1,s}) & 0 \\ 0 & F_2''(r_{2,s})
\end{pmatrix}(\bar\rr - \rr)\,s\,\intd s.
\end{split}
\end{equation}
With assumption \eqref{eq:thetabdd_swap} and using that $\theta_{1,2}(\uu_s)F_2''(r_{2,s}) = \partial_{r_1}\partial_{r_2}h(\rr_s) = \theta_{2,1}(\uu_s)F_1''(r_{1,s})$ and the non-negativity of $F_j''$, we can estimate
\begin{align*}
&\left|(\bar\rr - \rr)\cdot\begin{pmatrix}
\theta_{1,1}(\uu_s) & \theta_{1,2}(\uu_s) \\ \theta_{2,1}(\uu_s) & \theta_{2,2}(\uu_s)
\end{pmatrix} \begin{pmatrix}
F_1''(r_{1,s}) & 0 \\ 0 & F_2''(r_{2,s})
\end{pmatrix}(\bar\rr - \rr) \right| \\
&\leq \sum_{j=1,2} \left|\theta_{j,j}(\uu_s)\right|(\bar r_j - r_j)^2F_j''(r_{j,s}) + 2\left|\theta_{2,1}(\uu_s)(\bar r_1 - r_1)(\bar r_2 - r_2) \right| F_1''(r_{1,s}) \\
&\leq (\kappa_{1,1} + \kappa_{2,1})(\bar r_1 - r_1)^2 F_1''(r_{1,s}) + (\kappa_{1,2} + \kappa_{2,2})(\bar r_2 - r_2)^2 F_2''(r_{2,s}) \\ &\leq \kappa\, (\bar\rr - \rr)\cdot \begin{pmatrix}
F_1''(r_{1,s}) & 0 \\ 0 & F_2''(r_{2,s})
\end{pmatrix}(\bar\rr - \rr)
\end{align*}
with the constant $\kappa = \max_{i=1,2}\{\kappa_{1,i} + \kappa_{2,i}\}$, at every $s \in [0,1]$. Integrated from 0 to 1, this implies with \eqref{eq:taylor_h} and the same Taylor formula argument as above
\begin{align*}
\left| h(\rr) - h(\bar\rr) - \nabla_\rr h(\bar\rr)\cdot(\rr - \bar\rr) \right| \leq \kappa \left( \dst_{F_1}(r_1|\bar r_1) + \dst_{F_2}(r_2|\bar r_2) \right)
\end{align*}
for every $\rr, \bar\rr \in \Rnn^2$. Hence with \eqref{eq:ndiff_rep} and \eqref{eq:I_estim_diffL}, we obtain
\begin{align*}
\left|\mn_\eps(\rrho) - \mn_\eps(\bar\rrho)\right| \leq \kappa\, \mi_F(\rrho) \leq \kappa \left(1 + \frac{C_K}{\lambda_\eps}\right) (\ml_\eps(\rrho) - \ml_\eps(\bar\rrho)) \leq \kappa \left(1 + \frac{2C_K}{\lambda}\right) (\ml_\eps(\rrho) - \ml_\eps(\bar\rrho)),
\end{align*}
assuming that $\eps$ is small enough such that $\lambda_\eps \geq \frac{\lambda}{2}$.
This proves the first claimed inequality. The second claim follows directly from the first by the identity $\me_\eps = \ml_\eps + \eps \mn_\eps$. 
\end{proof}

Next, we show that in the inequality $\eqref{eq:key_estimate}$, the term $c_1 (\ml_\eps(\rrho) - \ml_\eps(\bar\rrho))$ can be controlled by the other term on the right hand side and thus can be omitted.
\begin{corollary} \label{cor:estim_slopes}
Fix $\bar\eps$ small enough such that $\lambda_\eps = \lambda - K_0 \eps \geq \frac{\lambda}{2}$ for all $\eps \in (0, \bar\eps]$. Then
there exists an $\eps$-independent constant $C$ such that for all $\eps \in (0, \bar\eps]$ and all $\rrho \in \sspc$ with $\nabla F_j'(\rho_j) \in L^2(\intd \rho_j)$, it holds
\begin{equation*}
\begin{split}
&\intrd \left[ \left| \nabla \left[\partial_{r_1} h(\rrho) - V_1 \right] \right|^2 \,\rho_1 + \left| \nabla \left[\partial_{r_2} h(\rrho) - V_2 \right] \right|^2\,\rho_2 \right] \, \intd x \\
&\leq C \intrd \left[ \left| \nabla \left[ F_1'(\rho_1) + \eps V_1 + K \ast \rho_2 \right] \right|^2 \,\rho_1 + \left| \nabla \left[ F_2'(\rho_2) + \eps V_2 + K \ast \rho_1 \right] \right|^2 \,\rho_2  \right]\, \intd x.
\end{split}
\end{equation*}
\end{corollary}
\begin{remark}
Formally, this is the inequality
\begin{equation*}
\intrd \left[\left|\nabla \frac{\delta \mn_\eps}{\delta \rho_1} \right|^2\rho_1 + \left|\nabla\frac{\delta \mn_\eps}{\delta \rho_2}\right|^2\rho_2\right]\,\intd x \leq C \intrd \left[\left|\nabla \frac{\delta \ml_\eps}{\delta \rho_1} \right|^2\rho_1 + \left|\nabla\frac{\delta \ml_\eps}{\delta \rho_2}\right|^2\rho_2\right]\,\intd x
\end{equation*}
for all $\rrho \in \sspc$ as above, which is expected to imply exponential convergence to equilibrium as demonstrated heuristically in the introduction.
\end{remark}
\begin{proof}
We use Lemma \ref{lem:above_tangent} with $\rrho^0 = \rrho$ and $\rrho^1 = \bar\rrho$. Using that by definition of $T_j$, it holds $\wass(\rho_j^1, \rho_j^0)^2 = \intrd \rho_j^0\, |T_j(x) - x|^2\,\intd x$, we estimate the integrals on the right-hand side:
\begin{align*}
&\intrd \rho_1^0\, \nabla\left[F_1'(\rho_1^0) + \eps V_1 + K \ast \rho_2^0\right] \cdot (T_1(x) - x) \,\intd x \\ &\geq - \frac{1}{2\lambda_\eps} \intrd \rho_1^0\, \left|  \nabla\left[F_1'(\rho_1^0) + \eps V_1 + K \ast \rho_2^0\right] \right|^2\,\intd x - \frac{\lambda_\eps}{2} \wass(\rho_j^1, \rho_j^0)^2
\end{align*}
Inserting into \eqref{eq:above_tangent} and multiplying by $-1$ yields
\begin{equation} \label{eq:geoconv_estim}
\begin{split}
&\ml_\eps(\rrho) - \ml_\eps(\bar\rrho) \\
&\leq \frac{1}{2\lambda_\eps} \intrd \left[ \rho_1\,\left| \nabla \left[ F_1'(\rho_1) + \eps V_1 + K \ast \rho_2 \right] \right|^2  +\rho_2\, \left| \nabla \left[ F_2'(\rho_2) + \eps V_2 + K \ast \rho_1 \right] \right|^2  \right]\, \intd x.
\end{split}
\end{equation}
We set $C := 2\left(\frac{c_1}{\lambda} + c_2\right)$ with $c_1$ and $c_2$ from \eqref{eq:key_estimate}. Observing that by assumption $\frac{1}{2\lambda_\eps} \leq \frac{1}{\lambda}$, the claim now follows from Proposition \ref{thm:key_estimate}.
\end{proof}

As a consequence of this estimate, we prove uniqueness of the minimizer of $\me_\eps$ for every sufficiently small $\eps$, as claimed in Theorem \ref{thm:main}.

\begin{corollary} \label{cor:min_unique}
Assuming that $\eps > 0$ is sufficiently small, the global minimizer $\bar\rrho \in \sspc$ of the functional $\me_\eps$ is unique.
\end{corollary}
\begin{proof}
Fix some $\eps \leq \bar\eps$ such that $\sqrt C \eps < 1$, with $\bar\eps, C$ being the constants from the previous corollary. Fix a global minimizer $\bar\rrho$ of $\me_\eps$ and denote $V_j = \partial_{r_j} h(\bar\rrho)$, as above. Now suppose $\bar\eeta = (\bar\eta_1, \bar\eta_2) \in \sspc$ is another global minimizer. Applying Theorem \ref{thm:euler-lagrange} for $\bar\eeta$, we get the Euler-Lagrange system \eqref{eq:euler-lagrange}, which we rewrite as
\begin{align*}
 - \eps \partial_{r_1}h(\bar\eeta) &= F_1'(\bar\eta_1) - (C_{1,\eta} - K \ast \bar\eta_2)_+ \\
 - \eps \partial_{r_2}h(\bar\eeta) &= F_2'(\bar\eta_2) - (C_{2,\eta} - K \ast \bar\eta_1)_+
\end{align*}
with constants $C_{j,\eta}$. Applying Corollary \ref{cor:estim_slopes} with $\rrho = \bar\eeta$, which is possible as $\nabla F_j'(\bar\eta_j)$ is globally bounded in $\Rd$ by Lemma \ref{lem:minim_fj_lip}, thus yields
\begin{align*}
&C \intrd \left[ \left| \nabla \left[ F_1'(\bar\eta_1) + \eps V_1 + K \ast \bar\eta_2 \right] \right|^2\, \bar\eta_1 + \left| \nabla \left[ F_2'(\bar\eta_2) + \eps V_2 + K \ast \bar\eta_1 \right] \right|^2\, \bar\eta_2 \right] \intd x \\
&\geq \intrd \left| \nabla \left[V_1 - \partial_{r_1} h(\bar\eeta)  \right] \right|^2\, \bar\eta_1 \, \intd x + \intrd \left| \nabla \left[V_2 - \partial_{r_2} h(\bar\eeta) \right] \right|^2\, \bar\eta_2 \, \intd x \\
&= \frac{1}{\eps^2} \left[ \intrd \left| \nabla \left[\eps V_1 - \eps \partial_{r_1} h(\bar\eeta)\right] \right|^2\, \bar\eta_1 \, \intd x + \intrd \left| \nabla \left[\eps V_2 - \eps\partial_{r_2} h(\bar\eeta) \right] \right|^2\, \bar\eta_2 \, \intd x \right] \\
&\geq C \left[ \intrd \left| \nabla \left[F_1'(\bar\eta_1) + \eps V_1 - (C_{1,\eta} - K \ast \bar\eta_2)_+ \right] \right|^2\, \bar\eta_1 \, \intd x \right. \\ &\quad + \left. \intrd \left| \nabla \left[F_2'(\bar\eta_2) + \eps V_2 - (C_{2,\eta} - K \ast \bar\eta_1)_+ \right] \right|^2\, \bar\eta_2 \, \intd x \right] \\
&= C \intrd \left[ \left| \nabla \left[ F_1'(\bar\eta_1) + \eps V_1 + K \ast \bar\eta_2 \right] \right|^2\, \bar\eta_1 + \left| \nabla \left[ F_2'(\bar\eta_2) + \eps V_2 + K \ast \bar\eta_1 \right] \right|^2\, \bar\eta_2 \right] \intd x,
\end{align*}
implying equality in every step.
The last equality follows from the fact that $C_{1,\eta} > K \ast \bar\eta_2$ on $\{\bar\eta_1 > 0\}$ and $C_{2,\eta} > K \ast \bar\eta_1$ on $\{\bar\eta_2 > 0\}$. Since $C < \frac{1}{\eps^2}$ by assumption, equality in the second-to-last step means that the sum of the integral terms must vanish, and thus every expression above is equal to 0. In particular,
\begin{equation*}
\intrd \left[ \left| \nabla \left[ F_1'(\bar\eta_1) + \eps V_1 + K \ast \bar\eta_2 \right] \right|^2\, \bar\eta_1 + \left| \nabla \left[ F_2'(\bar\eta_2) + \eps V_2 + K \ast \bar\eta_1 \right] \right|^2\, \bar\eta_2 \right] \intd x = 0.
\end{equation*}
Together with \eqref{eq:geoconv_estim}, this implies $\ml_\eps(\bar\eeta) - \ml_\eps(\bar\rrho) \leq 0$. Since $\bar\rrho$ is the unique global minimizer of $\ml_\eps$ by Corollary \ref{cor:l_min_rhobar}, this yields $\bar\eeta = \bar\rrho$.
\end{proof}
This uniqueness result allows us to prove additional properties of the minimizer $\bar\rrho \in \sspc$. In particular, we obtain the following symmetry result:
\begin{lemma} \label{lem:min_uniq}
Under the same assumptions on $\eps$ as in the previous corollary, both components $\bar\rho_j$ of the unique global minimizer $\bar\rrho \in \sspc$ are radially symmetric around their center of mass $\mom_1[\bar\rho_j] = 0$.
\end{lemma}
\begin{proof}
The claim will follow from rotational invariance of the energy $\me_\eps$. Let $R: \Rd \to \Rd$ be any orthogonal transformation. Define the pair of probability densities $\rrho^R := (\rho_1^R, \rho_2^R)$ by $\rho_j^R(x) := \bar\rho_j(Rx)$. It holds $\rrho^R \in \sspc$, and
\begin{align*}
\me_\eps(\rrho^R) &= \iintrdrd (\mf_\eps(\bar\rho_1(Rx), \bar\rho_2(Ry)) + \bar\rho_1(Rx)\bar\rho_2(Ry) K(x-y) )\, \intd x \intd y \\
&= \iintrdrd \left(\mf_\eps(\bar\rho_1(x'), \bar\rho_2(y')) + \bar\rho_1(x')\bar\rho_2(y') K(R^Tx' - R^Ty') \right)\, \intd x' \intd y' \\
&= \me_\eps(\bar\rrho),
\end{align*}
where we have substituted $(x', y') = (Rx, Ry)$. The last equality follows from the fact that $K(R^Tx' - R^Ty') = K(x' - y')$ by radial symmetry of the interaction potential. By uniqueness of the minimizer $\bar\rrho$, this implies $\rrho^R = \bar\rrho$. Since $R$ was arbitrary, this yields radial symmetry of $\bar\rho_j$.
\end{proof}

\subsection{Exponential decay}

We shall now use the estimates from Lemma \ref{lem:above_tangent} and Corollary \ref{cor:estim_slopes} to prove a decay estimate for the functional $\ml_\eps$ in the minimizing movement scheme for $\me_\eps$, as analyzed in section \ref{sec:min_move_scheme}.
\begin{lemma} \label{lem:decay_estim_step}
Let $\hat\rrho \in \mptrd^2$ be a pair such that $\me_\eps(\hat\rrho) < +\infty$, and let $\rrho^+$ be a global minimizer of $\eepstaurhat{\cdot}$. Then with the constant $C$ from Corollary \ref{cor:estim_slopes}, there holds
\begin{equation} \label{eq:decay_estim_step}
\ml_\eps(\hat\rrho) - \ml_\eps(\rrho^+) \geq 2 \lambda_\eps \tau \left( 1- \sqrt C \eps\right)\, \left( \ml_\eps(\rrho^+) - \ml_\eps(\bar\rrho)\right).
\end{equation}
\end{lemma}
\begin{proof}
We first prove the estimate under the stronger assumption that $\hat\rho_1$ and $\hat\rho_2$ are strictly positive inside balls $B_{R_1}$ and $B_{R_2}$ and vanish almost everywhere outside, so that Lemma \ref{lem:el_yosida} and its consequences are applicable. In this case, it follows from Lemma \ref{lem:h1est_bigrho_2} that $\nabla F_j'(\rho_j^+) \in L^2(\intd \rho_j^+)$, hence we can apply Lemma \ref{lem:above_tangent} to $\rrho^0 = \rrho^+$ and  $\rrho^1 = \hat\rrho$ to obtain
\begin{align*}
\ml_\eps(\hat\rrho) - \ml_\eps(\rrho^+) &\geq \intrd \rho_1^+\, \nabla\left[F_1'(\rho_1^+) + \eps V_1 + K \ast \rho_2^+\right] \cdot (T_1(x) - x) \,\intd x \\
&+ \intrd \rho_2^+\, \nabla\left[F_2'(\rho_2^+) + \eps V_2 + K \ast \rho_1^+\right] \cdot (T_2(x) - x)\,\intd x + \frac{\lambda_\eps}{2} \dst(\hat\rrho, \rrho^+)^2
\end{align*}
with the optimal transport maps $T_j$ pushing $\rho_j^+$ to $\hat\rho_j$. For brevity, we denote $\md_j(\rrho) := F_j'(\rho_j) + \eps V_j + K \ast \rho_{j'}$. Inserting \eqref{eq:el_map_yosida} and omitting the non-negative distance term yields
\begin{align*}
&\ml_\eps(\hat\rrho) - \ml_\eps(\rrho^+)\geq \tau \sum_{j = 1,2} \intrd \rho_j^+\,\nabla\md_j(\rrho^+) \cdot \nabla \left[F_j'(\rho_j^+) + \eps \partial_{r_j} h(\rho_1^+, \rho_2^+) + K \ast \rho_{j'}^+\right]\,\intd x \\
&= \tau \sum_{j = 1,2} \left( \eps\intrd \rho_j^+\,\nabla\md_j(\rrho^+) \cdot \nabla \left[\partial_{r_j}h(\rho_1^+, \rho_2^+) - V_j\right]\,\intd x + \intrd \rho_j^+\, \left| \nabla \md_j(\rrho^+) \right|^2\,\intd x \right) \\
&\geq \tau \sum_{j = 1,2} \left( \left(1 - \frac{\sqrt C}{2}\eps \right) \intrd \rho_j^+\, \left| \nabla \md_j(\rrho^+) \right|^2\,\intd x - \frac{\eps}{2\sqrt C} \intrd \rho_j^+\left|\nabla\left[\partial_{r_j}h(\rho_1^+, \rho_2^+) - V_j\right]\right|^2\,\intd x \right) \\
&\geq \tau \sum_{j = 1,2} \left(1- \sqrt C \eps\right) \intrd \rho_j^+\, \left| \nabla \md_j(\rrho^+) \right|^2\,\intd x \geq 2 \lambda_\eps \tau \left( 1- \sqrt C \eps\right)\, \left( \ml_\eps(\rrho^+) - \ml_\eps(\bar\rrho)\right),
\end{align*}
where in the last two inequalities we used Corollary \ref{cor:estim_slopes} and \eqref{eq:geoconv_estim} with $\rrho = \rrho^+$, respectively.

In the case of a more general $\hat\rrho$, we use the approximating sequence of pairs $\hat\rrho_n$ from the proof of Theorem \ref{thm:discrete_error}, which satisfy the stronger assumptions. Hence, we can apply the proven result to $\hat\rrho_n$ to obtain
\begin{align*}
\ml_\eps(\hat\rrho_n) \geq 2 \lambda_\eps \tau \left( 1- \sqrt C \eps\right)\, \left( \ml_\eps(\rrho^+_n) - \ml_\eps(\bar\rrho)\right) + \ml_\eps(\rrho^+_n).
\end{align*}
It follows from the construction that $\ml_\eps(\hat\rrho_n) \to \ml_\eps(\hat\rrho)$. On the right-hand side, we use the narrow convergence $\rrho_n^+ \to \rrho^+$ and lower semi-continuity of $\ml_\eps$ with respect to narrow convergence, which holds since the $V_j$ are bounded and continuous. This proves the claim.
\end{proof}
\begin{remark}
To see why Lemma \ref{lem:decay_estim_step} yields the desired decay estimate, observe that by abbreviating the relative energy as $\ml^r_\eps(\rrho) := \ml_\eps(\rrho) - \ml_\eps(\bar\rrho)$, \eqref{eq:decay_estim_step} can be written as
\begin{align*}
\frac{\ml_\eps^r(\rrho^+) - \ml^r_\eps(\hat\rrho)}{\tau} \leq -2\lambda_\eps \left(1 - \sqrt C \eps \right) \ml^r_\eps(\rrho^+).
\end{align*}
In the following theorem, we show that this estimate leads to exponential decay of $\ml_\eps^r$ along the limit curve $\rrho^*$ constructed in Lemma \ref{lem:interpol_cvgce}, which thus can be viewed as a discrete version of Gronwall's lemma.
\end{remark}
\begin{theorem} \label{thm:decay_estim}
For any initial condition $\rrho^0 \in \sspc$ with $\me_\eps(\rrho^0)$ finite, the weak solution $\rrho^*:[0,+\infty) \to \sspc$ to system \eqref{eq:system-intro} with $\rrho^*(0) = \rrho^0$ from subsection \ref{subsec:weaksol} satisfies
\begin{equation} \label{eq:decay_estim}
\ml_\eps(\rrho^*(t)) - \ml_\eps(\bar\rrho) \leq e^{-2\lambda_\eps \left(1-\sqrt C \eps \right)t}\, \left( \ml_\eps(\rrho^0) - \ml_\eps(\bar\rrho)\right)
\end{equation}
at every $t \geq 0$. In particular, if $\eps < 1/\sqrt C$, the functional $\ml_\eps$ decays at an exponential rate along the solution curve $\rrho^*$. If in addition, $\eps < \frac{1}{C_\mn}$ with $C_\mn$ from Lemma \ref{lem:neps_estimate_leps}, it holds
\begin{equation} \label{eq:decay_estim_eeps}
\me_\eps(\rrho^*(t)) - \me_\eps(\bar\rrho) \leq \frac{1+\eps C_\mn}{1-\eps C_\mn}\, e^{-2\lambda_\eps \left(1-\sqrt C \eps \right)t}\, \left( \me_\eps(\rrho^0) - \me_\eps(\bar\rrho)\right),
\end{equation}
at every $t \geq 0$, proving the exponential decay for $\me_\eps$ claimed in Theorem \ref{thm:main}.
\end{theorem}
\begin{proof}
Fix any $\eps > 0$ and consider the interpolation scheme \eqref{eq:interpol_def}. Applying Lemma \ref{lem:decay_estim_step}, we obtain for every $n$ the estimate
\begin{align*}
\ml_\eps(\hat\rrho_n) - \ml_\eps(\rrho_n^+) \geq 2 \lambda_\eps \tau \left( 1- \sqrt C \eps\right)\, \left( \ml_\eps(\rrho_n^+) - \ml_\eps(\bar\rrho)\right)
\end{align*}
After adding $\ml_\eps(\rrho_n^+) - \ml_\eps(\bar\rrho)$ and assuming $\tau$ is small enough such that $2\lambda_\eps \tau(1-\sqrt C \eps) > -1$, this becomes
\begin{align*}
\ml_\eps(\rrho_n^+) - \ml_\eps(\bar\rrho) \leq \left(1 + 2 \lambda_\eps \tau \left( 1- \sqrt C \eps\right) \right)^{-1} \left(\ml_\eps(\hat\rrho_n) - \ml_\eps(\bar\rrho)\right).
\end{align*}
Using that $\rrho_{n-1}^+ = \hat\rrho_{n}$ for every $n \geq 1$, induction over $n$ yields
\begin{equation} \label{eq:decay_estim_multi}
\ml_\eps(\rrho_n^+) - \ml_\eps(\bar\rrho) \leq \left(1 + 2 \lambda_\eps \tau \left( 1- \sqrt C \eps\right) \right)^{-(n+1)} \left(\ml_\eps(\rrho^0) - \ml_\eps(\bar\rrho)\right)
\end{equation}
for every $n \geq 0$. Now fix $t > 0$ and $n$ such that $\rrho^\tau(t) = \rrho_n^+$, which holds for $n\tau < t \leq (n+1)\tau$. In particular, we have $n + 1 \geq t/\tau$, and thus with \eqref{eq:decay_estim_multi},
\begin{equation} \label{eq:decay_estim_tau}
\ml_\eps\left(\rrho^\tau(t)\right) - \ml_\eps(\bar\rrho) \leq \left(1 + 2 \lambda_\eps \tau \left( 1- \sqrt C \eps\right) \right)^{-t/\tau} \left(\ml_\eps(\rrho^0) - \ml_\eps(\bar\rrho)\right)
\end{equation}
In the limit $\tau \downarrow 0$, the right-hand side of \eqref{eq:decay_estim_tau} converges to the right-hand side of \eqref{eq:decay_estim}. On the left-hand side, note that $\rrho^\tau(t) \to \rrho^*(t)$ narrowly by Lemma \ref{lem:interpol_cvgce} and that $\ml_\eps$ is lower semi-continuous with respect to narrow convergence. This proves \eqref{eq:decay_estim}. The estimate \eqref{eq:decay_estim_eeps} follows directly from \eqref{eq:decay_estim} and Lemma \ref{lem:neps_estimate_leps}.
\end{proof}
Note that $\lambda_\eps(1 - \sqrt C \eps) = \lambda - (K_0 + \sqrt C)\eps + K_0\sqrt C \eps^2$, hence the exponential decay estimate \eqref{eq:expdecay_intro} claimed in the introduction is satisfied for every sufficiently small $\eps$ with any constants $C_0 > K_0 + \sqrt C$ and $C_1 > 1$.

In order to prove that this decay estimate implies exponential convergence of $\rrho^*$ to $\bar\rrho$, we first need the following technical lemma, which estimates the behaviour of the $\bar\rho_j$ near the boundary of their support.
\begin{lemma}  \label{lem:estim_intfjpp}
For $\bar\eps > 0$ sufficiently small, there exists an $\eps$-independent constant $A < +\infty$ such that for every $\eps \in (0, \bar\eps]$, it holds
\begin{align*}
\int_{\{\bar\rho_j > 0\}} \frac{1}{F_j''(\bar\rho_j)}\,\intd x \leq A.
\end{align*}
\end{lemma}
\begin{proof}
The proof is based on assumption \eqref{eq:fjpp_power} and the estimates from Lemma \ref{lem:minim_bdd}. Recall that for $j = 1,2$, it holds
\begin{align*}
(C_j - K \ast \bar\rho_{j'})_+ \leq (1+\eps \kappa_{j,j})\,F_j'(\bar\rho_j),
\end{align*}
see \eqref{eq:el_fj_lowerbd}. \eqref{eq:fjpp_power} implies that for a sufficiently small $r_0 > 0$, we have for every $r \in [0, r_0]$ the estimate
\begin{align*}
M_j F_j''(\bar\rho_j) \geq \bar\rho_j^{\,\beta_j} \geq \left(m_j(\beta_j+1)\right)^{\frac{\beta_j}{\beta_j+1}} F_j'(\bar\rho_j)^{\frac{\beta_j}{\beta_j+1}} \geq\left(\frac{m_j(\beta_j+1)}{1+\eps \kappa_{j,j}}\right)^{\frac{\beta_j}{\beta_j+1}} (C_j - K \ast \bar\rho_{j'})_+^{\frac{\beta_j}{\beta_j+1}}
\end{align*}
with the constants $m_j,M_j,\beta_j$ from \eqref{eq:fjpp_power}. The last estimate follows from \eqref{eq:el_fj_lowerbd}. Since the constant in the last expression is $\eps$-uniformly bounded for small $\eps$, the claim will follow if we show that the value of the integral
\begin{align*}
\int_{\{\bar\rho_j > 0\}} (C_j - K \ast \bar\rho_{j'})_+^{-\frac{\beta_j}{\beta_j+1}}\,\intd x = \int_{\{K \ast \bar\rho_{j'} < C_j\}} (C_j - K \ast \bar\rho_{j'})^{-\frac{\beta_j}{\beta_j+1}}\,\intd x
\end{align*}
is $\eps$-uniformly bounded for small $\eps$. Recall that by Lemma \ref{lem:min_uniq}, $\bar\rho_{j'}$ is radially symmetric, thus $K \ast \bar\rho_{j'}$, which is $\lambda$-convex, is radially increasing around 0. This implies that the set $\{\bar\rho_j > 0\} = \{K \ast \bar\rho_{j'} < C_j\}$ is a ball $B_R$, whose radius $R$ is $\eps$-uniformly bounded by Lemma \ref{lem:minim_bdd}. Since the density $\bar\rho_j$ is $\eps$-uniformly bounded by $H_0$, again by Lemma \ref{lem:minim_bdd}, the radius $R$ is also uniformly bounded away from 0 for small $\eps$. By $\lambda$-convexity and radial symmetry, this implies that the gradient of $K \ast \bar\rho_{j'}$ is uniformly bounded away from 0 near the boundary of $B_R$, implying that for some constant $c > 0$, it holds $C_j - K \ast \bar\rho_{j'}(x) \geq c\, (R-|x|)$  at every $x \in B_R$, meaning we can estimate the integral above by
\begin{align*}
\int_{\{K \ast \bar\rho_{j'} < C_j\}} (C_j - K \ast \bar\rho_{j'})^{-\frac{\beta_j}{\beta_j+1}}\,\intd x \leq c^{-\frac{\beta_j}{\beta_j+1}}\int_{B_R} (R - |x|)^{-\frac{\beta_j}{\beta_j+1}}\,\intd x,
\end{align*}
and the integral on the right is finite since $0 \leq \frac{\beta_j}{\beta_j + 1} < 1$. This proves the claim.
\end{proof}
With the help of this estimate, we can prove the following version of the Csiszár-Kullback inequality for the functional $\ml_\eps$:
\begin{lemma} \label{lem:csiszar-kullback}
There is an $\eps$-independent constant $C_{CK}$ such that for all $\rrho \in \sspc$ with $\rho_j$ absolutely continuous w.r.t the Lebesgue measure and all $\eps$ such that $\lambda_\eps \geq \frac{\lambda}{2}$, it holds
\begin{align*}
\|\rho_1 - \bar\rho_1\|_{L^1(\Rd)}^2 + \|\rho_2 - \bar\rho_2\|_{L^1(\Rd)}^2 \leq C_{CK}\left(\ml_\eps(\rrho) - \ml_\eps(\bar\rrho)\right).
\end{align*}
\end{lemma}
\begin{proof}
Inequality \eqref{eq:I_estim_diffL} and $\lambda_\eps \geq \frac{\lambda}{2}$ give the estimate
\begin{align*}
\intrd\left[\dst_{F_1}(\rho_1|\bar\rho_1) + \dst_{F_2}(\rho_2|\bar\rho_2)\right]\,\intd x \leq \left(1 + \frac{2C_K}{\lambda}\right) \left(\ml_\eps(\rrho) - \ml_\eps(\bar\rrho)\right),
\end{align*}
thus the claim will follow if we show an estimate of the form
\begin{align*}
\|\rho_j - \bar\rho_j\|_{L^1(\Rd)}^2  \leq C \intrd \dst_{F_j}(\rho_j|\bar\rho_j)\,\intd x
\end{align*}
with some $\eps$-uniform constant $C$. To prove this, observe that
\begin{equation} \label{eq:l1_split}
\|\rho_j - \bar\rho_j\|_{L^1(\Rd)} = \int_{\{\rho_j < \bar\rho_j\}} (\bar\rho_j - \rho_j)\,\intd x + \int_{\{\rho_j \geq \bar\rho_j\}} (\rho_j - \bar\rho_j)\,\intd x = 2 \int_{\{\rho_j < \bar\rho_j\}} (\bar\rho_j - \rho_j)\,\intd x,
\end{equation}
where the last equality follows from the fact that $\rho_j$ and $\bar\rho_j$ have unit mass. Additionally, it follows from the non-negativity of $F_j''$ that for every $0 \leq r < \bar r$, it holds
\begin{align*}
\dst_{F_j}(r|\bar r) &= \int_r^{\bar r} F_j''(s)(s-r)\,\intd s \geq \int_{\frac{r+\bar r}{2}}^{\bar r} F_j''(s) (s-r)\,\intd s \geq \left( \frac{\bar r - r}{2}\right)^2 \min_{s \in \left[\frac{r+\bar r}{2}, \bar r\right]} F_j''(s).
\end{align*}
The minimum on the right-hand side can be estimated further with assumption \eqref{eq:fjpp_power}, as for every $s \in \left[\frac{r+\bar r}{2}, \bar r\right]$, it holds $s \geq \frac{\bar r}{2}$ and thus
\begin{align*}
F_j''(s) \geq \frac{1}{M_j}s^{\beta_j} \geq \frac{1}{M_j 2^{\beta_j}}\, \bar r^{\beta_j} \geq \frac{m_j}{M_j 2^{\beta_j}} F_j''(\bar r).
\end{align*}
Combining the last two inequalities yields
\begin{align*}
\dst_{F_j}(r|\bar r) \geq \frac{m_j}{M_j 2^{\beta_j + 2}} \,F_j''(\bar r) (\bar r - r)^2.
\end{align*}
With \eqref{eq:l1_split}, the Cauchy-Schwarz inequality, Lemma \ref{lem:estim_intfjpp} and $\dst_{F_j}(r|\bar r) \geq 0$, we obtain
\begin{align*}
&\|\rho_j - \bar\rho_j\|_{L^1(\Rd)}^2 = 4 \left( \int_{\{\rho_j < \bar\rho_j\}} (\bar\rho_j - \rho_j)\,\intd x \right)^2  \\
&\leq 4 \left(\int_{\{\rho_j < \bar\rho_j\}} F_j''(\bar\rho_j)(\bar\rho_j - \rho_j)^2\,\intd x\right) \left( \int_{\{\rho_j < \bar\rho_j\}} \frac{1}{F_j''(\bar\rho_j)}\,\intd x \right) \leq \frac{M_j 2^{\beta_j + 4} A}{m_j} \intrd \dst_{F_j}(\rho_j|\bar \rho_j)\,\intd x.
\end{align*}
This is an inequality of the desired form given above, thus proving the claim.
\end{proof}
Exponential convergence to the equilibrium for the solution curve $\rrho^*$ is now a direct consequence of the previous results:

\begin{theorem}
Assume that $\bar \eps >0 $ is sufficiently small. Then given any $\eps \in (0, \bar\eps]$ and any initial data $\rrho^0 \in \sspc$ with $\me_\eps(\rrho^0) < +\infty$, the weak solution $\rrho^*$ to system \eqref{eq:system-intro} from subsection \ref{subsec:weaksol} satisfies at every $t \geq 0$ the estimate
\begin{align*}
\|\rho_1^*(t) - \bar\rho_1\|_{L^1(\Rd)}^2 + \|\rho_2^*(t) - \bar\rho_2\|_{L^1(\Rd)}^2 &\leq C_{CK}\, \left( \ml_\eps(\rrho^0) - \ml_\eps(\bar\rrho)\right)\, e^{-2\lambda_\eps \left(1-\sqrt C \eps \right)t} \\
&\leq \frac{C_{CK}}{1 - \eps C_\mn} \left( \me_\eps(\rrho^0) - \me_\eps(\bar\rrho)\right)\, e^{-2\lambda_\eps \left(1-\sqrt C \eps \right)t}
\end{align*}
with the constants $C_{CK}$ from the previous lemma and $C_\mn$ from Lemma \ref{lem:neps_estimate_leps}. In particular, $\rrho^*(t)$ converges to the unique global minimizer $\bar\rrho$ of the energy functional $\me_\eps$ at an exponential rate in $L^1$.
\end{theorem}
\begin{proof}
The first inequality follows directly from Lemma \ref{lem:csiszar-kullback} and Theorem \ref{thm:decay_estim}, while the second one follows from Lemma \ref{lem:neps_estimate_leps}.
\end{proof}

\section*{Acknowledgements}
The authors would like to thank Antonio Esposito for numerous discussions and suggestions.

\appendix
\section{Example}
In this section, we give a generic example of functions $F_1, F_2$ and $h$ and prove that they satisfy all assumption from Hypotheses \ref{hyp:general} and \ref{hyp:theta} under certain conditions. Note that the assumptions on $K$ are independent from $F_1, F_2$ and $h$ and vice versa.

\begin{proposition} \label{prop:example}
Let $a_1, a_2, b_1, b_2,\gamma$ be positive real numbers. Consider
\begin{align*}
F_1(r) := \frac{1}{a_1}\,r^{a_1}, \quad F_2(r) := \frac{1}{a_2}\, r^{a_2}, \quad h(\rr) := \frac{r_1^{b_1} r_2^{b_2}}{(1+r_1+r_2)^\gamma}.
\end{align*}
Then the nonlinearities $F_1$ and $F_2$ and the coupling function $h$ fulfill all the assumptions from Hypotheses \ref{hyp:general} and \ref{hyp:theta} if and only if for $j = 1,2$, it holds
\begin{equation} \label{eq:example_cond}
a_j \geq 2, \qquad b_j \geq 2a_j - 1,\qquad b_1 + b_2 \leq \gamma + \min\{a_1, a_2\}.
\end{equation}
\end{proposition}
\begin{proof}
We have $F_j'(r) = r^{a_j-1}, F_j''(r) = (a_j - 1) r^{a_j - 2}$. Hence assumption \eqref{eq:fjpp_power} is satisfied with $\beta_j = a_j - 2$, which is non-negative if and only if $a_j \geq 2$. In this case, the McCann-condition \eqref{eq:mccann_cond} is always fulfilled, since \begin{align*}
F_j(r) - rF_j'(r) + r^2F_j''(r) = \left(\frac{1}{a_j} -1+a_j - 1 \right)\, r^{a_j} \geq 0.
\end{align*}
The other conditions on $F_j$ from Hypothesis \ref{hyp:general} are clearly satisfied as well. In order to verify the flatness assumptions on $h$, in particular Hypothesis \ref{hyp:theta}, we introduce the following notations: Given arbitrary real numbers $b_1', b_2'$ and $\gamma'$, we define $Q_{b_1',b_2',\gamma'}: \Rnn^2 \to \R$ as
\begin{align*}
Q_{b_1',b_2',\gamma'}(\rr) := \frac{r_1^{b_1'}r_2^{b_2'}}{(1+r_1+r_2)^{\gamma'}}.
\end{align*}
Note that $h = Q_{b_1,b_2,\gamma}$. Observe that $Q_{b_1',b_2',\gamma'}$ is bounded in $\Rnn^2$ if and only if
\begin{equation} \label{eq:qbound_cond}
b_1', b_2' \geq 0,\quad b_1'+b_2' \leq \gamma',
\end{equation}
while for local boundedness, only the first condition is required. If $b_1',b_2' >0$, the function $Q_{b_1',b_2',\gamma'}$ vanishes on $\partial \Rnn^2$.

Additionally, for any functions $Q_1, \dots, Q_n: \Rnn^2 \to \R$, we define $V(Q_1,\dots,Q_n)$ to be the set of all linear combinations of the functions $Q_1,\dots,Q_n$. It is clear that $Q \in V(Q_1,\dots,Q_n)$ is (locally) bounded in $\Rnn^2$ if every $Q_i$ is (locally) bounded for $i = 1,\dots,n$, and that $Q$ vanishes everywhere on $\partial\Rnn^2$ if every $Q_i$ does. Now, observe that
\begin{equation} \label{eq:qdiff_rep}
\begin{split}
\partial_{r_1} Q_{b_1',b_2', \gamma'} = b_1'\,Q_{b_1'-1,b_2',\gamma'} - \gamma'\,Q_{b_1',b_2',\gamma'+1} \in V(Q_{b_1'-1,b_2',\gamma'}, Q_{b_1',b_2',\gamma'+1}) \\
\partial_{r_2} Q_{b_1',b_2', \gamma'} = b_2'\,Q_{b_1',b_2'-1,\gamma'} - \gamma'\,Q_{b_1',b_2',\gamma'+1} \in V(Q_{b_1',b_2'-1,\gamma'}, Q_{b_1',b_2',\gamma'+1}).
\end{split}
\end{equation}
In follows that $\partial_{r_1} h \in V(Q_{b_1-1,b_2,\gamma}, Q_{b_1,b_2,\gamma+1})$ and $\partial_{r_2} h \in V(Q_{b_1,b_2-1,\gamma}, Q_{b_1,b_2,\gamma+1})$. Hence, the assumption $\partial_{r_j}h(\rr) = 0$ for every $\rr \in \partial\Rnn^2$ from Hypothesis \ref{hyp:general} is satisfied if and only if $b_j > 1$ for $j = 1,2$. It thus remains to verify the assumptions from Hypothesis \ref{hyp:theta}. We start with condition \eqref{eq:bddswap_h_r}. For the second derivatives of $h$, we obtain by applying \eqref{eq:qdiff_rep} twice:
\begin{equation} \label{eq:d2h_rep}
\begin{split}
\partial_{r_1}\partial_{r_1}h &\in V(Q_{b_1 - 2, b_2, \gamma}, Q_{b_1 - 1,b_2,\gamma+1}, Q_{b_1,b_2,\gamma + 2}) \\
\partial_{r_1}\partial_{r_2}h &\in V(Q_{b_1 - 1, b_2 - 1, \gamma}, Q_{b_1 - 1,b_2,\gamma+1}, Q_{b_1,b_2 - 1,\gamma + 1}, Q_{b_1,b_2,\gamma + 2}) \\
\partial_{r_2}\partial_{r_2}h &\in V(Q_{b_1, b_2 - 2, \gamma}, Q_{b_1,b_2 - 1,\gamma+1}, Q_{b_1,b_2,\gamma + 2}).
\end{split}
\end{equation}
In order to prove \eqref{eq:bddswap_h_r}, note that for any $m_1,m_2 \in \R$, it holds
\begin{align} \label{eq:q_power}
r_1^{m_1}r_2^{m_2}Q_{b_1',b_2',\gamma'}(\rr) = Q_{b_1' + m_1,b_2' + m_2,\gamma'}(\rr).
\end{align}
From this observation and $F_j'(r) = r^{a_j - 1}$, $F_j''(r) = (a_j - 1)r^{a_j - 2}$, we obtain with \eqref{eq:d2h_rep}
\begin{equation} \label{eq:theta_rep_q}
\begin{split}
&\frac{\partial_{r_1}\partial_{r_1} h(\rr)}{F_1''(r_1)} \in V(Q_{b_1 - a_1, b_2, \gamma}, Q_{b_1 - a_1 + 1,b_2,\gamma+1}, Q_{b_1 - a_1 + 2,b_2,\gamma + 2}) \\
&\frac{\partial_{r_1}\partial_{r_1} h(\rr)}{F_1''(r_1)F_1'(r_1)} \in V(Q_{b_1 - 2a_1 + 1, b_2, \gamma}, Q_{b_1 - 2a_1 + 2,b_2,\gamma+1}, Q_{b_1 - 2a_1 + 3,b_2,\gamma + 2}) \\
&\frac{\partial_{r_1}\partial_{r_1} h(\rr)}{F_1''(r_1)F_2'(r_2)} \in V(Q_{b_1 - a_1, b_2 - a_2 + 1, \gamma}, Q_{b_1 - a_1 + 1,b_2 - a_2 +1,\gamma+1}, Q_{b_1 - a_1 + 2,b_2 - a_2 + 1,\gamma + 2}) \\
&\frac{\partial_{r_1}\partial_{r_2} h(\rr)}{F_1''(r_1)} \in V(Q_{b_1 - a_1 + 1, b_2 - 1, \gamma}, Q_{b_1 -a_1+1,b_2,\gamma+1}, Q_{b_1 - a_1 + 2,b_2 - 1,\gamma + 1}, Q_{b_1 - a_1 + 2,b_2,\gamma + 2}) \\
&\frac{\partial_{r_1}\partial_{r_2} h(\rr)}{F_1''(r_1)F_1'(r_1)} \in V(Q_{b_1 - 2a_1 + 2, b_2 - 1, \gamma}, Q_{b_1 -2a_1+2,b_2,\gamma+1}, Q_{b_1 - 2a_1 + 3,b_2 - 1,\gamma + 1}, Q_{b_1 - 2a_1 + 3,b_2,\gamma + 2}) \\
&\frac{\partial_{r_1}\partial_{r_2} h(\rr)}{F_1''(r_1)F_2'(r_2)} \in V(Q_{b_1 - a_1 + 1, b_2 - a_2, \gamma}, Q_{b_1 - a_1+1,b_2 - a_2 + 1,\gamma+1}, Q_{b_1 - a_1 + 2,b_2 - a_2,\gamma + 1}, Q_{b_1 - a_1 + 2,b_2 - a_2 + 1,\gamma + 2}) \\
&\frac{\partial_{r_1}\partial_{r_2} h(\rr) \sqrt{r_2}}{F_1''(r_1)\sqrt{r_1}} \in V(Q_{b_1 - a_1 + \frac{1}{2}, b_2 - \frac{1}{2}, \gamma}, Q_{b_1 -a_1+\frac{1}{2},b_2 + \frac{1}{2},\gamma+1}, Q_{b_1 - a_1 + \frac{3}{2},b_2 - \frac{1}{2},\gamma + 1}, Q_{b_1 - a_1 + \frac{3}{2},b_2 + \frac{1}{2},\gamma + 2}).
\end{split}
\end{equation}
In order for inequality \eqref{eq:bddswap_h_r} to be true, we need all these functions to be bounded in $\Rnn^2$. All other functions that are assumed to be bounded by \eqref{eq:bddswap_h_r} are analogous to one of the functions appearing above, hence we obtain the set of conditions on $a_1, a_2,b_1,b_2, \gamma$ required to fulfill \eqref{eq:bddswap_h_r} by analyzing when all functions above are bounded. In order to get this boundedness, \eqref{eq:qbound_cond} needs to be satisfied for all $b_1', b_2', \gamma'$ appearing above. The condition $b_1' \geq 0$ yields $b_1 \geq 2a_1 - 1$, since $b_1 - 2a_1 + 1$ appearing in the second line is the smallest $b_1'$ appearing above. By symmetry, this implies that $b_2 \geq 2a_2 - 1$ is needed too, which is sufficient for $b_2' \geq 0$ for all $b_2'$ above. For the inequality $b_1' + b_2' \leq \gamma'$ observe that in every line above, the value of $b_1' + b_2' - \gamma'$ is the same for all $Q_{b_1', b_2', \gamma'}$ appearing in that line, thus we only need to analyze the first term. The largest value of $b_1' + b_2' - \gamma'$ appearing above is $b_1 - a_1 + b_2 - \gamma$ in the first line, hence we require $b_1 + b_2 \leq \gamma + a_1$, and by symmetry, $b_1 + b_2 \leq \gamma + \min\{a_1, a_2\}$. This shows that $F_1, F_2$ and $h$ satisfy assumption \eqref{eq:bddswap_h_r} if and only if \eqref{eq:example_cond} holds.

It remains to show that \eqref{eq:example_cond} implies local Lipschitz-continuity of all $\theta_{j,i}$. It holds
\begin{align*}
\theta_{j,i}(\uu) = \frac{\partial_{r_i}\partial_{r_j}h(\rr)}{F_i''(r_i)} \ \text{ with } \rr = (r_1, r_2) = \left(u_1^\frac{1}{a_1 - 1}, u_2^\frac{1}{a_2 - 1}\right).
\end{align*}
We analyze under which condition a function of the form $\uu \mapsto Q_{b_1', b_2', \gamma'}(u_1^{p_1}, u_2^{p_2})$ with positive exponents $p_1, p_2$ is locally Lipschitz in $\Rnn^2$. It holds by \eqref{eq:qdiff_rep} and \eqref{eq:q_power}:
\begin{align*}
\partial_{u_1} Q_{b_1', b_2', \gamma'}(u_1^{p_1}, u_2^{p_2}) &= p_1 u_1^{p_1 - 1} \left(b_1' Q_{b_1' - 1, b_2', \gamma'}(u_1^{p_1}, u_2^{p_2}) - \gamma' Q_{b_1', b_2', \gamma' + 1}(u_1^{p_1}, u_2^{p_2})\right) \\
&\in V\left(Q_{b_1' - 1 + \frac{p_1 - 1}{p_1}, b_2', \gamma'}(u_1^{p_1}, u_2^{p_2}), Q_{b_1' + \frac{p_1 - 1}{p_1}, b_2', \gamma' + 1}(u_1^{p_1}, u_2^{p_2})\right),
\end{align*}
which by \eqref{eq:qbound_cond} is locally bounded if and only if $b_1' + \frac{p_1 - 1}{p_1} \geq 1$ and $b_2' \geq 0$. By applying the same argument for $\partial_{u_2} Q_{b_1', b_2', \gamma'}(u_1^{p_1}, u_2^{p_2})$, if follows that $\uu \mapsto Q_{b_1', b_2', \gamma'}(u_1^{p_1}, u_2^{p_2})$ is locally Lipschitz in $\Rnn^2$ if and only if $b_j' - \frac{1}{p_j}\geq 0$ for $j = 1,2$. In our case, we have $\frac{1}{p_j} = a_j - 1$. Hence, local Lipschitz-continuity of $\theta_{j,i}$ holds if and only if $b_j' - a_j + 1\geq 0$ for every $b_j'$ appearing in the first and fourth line of \eqref{eq:theta_rep_q}. This however follows from the inequality $b_j \geq 2a_j - 1$, thus concluding the proof.
\end{proof}

\bibliography{main}
\bibliographystyle{abbrv}

\end{document}